\documentclass[reqno,11pt]{amsart}
\usepackage{a4wide,color,eucal,enumerate,mathrsfs}
\usepackage[normalem]{ulem}
\usepackage{amsmath,amssymb,epsfig,amsthm} 
\usepackage[latin1]{inputenc}
\usepackage{psfrag}
\usepackage{mathtools}
\mathtoolsset{showonlyrefs}
\numberwithin{equation}{section}

\newtheorem{theorem}{Theorem}[section]

\newtheorem{corollary}[theorem]{Corollary}
\newtheorem{lemma}[theorem]{Lemma}

\newtheorem{definition}[theorem]{Definition}

\definecolor{gray}{RGB}{128, 128, 128}

\newtheorem*{namedtheorem}{\theoremname}
\newcommand{\theoremname}{testing}

\theoremstyle{remark}
\newtheorem{remark}[theorem]{Remark}

\DeclareMathOperator{\diam}{diam}
\DeclareMathOperator{\dist}{dist}

\newcommand{\N}{\mathbb{N}}
\renewcommand{\S}{\mathbb{S}}
\newcommand{\D}{\mathbb{D}}
\newcommand{\R}{\mathbb{R}}

\newcommand{\rf}{\mathcal R}

\DeclareMathOperator{\elle}{le}

\renewcommand{\d}{{\mathrm d}}

\def\az{\alpha}

\def\dist{{\mathop\mathrm{\,dist\,}}}
\def\loc{{\mathop\mathrm{\,loc\,}}}

\def\ez{\epsilon}

\def\bz{\beta}

\def\gz{{\gamma}}

\def\boz{{\Omega}}

\def\wz{\widetilde}

\def\ls{\lesssim}
\def\gs{\gtrsim}

\def\bint{{\ifinner\rlap{\bf\kern.35em--}
\int\else\rlap{\bf\kern.45em--}\int\fi}\ignorespaces}

\def\bbint{{\ifinner\rlap{\bf\kern.35em--}
\hspace{0.078cm}\int\else\rlap{\bf\kern.45em--}\int\fi}\ignorespaces}

\def\diam{{\mathop\mathrm{\,diam\,}}}

\def\bint{{\ifinner\rlap{\bf\kern.35em--}
\int\else\rlap{\bf\kern.45em--}\int\fi}\ignorespaces}

\definecolor{purple}{rgb}{0.75, 0, 0.9}

\begin{document}

\title[A geometric characterization of planar Sobolev extension domains]
{A geometric characterization of\\ planar Sobolev extension domains}

\author{Pekka Koskela}
\author{Tapio Rajala}
\author{Yi Ru-Ya Zhang}

\address{Department of Mathematics and Statistics \\
         P.O. Box 35 (MaD) \\
         FI-40014 University of Jyv\"as\-kyl\"a \\
         Finland}
\email{pekka.j.koskela@jyu.fi} 
\email{tapio.m.rajala@jyu.fi}
         
\address{Academy of Mathematics and Systems Science\\
The Chinese Academy of Sciences\\ Beijing 100190\\ China}

\email{yzhang@amss.ac.cn}

\thanks{The first two authors partially supported by the Academy of Finland. The third author is funded by NSFC grant No. 12288201,  the Chinese Academy of Science, and CAS Project for Young Scientists in Basic Research, Grant No. YSBR-031.}
\date{\today}


\begin{abstract}
We characterize bounded simply connected planar $W^{1,p}$-extension domains for $1 < p <2$ as 
those bounded simply connected domains
$\Omega \subset \R^2$
for which any two points $z_1,z_2 \in \R^2 \setminus \Omega$ can be connected 
with a curve 
$\gamma\subset \R^2 \setminus \Omega$  satisfying
\[
 \int_\gamma \dist(z,\partial \Omega)^{1-p}\,\d s(z)\leq  C(\Omega,p) |z_1-z_2|^{2-p}.
\]
By combining with earlier results, we obtain the following duality result:
a Jordan domain $\Omega \subset \R^2$ is a $W^{1,p}$-extension domain, $1 < p < \infty$, if and only 
if the
complementary domain $\R^2 \setminus \overline\Omega$ is a $W^{1,p/(p-1)}$-extension domain.
\end{abstract}


\maketitle


\tableofcontents


\section{Introduction}

In this paper we study those planar domains $\Omega \subset \R^2$ for
which there exists an extension operator $E \colon W^{1,p}(\Omega) \to 
W^{1,p}(\R^2)$.
Here the Sobolev space $W^{1,p}$, $1 \le p \le \infty,$ is 
\[
 W^{1,p}(\Omega) = \left\{u \in L^p(\Omega) ~:~\nabla u \in 
L^p(\Omega,\R^2)\right\},
\]
where $\nabla u$ denotes the distributional gradient of $u$. The usual norm in
$W^{1,p}(\Omega)$ is 
$\|u\|_{W^{1,p}(\Omega)} = \|u\|_{L^p(\Omega)} + \|\nabla u\|_{L^p(\Omega)}$.
More precisely,  $E \colon W^{1,p}(\Omega) \to W^{1,p}(\R^2)$ is an extension 
operator
if there exists a constant $C \ge 1$ so that for every 
$u \in W^{1,p}(\Omega)$ we have 
\[
\|Eu\|_{W^{1,p}(\R^2)} \le C\|u\|_{W^{1,p}(\Omega)}
\]
and $Eu|_\Omega = u$. Notice that we are not assuming the operator $E$ to be 
linear.
However, for $p>1,$ there also always 
exists 
a \emph{linear} extension operator provided that there exists an extension operator, 
see
\cite{hakotu2008} and also \cite{sh2006}. Finally, a domain 
$\Omega \subset \R^2$
is called a $W^{1,p}$-extension domain if there exists an extension operator
$E \colon W^{1,p}(\Omega) \to W^{1,p}(\R^2)$. For example, each Lipschitz
domain is a $W^{1,p}$-extension domain for each $1\le p\le \infty$ by the
results of Calder\'on \cite{cal1961} and Stein \cite{stein1970}. However, as proven
by Jones \cite{jo1981}, the class of extension domains is much larger. The boundary
of a $W^{1,p}$-extension domain can be of full Hausdorff dimension and it can include
fractal parts.

In this paper we prefer to use the homogeneous seminorm 
$\|u\|_{L^{1,p}(\Omega)} = \|\nabla u\|_{L^p(\Omega)}$.
This makes no difference because we only consider domains  
$\Omega$ with bounded (and hence compact) boundary; for such domains one has 
a bounded (linear)
extension operator for the 
homogeneous seminorms if and only if there is one
for the non-homogeneous ones; see \cite{heko91}.
In what follows, the norm of the extension operator is usually  with respect to the homogeneous seminorms.

The main result of our paper is the following  geometric characterization of 
simply connected bounded
planar $W^{1,p}$-extension domains. 

\begin{theorem}\label{thm:main}
 Let $1 < p < 2$ and let $\Omega \subset \R^2$ be a bounded simply connected 
domain.
 Then $\Omega$ is a $W^{1,p}$-extension domain if and only if for all
 $z_1,z_2 \in \R^2 \setminus \Omega$ there exists a curve $\gamma \subset \R^2 \setminus \Omega$
 joining $z_1$ and $z_2$ such that
 \begin{equation}\label{eq:extcharcompl}
  \int_{\gamma}\dist(z,\partial \Omega)^{1-p}\,\d s(z) 
 \le  C(\Omega,p)|z_1-z_2|^{2-p}.
\end{equation}
\end{theorem}

Both the necessity and sufficiency in Theorem \ref{thm:main} are new. Notice that the curve 
$\gamma$
above is allowed to touch the boundary of $\Omega$ even if the points in question lie outside
the closure of $\Omega.$ This is crucial: there exist bounded simply connected $W^{1,p}$-extension 
domains for which $\R^2 \setminus \overline{\Omega}$ has multiple components; see e.g. \cite{kos1990}, 
\cite{deheu2014}. 

When combined with
earlier results, Theorem \ref{thm:main} essentially completes the search for a geometric 
characterization
of bounded simply connected planar $W^{1,p}$-extension domains. The unbounded case requires
extra technical work and it will be discussed elsewhere.

The condition \eqref{eq:extcharcompl} on the complement in Theorem 
\ref{thm:main} appears
also in the characterization of $W^{1,q}$-extension domains when 
$2 < q < \infty$.
For such domains a characterization using condition 
\eqref{eq:extcharcompl} in the domain itself
with the H\"older dual exponent $q/(q-1)$ of $q$ was 
proved in \cite[Theorem 1.2]{sh2010}; see also earlier partial results 
in \cite{buko1996, ko1998}.

\begin{theorem}[Shvartsman]\label{thm_sh}
Let $2 < q < \infty$ and let $\Omega$ be a bounded simply connected  
planar domain. Then $\Omega$ is a $W^{1,q}$ -extension domain if and 
only if for all $z_1, z_2 \in \Omega$ there exists 
a rectifiable curve $\gamma \subset \Omega$ joining $z_1$ and $z_2$ such that
\begin{equation}\label{eq:extchar}
  \int_{\gamma}\dist(z,\partial \Omega)^\frac{1}{1-q}\,\d s(z) 
\le C(\Omega,q) |z_1-z_2|^\frac{q-2}{q-1}.
\end{equation}
\end{theorem}

The above two theorems leave out the case $p=2.$ This is settled by 
earlier results \cite{golavo1979,gore1990,govo1981,jo1981}, according to
which a bounded  simply connected domain is a $W^{1,2}$-extension domain
if and only if it is a quasidisk (equivalently, a uniform Jordan domain).
Thus, $\Omega$ is a bounded simply connected $W^{1,2}$-extension domain if and only if $\Omega$ is a uniform (bounded) Jordan domain which in turn is true if and only if the $\Omega$ is a Jordan domain and $\mathbb R^2\setminus \overline{\Omega}$ is uniform or, equivalently, if and only if $\Omega$ is bounded and simply connected with
$\mathbb R^2 \setminus \overline{\Omega}$ a $W^{1,2}$-extension domain.

By combining (the proof of) our characterization in Theorem \ref{thm:main} with 
Shvartsman's 
characterization stated in Theorem \ref{thm_sh},
we verify the following 
duality result between the extendability of 
Sobolev functions from a Jordan domain and from its complementary domain in Subsection 
\ref{sufffinal}.

\begin{corollary}\label{cor:dual}
 Let $1 < p,q < \infty$ be H\"older dual exponents and let 
$\Omega \subset \R^2$ be a Jordan
 domain. Then $\Omega$ is a $W^{1,p}$-extension domain if and only if 
 $\R^2 \setminus \bar\Omega$ is a $W^{1,q}$-extension domain.
\end{corollary}

Corollary \ref{cor:dual} was hinted by the example in \cite{koyazh2010} 
(see also \cite{ma1981, ro1993}) that exhibits such duality. 

\begin{corollary}\label{cor:1}
Let $\Omega \subset \R^2$ be a bounded, simply connected  
$W^{1,p}$-extension domain, where $1<p\le2.$ Then there is $q>p$ so that
$\Omega$ is a 
$W^{1,s}$-extension domain for all $1<s< q.$
\end{corollary}

The case $1<p<2$ follows from Theorem \ref{thm:main} together with the fact 
that \eqref{eq:extcharcompl} implies 
the analogous inequality
for all $1 < s < p+\epsilon$. The case of smaller $s$ is essentially just 
H\"older's 
inequality, see \cite{la1985}, while the improvement to larger exponents follows
from Lemma~\ref{selfimprove} that relies on ideas in the proof of Proposition 2.6 in \cite{sh2010}.
Again, the case $p=2$ of Corollary \ref{cor:1} was already known to hold: one 
then has extendability
for all $1<s<\infty.$  

By combining Corollary \ref{cor:1} with results from \cite{ko1998} and 
\cite{sh2010} we obtain an
open-ended property.

\begin{corollary}\label{cor:2}
Let $\Omega \subset \R^2$ be a bounded, simply connected  
$W^{1,p}$-extension domain, where $1<p<\infty.$ Then the set of all $1<s<\infty$ for which  
$\Omega$ is a 
$W^{1,s}$-extension domain is an open interval.
\end{corollary}

Actually, the open interval above can only be one of $1<s<\infty,$ $1<s<q$ with $q\le 2,$ or
$q<s<\infty$ with $q\ge 2.$

Let us finally comment on some earlier partial results related to Theorem \ref{thm:main}.
First of all, bounded simply connected 
$W^{1,p}$-extension domains are 
John domains when $1 \le p < 2$; see e.g.\ \cite[Theorem 6.4] {kos1990}, \cite[Theorem 3.4]{gore1990}, \cite[Theorem 4.5]{nava1991} and 
references therein. 
The definition of a John domain is given in Definition \ref{def:John} below. However, there exist 
John domains
that fail to be extension domains and, even after Theorem \ref{thm:main} there is no interior 
geometric
characterization available for this range of exponents. Secondly, in \cite{kosmirsha10} it was 
shown that the complement of a bounded
simply connected $W^{1,1}$-extension domain is quasiconvex.
This was obtained as a corollary to a characterization of 
bounded simply connected $BV$-extension domains.
Recall that a set $E \subset \R^2$
is called {\it{quasiconvex}} if there exists a constant $C\ge1$ such that any pair of 
points $z_1,z_2 \in E$ can be connected to each other with a rectifiable curve 
$\gamma \subset E$
whose length satisfies $\elle(\gamma) \le C|z_1-z_2|$. 
In \cite{kosmirsha10} it was conjectured that quasiconvexity of the complement
holds for every bounded simply connected $W^{1,p}$-extension planar domain when $1 < p < 2$.
This conjecture follows from  Theorem \ref{thm:main} (see Lemma~\ref{kvasikonveksi}), but again, quasiconvexity is a weaker
condition than our geometric characterization.

 \subsection{Idea of the proof}
 
We show the necessity of
\eqref{eq:extcharcompl} in Section \ref{sec:nec} by first verifying this condition
under the additional requirement that the domain in question is a Jordan domain. This additional assumption
together with the extension property allows us to construct suitable
test functions that are employed to verify \eqref{eq:extcharcompl}; see Lemma~\ref{lma:testfunction}. 
These are motivated by the function $u(x,y) = \frac{y}{x}$ on $\Omega = \{(x,y)\,:\, 0 < y < x, 0 < x < 1\}$.
In the complex notation we have 
$|\nabla u(z)|\le C |z|^{-1};$ a variant of this property, see \eqref{eq:phiestimate}, holds for the function $\phi$ constructed in the proof of Lemma~\ref{lma:testfunction}. 
The curve $\gamma$ in \eqref{eq:extcharcompl} in the case of a Jordan domain is given as the image under the exterior Riemann mapping function of a uniform curve (in the sense of Definition~\ref{inneruniform}) that we construct by hand in the exterior of the unit disk; see Figure~\ref{fig:gamma}. For readers who are familiar with conformal geometry,
it might be helpful to think of the constructed curve in the exterior of the unit disk as
almost a quasihyperbolic geodesic for which many useful geometric properties are preserved under conformal maps. 

The general
case is then handled via an approximation argument, for which we fill $\Omega$ by an increasing sequence
of Jordan $W^{1,p}$-extension domains with control on the norms of the respective extension operators; see Theorem~\ref{inner extension}.
These domains $\Omega_n$ are the images of $B(0,1-\frac1n)$, $n = 1,2,\dots$, under the Riemann mapping function from the unit disk onto the domain $\Omega$. To prove the uniform $W^{1,p}$-extension property we employ a variant of the technique used by Jones in \cite{jo1981} to construct extensions from $\Omega_n$ to $\Omega$.
The main differences with the setting in \cite{jo1981} are that we only extend to an annular region and that the uniformity of the domain considered by Jones is with respect to the Euclidean metric, 
while  in our case it is with respect to the inner metric of the domain.

For sufficiency, we again first deal with
the Jordan case, and then use a compactness argument to pass to a limit. This is done in
Section \ref{sec:suf}. The crucial point in the proof is the introduction of a 
new version of the Whitney extension technique in the case of Jordan domains. The extension operator is defined in Section~\ref{jordandef}. In order to build this extension, we assign a Whitney square of $\Omega$ to each
complementary Whitney square of size at most the size of $\Omega.$ 
In \cite{jo1981}, Jones chooses a square of comparable diameter, and the uniformity gives that each Whitney square of $\Omega$ gets assigned to at most uniformly finitely many Whitney squares of the exterior. Roughly speaking, this gives a bi-Lipschitz correspondence between Whitney squares. 
In our case this kind of a correspondence cannot be expected. 
To overcome this problem, we pick a  Whitney square whose \emph{shadow}
along hyperbolic rays has diameter comparable to that of the shadow of the complementary square. In a sense we reflect with respect to harmonic measure. Hence the diameter of a ``reflected" square can be much larger than the
diameter of the original one and we cannot uniformly bound the number of exterior squares  that correspond to a single Whitney square of $\Omega$. Nevertheless, \eqref{eq:extcharcompl}  allows us to eventually 
establish appropriate bounds on our extension. Towards this, roughly speaking, we establish estimates on the many-to-one map from the collection of the exterior Whitney squares to Whitney squares of $\Omega$ ; see Lemma~\ref{strong sum estimate}. This is obtained through a delicate
analysis of the locations and sizes of those complementary squares that share a ``reflected" square.

The reader familiar with Sobolev extensions may wonder why do we not simply employ the existing extension operators such as those in \cite{sh2006} and \cite{hakotu2}. This is because we have not been able to directly show that these operators work under our assumptions.
However, once we know by our main theorem that the domains are extension domains, we conclude that also these extension operators work under our assumptions.

\medskip

Section \ref{sec:preli} introduces notation and initial results. Theorem \ref{thm:main} gets proven in Sections \ref{sec:nec} and \ref{sec:suf}.
Finally, Corollary \ref{cor:dual} is proven at the very end of this paper in Section \ref{sec:cor}.


\section{Preliminaries}\label{sec:preli}

Let us  fix some notation. When we make estimates, we often write the 
constants as 
positive real numbers $C(\cdot)$ with
the parentheses including all the parameters which the constant depends on; 
we just simply write $C$ if it is absolute. The constant $C(\cdot)$ may
vary between appearances, even within a chain of inequalities. 
By $a\ls b$ we mean that $a \le Cb$ for some constant $C \ge 2$.
Then $a \sim b$ means that both $a\ls b$ and $b\ls a$ hold.  
If we need to stress the dependence of the respective constant $C$ 
only on data $A$, we write $a\ls_{A}b$, $a {\sim}_{A} b,$
respectively.
The Euclidean distance between two sets $A,\,B \subset \R^2$ is denoted 
by $\dist(A,\,B)$. 
By $\D$ we always mean the open unit disk in $\R^2$ and by $\S^1$ its boundary. 
The interior of a set $A$ is denoted by $A^\circ$ and the closure by 
$\overline A.$ Given a measurable set $A$ of strictly positive area $|A|$
and a function $u\in L^1(A),$ we write
$$u_A=\bint_A u =\frac 1 {|A|}\int_A u\,dz.$$

\subsection{Curves and integrals over curves}\label{subsection:curves}

Let us next define the curves and line integrals that we use throughout this paper.
A continuous map $\gamma \colon I \to \mathbb R^2$ is called a curve when $I$ is a (possibly unbounded) interval. When there is no danger for confusion, we sometimes refer also to the image $\gamma(I) \subset \mathbb R^2$ by $\gamma.$ Recall that the derivative $\gamma'(t)$ exists for almost every $t \in I$ for a locally Lipschitz $\gamma.$ Then the Euclidean length of such a curve can be defined by
\[
 \elle(\gamma) = \int_I|\gamma'(t)|\,\d t.
\]
In general, 
\[
\elle(\gamma)=\sup\{\sum_{j=1}^k|\gamma(t_{j+1})-\gamma(t_j)|\},
\]
where the supremum runs over all $k\ge 1$ and all $t_1<t_2<\dots <t_{k+1}\in I.$ 
If $\elle(\gamma) < \infty$, which is the case for the Lipschitz curves defined on compact intervals, we call the curve $\gamma$ rectifiable. In this case, after a reparametrization, we may assume that $\gamma \colon [0,1] \to \mathbb R^2$ and that $|\gamma'(t)|=\elle(\gamma)$ almost everywhere. We call such a parametrization a constant speed parametrization. Alternatively, we may parametrize the curve $\gamma$ by arc-length. This way we obtain $\tilde \gamma \colon [0,\elle(\gamma)] \to \mathbb R^2$ with $|\tilde\gamma'(t)|= 1$ almost everywhere. 
A change of variable argument argument shows that  reparametrization does not change the length of the curve. From now on, when using the term \emph{rectifiable curve}, by default, we will assume the constant speed parametrization on $[0,1]$, unless otherwise stated.

Notice that we are not requiring our curves to be injective. However, if $E \subset \mathbb R^n$ is a continuum with $\mathcal H^1(E) < \infty$, then for any $x,y \in E$ there exists an injective curve $\gamma_{x,y} \colon [0,1] \to E$ with $\gamma(0) = x$, $\gamma(1) = y$ and $\elle(\gamma_{x,y}) \le \mathcal H^1(E)$, see \cite[Lemma 3.12]{Falconer198}.

The line integral of a measurable function $f \colon \mathbb R^2 \to \overline{\mathbb R}$ along a rectifiable curve $\gamma$ is defined as
\[
 \int_\gamma f(z)\,\d s(z) = \int_0^1 f(\gamma(t))|\gamma'(t)|\,\d t
 = \int_0^1 f(\gamma(t))\elle(\gamma)\,\d t,
\]
whenever the integral on the right-hand side exists.
Alternatively, for the arc-length parametrization $\tilde \gamma$ for $\gamma$ we have
\[
 \int_\gamma f(z)\,\d s(z) = \int_0^{\elle(\gamma)}f(\tilde \gamma(t))\,\d t.
\]


We equip $[0,\infty]$ with the topology whose basis consists of restrictions of open sets of $\mathbb R$ to $[0,\infty)$ together with all the intervals $(M,\infty]$ with $M>0$.
A function $f\colon \mathbb R^2 \to [0,\infty]$ is then continuous at $x_0 \in \mathbb R^2$ where $f(x_0) = \infty$, if for every $M>0$ there exists $\varepsilon >0$ for which $f(x) > M$ for all 
$x \in B(x_0,\varepsilon)$. The continuity of  $f$  at each $x_0,$ where $f(x_0)<\infty$ has the usual meaning.


By the Arzel\'a-Ascoli lemma, we have the following result.
\begin{lemma}\label{arsela}
 Let $\gamma_i \colon [0,1] \to \mathbb R^2$, $i \in \mathbb N$, be a collection of rectifiable curves so that $\bigcup_i \gamma_i([0,1])$ is bounded and $\sup_i\elle(\gamma_i) < \infty$. Then there exists a sequence $i_j \nearrow \infty$ and a rectifiable curve $\gamma_\infty$ so that $\gamma_{i_j} (t)\to \gamma_\infty(t)$ for all $t\in [0,1]$ when $j \to \infty$. Moreover, for any continuous function $f \colon \mathbb R^2 \to [0,\infty]$ we have
 \[
  \int_{\gamma_\infty}f(z)\,\d s(z) \le \liminf_{j \to \infty}\int_{\gamma_{i_j}}f(z)\,\d s(z).
 \]
 In particular, 
 \[
  \elle(\gamma_\infty) \le \liminf_{j \to \infty}\elle(\gamma_{i_j}).
 \]
\end{lemma}
%

If $\gamma_1,\gamma_2 \colon [0,1] \to \mathbb{R}^2$ are two non-constant rectifiable curves with $\gamma_1(1) = \gamma_2(0)$, we define their concatenation $\gamma_1\ast\gamma_2$ by
\[
 \gamma_1\ast\gamma_2(t) =
 \begin{cases}
  \gamma_1\left(t\frac{\elle(\gamma_1) + \elle(\gamma_2)}{\elle(\gamma_1)}\right),& \text{if }0 \le t \le \frac{\elle(\gamma_1)}{\elle(\gamma_1) + \elle(\gamma_2)}\\
  
  \gamma_2\left(t\frac{\elle(\gamma_1) + \elle(\gamma_2) }{\elle(\gamma_2)}- \frac{\elle(\gamma_1)}{\elle(\gamma_2)}\right), & \text{if } \frac{\elle(\gamma_1)}{\elle(\gamma_1) + \elle(\gamma_2)} < t \le 1 
 \end{cases}.
\]
For a curve $\gamma$ we define the reversed curve by $\overleftarrow{\gamma}$ by setting  $\overleftarrow{\gamma}(t) = \gamma(1-t)$.
Because of possible noninjectivity, a restriction
of a curve to a subcurve between $x,y \in \gamma$ can be defined in many ways. We will use the first times in the parameter space where we hit $x$ and $y$, i.e. the restriction $\gamma[x,y]$ is defined as
\[
 \gamma[x,y](t) = \gamma((1-t)t_x+tt_y),
\]
where $t_x = \inf\{t \in [0,1]\,:\,\gamma(t) = x\}$ and
$t_y = \inf\{t \in [0,1]\,:\,\gamma(t) = y\}$. Notice that with this definition we have $\gamma[x,y](0) = x$ and $\gamma[x,y](1) = y$, and the subcurve might go in the reversed direction along $\gamma$.
For $x,y \in \mathbb R^2$, we also use the notation $[x,y] \colon [0,1] \to \mathbb R^2 \colon t \mapsto (1-t)x + ty$ to denote the line segment from $x$ to $y$.


Let $\Omega \subset \mathbb R^2$ be a domain and $x,y \in \overline{\Omega}$. We say that a curve $\gamma \colon [0,1] \to \mathbb R^2$ joins $x$ and $y$ in $\Omega$, if $\gamma(0) = x$, $\gamma(1) = y$, and $\gamma([0,\,1])\subset\Omega\cup\{x,y\}$.
We then define the {\it{inner distance with respect to $\Omega$}} between $x,\,y\in\overline{\Omega}$ by setting
$$\dist_{\Omega}(x,\,y)=\inf_{\gamma\subset \Omega} \elle(\gamma),$$
where the infimum runs over all the curves joining $x$ and $y$ in $\Omega.$
We will postpone the proof of the fact that $\dist_\Omega$ is a distance on Jordan domains to Lemma~\ref{inner triangle}.
Notice that any $x,y \in \Omega$ are rectifiably joinable in $\Omega$, but the inner distance from $x \in \Omega$ to a point $y \in \partial\Omega$ might well be infinite.
If $\dist_{\Omega}(x,\,y)< \infty$, we say that \emph{$x$ and $y$ are rectifiably joinable in $\Omega$}.
The inner diameter $\diam_{\Omega}(E)$ of a set $E\subset \overline{\Omega}$ is 
then defined to be the supremum of $\dist_{\Omega}(x,\,y)$ over pairs of points
$x,y\in E$ and $B_\Omega(z,r) = \{y \in \Omega \,|\,\dist_{\Omega}(z,\,y)<r\}$ denotes the open ball in $\Omega$ with respect to the inner distance.

\subsection{Curve condition} 

We begin by recording a consequence of \eqref{eq:extcharcompl} that essentially follows 
from \cite[Lemma 2.1]{sh2010}, 
also see the proof of \cite[Theorem 2.15]{GM1985} and \cite{la1985}. Since the results of 
\cite[Lemma 2.1]{sh2010} are stated for curves contained in open sets, we check below that the 
arguments work in our setting. 

\begin{lemma}\label{kvasikonveksi}
Let $1<p<2$ and let $\Omega\subset \R^2$ be a bounded simply connected domain and let $z_1,\,z_2\in \mathbb R^2\setminus \Omega$. 

\noindent{\rm (1)} If
\begin{equation}\label{ineq 5}
\max\{\dist(z_1,\,\partial \Omega),\,\dist(z_2,\,\partial \Omega)\}\le 2|z_1-z_2|,
 \end{equation}
and if $\gamma\subset \mathbb R^2\setminus \Omega$ is a curve joining $z_1,\,z_2$ so that
$$
\int_{\gamma}\dist(z,\partial \Omega)^{1-p}\,\d s(z) 
 \le  C_1|z_1-z_2|^{2-p},
$$
then we have 
$$\elle(\gamma)\le C(p,\,C_1)|z_1-z_2|.$$

\noindent{\rm (2)} If
\begin{equation}\label{ineq 3}
\max\{\dist(z_1,\,\partial \Omega),\,\dist(z_2,\,\partial \Omega)\}> 2|z_1-z_2|,
\end{equation}
then the line segment $[z_1,\,z_2]\subset \mathbb R^2\setminus \Omega$ joining $z_1,\,z_2$ satisfies
$$\int_{[z_1,\,z_2]}\dist(z,\,\partial\Omega) ^{1-p}\,\d s(z) 
 \le C(p)  |z_1-z_2|^{2-p}.$$

Especially, if the curve condition \eqref{eq:extcharcompl} holds, then  $ \mathbb R^2\setminus \Omega$ is quasiconvex with a constant depending only on $p$ and $C_1.$ 
\end{lemma}

\begin{proof}
Let us first verify part (1).  
We claim that 
\begin{equation}\label{ineq 9}
\gamma\subset B(z_1,\,c|z_1-z_2|) \setminus \Omega
\end{equation} with $c=((2-p)(C_1+1)+3^{2-p})^{1/(2-p)}-2$. If \eqref{ineq 9} holds, then for any $z\in \gamma$, according to \eqref{ineq 5} we have
$$\dist(z,\,\partial\Omega)\le \dist(z_1,\,\partial \Omega)+c|z_1-z|\le (2+c) |z_1-z_2|, $$
and by $1<p<2$
 $$ (2+c)^{1-p}  |z_1-z_2|^{1-p} \elle(\gamma)\le \int_{\gamma}\dist(z,\partial \Omega)^{1-p}\,\d s(z)  \le C_1|z_1-z_2|^{2-p}.$$
Hence
$$\elle(\gamma)\le C(p,\,C_1)|z_1-z_2|,$$ and we conclude that we only need to
establish \eqref{ineq 9}.

Let us verify \eqref{ineq 9}. By the curve condition on $\gamma$, the triangle inequality, and \eqref{ineq 5}
\begin{equation}\label{eeka}
\begin{split}
C_1|z_1-z_2|^{2-p}& \ge \int_{\gamma}\dist(z,\partial \Omega)^{1-p}\,\d s(z)\\
& \ge \int_{\gamma}( \dist(z_1,\,\partial \Omega) + |z-z_1|)^{1-p}\,\d s(z)\\
& \ge  \int_{\gamma}( 2|z_1-z_2|+ |z-z_1|)^{1-p}\,\d s(z).
\end{split}
\end{equation}
Suppose that $\gamma\subset \mathbb R^2\setminus \Omega$ is not contained 
in $B(z_1,\,c|z_1-z_2|)\setminus \Omega$. Then 
by restricting the curve to the part contained in the disk $B(z_1,\,c|z_1-z_2|)$,
we further have
\begin{equation}\label{tooka}
\begin{split}
\int_{\gamma}( 2|z_1-z_2|+ |z-z_1|)^{1-p}\,\d s(z) & \ge \int_{|z_1-z_2|}^{c|z_1-z_2|} (2|z_1-z_2|+ t)^{1-p}\,dt \\
&=\frac {|z_1-z_2|^{2-p}}{2-p}((c+2)^{2-p}-3^{2-p}).
\end{split}
\end{equation}
By combining \eqref{eeka} and \eqref{tooka}, we arrive at
$$\frac 1 {2-p}((c+2)^{2-p}-3^{2-p})\le C_1,$$
which is impossible for our choice of $c$. Thus we conclude \eqref{ineq 9}, and we have proven part (1) of our claim. 

Towards (2), clearly  \eqref{ineq 3} implies $[z_1,\,z_2]\subset \mathbb R^2\setminus \Omega$. With possibly changing the roles of $z_1$ and $z_2$, we may assume that
$$\dist(z_1,\,\partial\Omega)>2|z_1-z_2|.$$
Thus we have 
\begin{align*}
   \int_{[z_1,\,z_2]} \dist(z,\,\partial \Omega)^{1-p} \, ds(z) 
\le  |z_1-z_2| 2^{p-1} \dist(z_1,\,\partial \Omega)^{1-p}\le C(p) |z_1-z_2|^{2-p},
\end{align*}
where we used the facts that $1<p<2$, and that by \eqref{ineq 3} together with the triangle inequality we have that, for each $z\in [z_1,\,z_2]$,
$$\dist(z,\,\partial \Omega)\ge \dist(z_1,\,\partial \Omega)-|z_1-z|\ge \dist(z_1,\,\partial \Omega)-|z_1-z_2|\ge \frac 1 2 \dist(z_1,\,\partial \Omega).$$
This gives the claim of the second part. 
\end{proof}

We establish the following self-improving property of 
\eqref{eq:extcharcompl}   via ideas from 
the proof of \cite[Proposition 2.6]{sh2010}. 

\begin{lemma}\label{selfimprove}
Let $\Omega \subset  \R^2$ be a bounded simply connected domain for which  
\eqref{eq:extcharcompl} holds for $\R^2\setminus \Omega.$ Then there exists $\epsilon>0$ that only depends
on $p$ and the constant $C_1=C(\Omega,p)$ in \eqref{eq:extcharcompl} so that,
for every $1 < \hat p < p+\epsilon$ and all  
 $z_1,z_2 \in \R^2 \setminus \Omega$ there exists a curve 
$\gamma \subset \R^2 \setminus \Omega$
 joining $z_1$ and $z_2$ such that
 \begin{equation*}
  \int_{\gamma}\dist(z,\partial \Omega)^{1-\hat p}\,\d s(z) 
 \le  C(p,\,C_1)|z_1-z_2|^{2-\hat p}.
\end{equation*}
\end{lemma}

\begin{proof}
We  begin by showing that, under the assumption of the lemma, for any pair of points
 $z_1,z_2 \in \R^2 \setminus \Omega$, there exists a rectifiable curve 
$\gamma\subset  \R^2 \setminus \Omega$ joining them  with $\elle(\gamma)\le C|z_1-z_2|$
such that, for all $w_1,\,w_2\in \gamma$,  any subcurve $\gamma[w_1,\,w_2]\subset \gamma$ joining $w_1$ and $w_2$ satisfies
 \begin{align}\label{ineq 4}
\int_{\gamma[w_1,\,w_2]} \dist(z,\,\partial \Omega)^{1-p} \, ds(z) \le c |w_1-w_2|^{2-p},
\end{align}
where the constants $C,c$ depend only on $p$ and $C_1$. 
In the case where $z_1,\,z_2$ satisfy \eqref{ineq 3}, we claim that we may
take $\gamma=[z_1,z_2],$ the line segment joining $z_1$ to $z_2.$ Towards this,
we may clearly assume that 
$$\dist(z_1,\,\partial\Omega)>2|z_1-z_2|.$$
Then since every subcurve of our line segment $\gamma$ is still a line segment,  we have 
\begin{align*}
   \int_{[w_1,\,w_2]} \dist(z,\,\partial \Omega)^{1-p} \, ds(z) 
&\le C(p)  |w_1-w_2| \dist(z_1,\,\partial \Omega)^{1-p}\\
&\le C(p) |w_1-w_2||z_1-z_2|^{1-p}\le C(p) |w_1-w_2|^{2-p},
\end{align*}
where we used the facts that $1<p<2$, and that by \eqref{ineq 3} with the triangle inequality we have that, for each $z\in [z_1,\,z_2]$,
$$\dist(z,\,\partial \Omega)\ge \dist(z_1,\,\partial \Omega)-|z_1-z|\ge \dist(z_1,\,\partial \Omega)-|z_1-z_2|\ge \frac 1 2 \dist(z_1,\,\partial \Omega).$$
Thus \eqref{ineq 4} holds whenever  \eqref{ineq 3} holds.

We are left with the case where \eqref{ineq 3} fails. Then \eqref{ineq 5} holds.
We claim that there exists a curve $\gamma\subset  \R^2 \setminus \Omega$ that joins $z_1,z_2$  and 
minimizes the integral in \eqref{eq:extcharcompl}.  
 
Let $\gamma_j$ be a sequence of curves joining $z_1$ and $z_2$ such that
 $$\int_{\gamma_j} \dist(z,\,\partial \Omega)^{1-p} \, ds(z) \le c_j |z_1-z_2|^{2-p},$$
 where $c_j\le C_1$  converge to the infimum $c$ of such constants $c_j$ for 
the pair $z_1$ and $z_2$. Then this condition  ensures that 
$$\elle(\gamma_j)\le C |z_1-z_2|$$
for all $j$ by part (1) of Lemma~\ref{kvasikonveksi}. 
 Therefore, by Lemma~\ref{arsela} there exists a sequence $j_i \to \infty$ and a limit curve $\gamma$ so that $\gamma_{j_i}(t)\to  \gamma(t)$ for all $t$ as $i \to \infty$ and
 \begin{equation}\label{kayra 1}
\int_{\gamma} \dist(z,\,\partial \Omega)^{1-p} \, ds(z) \le   \liminf_{i \to \infty}\int_{\gamma_{j_i}} \dist(z,\,\partial \Omega)^{1-p} \, ds(z) \le c|z_1-z_2|^{2-p}.
 \end{equation}

Now fix $z_1,\, z_2\in \mathbb R^2\setminus \Omega$ satisfying  \eqref{ineq 5}, and let  $\gamma\subset \mathbb R^2\setminus \Omega$  
be a minimizer for the integral in \eqref{eq:extcharcompl} for $z_1,\, z_2$. 
We claim that any subcurve $\gamma[w_1,\,w_2]$ of $\gamma$ is also a minimizer for $w_1$ and $w_2$. 
Otherwise, let $\gamma'[w_1,\,w_2]$ a minimizer for $w_1$ and $w_2$. Because of symmetry we may assume that $\gamma$ passes $z_1,\,w_1,\,w_2,\,z_2$ in this order. Then by the linearity of the integral we have that
\begin{align*}
 \int_{\gamma} \dist(z,\,\partial \Omega)^{1-p} \, ds(z)
 =& \left(\int_{\gamma[z_1,\,w_1]}+\int_{\gamma[w_1,\,w_2]}+\int_{\gamma[w_2,\,z_2]} \right) \dist(z,\,\partial \Omega)^{1-p} \, ds(z)\\
 > & \left(\int_{\gamma[z_1,\,w_1]}+\int_{\gamma'[w_1,\,w_2]}+\int_{\gamma[w_2,\,z_2]} \right) \dist(z,\,\partial \Omega)^{1-p} \, ds(z)\\
=& \int_{\gamma'} \dist(z,\,\partial \Omega)^{1-p} \, ds(z),
\end{align*}
where 
$$\gamma'=\gamma[z_1,\,w_1]\ast \gamma'[w_1,\,w_2]\ast \gamma[w_2,\,z_2] $$
joins $z_1$ and $z_2$.
This  contradicts the minimality assumption on $\gamma$. Thus our claim follows, and hence \eqref{ineq 4} also holds for points satisfying \eqref{ineq 5}.

To conclude,  for any pair of points
 $z_1,z_2 \in \R^2 \setminus \Omega$, there exists a rectifiable curve $\gamma\subset  \R^2 \setminus \Omega$ joining them  with $\elle(\gamma)\le C|z_1-z_2|$ 
and so that \eqref{ineq 4} holds.
In other words, the curve $\gamma$ satisfies the so-called ``strong $\az$-hyperbolicity" in \cite[Definition 2.4]{sh2010} with $\az=2-p$. Thus we can use  the proof of \cite[Proposition 2.6]{sh2010}  to conclude the lemma. For the sake of 
completeness, let us give the details of this argument. 

We first show that, whenever a curve $\gamma$ satisfies \eqref{ineq 4} and $w_1,w_2\in \gamma,$ we have 
\begin{equation}\label{ineq 11}
 \frac 1 {\elle(\gamma[w_1,{w_2}])}\int_{\gamma[w_1,w_2]} \dist(z,\,\partial \Omega)^{1-p} \, ds(z) \le C(p,\,c) \min_{z\in \gamma[w_1,w_2]} \dist(z,\,\partial \Omega)^{1-p}. 
\end{equation}
We have two cases. 
If $$\max_{z\in\gamma[w_1,w_2]} \dist(z,\,\partial \Omega)< 2\elle(\gamma[w_1,w_2]),$$
then as $1<p<2$,
$$\min_{z\in\gamma[w_1,w_2]} \dist(z,\,\partial \Omega)^{1-p}>2^{1-p}\elle(\gamma[w_1,w_2])^{1-p}.$$
Therefore
$$ \int_{\gamma[w_1,w_2]} \dist(z,\,\partial \Omega)^{1-p} \, ds(z) \le c|w_1-w_2|^{2-p}\le c \elle(\gamma[w_1,w_2]) \elle(\gamma[w_1,w_2])^{1-p}$$
$$\le C(p,\,c) \elle(\gamma[w_1,w_2])\min_{z\in\gamma[w_1,w_2]} \dist(z,\,\partial \Omega)^{1-p},$$
and \eqref{ineq 11} follows. 
If 
$$\max_{z\in\gamma[w_1,w_2]} \dist(z,\,\partial \Omega)\ge 2\elle(\gamma[w_1,w_2]),$$
then by the triangle inequality
$$\max_{z\in\gamma[w_1,w_2]} \dist(z,\,\partial \Omega)\le \min_{z\in \gamma[w_1,w_2]} \dist(z,\,\partial \Omega) + \elle(\gamma[w_1,w_2])$$
$$\le \min_{z\in \gamma[w_1,w_2]} \dist(z,\,\partial \Omega) + \frac 1 2 \max_{z\in\gamma[w_1,w_2]} \dist(z,\,\partial \Omega).$$
Thus 
$$ \min_{z\in \gamma[w_1,w_2]} \dist(z,\,\partial \Omega)\le  \max_{z\in\gamma[w_1,w_2]} \dist(z,\,\partial \Omega)\le 2 \min_{z\in \gamma[w_1,w_2]} \dist(z,\,\partial \Omega),$$
and  \eqref{ineq 11} again follows  from \eqref{ineq 4}. 

Now let us complete the proof by relying on $\elle(\gamma)\le C|z_1-z_2|,$
\eqref{ineq 4} and \eqref{ineq 11}. 
Parametrize $\gamma$ by arc length, $\gamma\colon [0,\,\elle(\gamma)] \to \mathbb R^2\setminus \Omega$. 
Then \eqref{ineq 11} gives the estimate
$$\frac 1 { |t_2-t_1|} \int_{t_1}^{t_2} \dist(\gamma(t),\,\partial \Omega)^{1-p}\, dt\le C(p,\,c) \min_{t\in [t_1,\,t_2]} \dist(\gamma(t),\,\partial \Omega)^{1-p}, $$
 for all $0\le t_1< t_2\le \elle(\gamma).$ 
This implies that $\omega(t)=\dist(\gamma(t),\,\partial \Omega)^{1-p}$ is a Muckenhoupt  $\mathcal A_1$-weight on $[0,\,\elle(\gamma)]$. By the reverse H\"older inequality (see e.g.\ \cite[15.3]{HKM1993}) there exists $\bz>1$ 
that only depends  on $C(p,c)$ such that that
$$\left(\frac 1 {\elle(\gamma)} \int_{0}^{\elle(\gamma)} \omega(t)^{\bz} \, dt \right)^{\frac 1 \bz} \le C(p,\,c) \frac 1 {\elle(\gamma)} \int_{0}^{\elle(\gamma)} \omega(t)  \, dt. $$
This estimate together with \eqref{ineq 4} and the fact that $|z_1-z_2|\le \elle(\gamma)\le 
C|z_1-z_2|$ implies the claim. 
\end{proof}

We close this subsection with the following technical existence result that will be employed in Section~\ref{sec:suf}.

\begin{lemma}\label{pisteetlahella}
Let $\frac 1 2 <R<1$ and $1<\hat p<2.$ There is an absolute constant $\delta>0$ and a constant $C(\hat p)$ that only depends on $\hat p$ so that the following holds. 
Let $w_1\in \mathbb D\setminus \overline{B}(0,R)$ and let $w_1\neq w_2\in B(z_1,\delta (1-|w_1|))\setminus \overline {B}(0,R).$ Then there is a curve
$\gamma\subset B(w_1, (1-|w_1|)/2)\setminus \overline {B}(0,R)$ joining $w_2$ to $w_1$ so that
$$\int_{\gamma}\dist (z,\partial B(0,R)\cup \partial B(w_1,(1-|w_1|)/2))^{1-\hat p}\, ds(z)\le C(\hat p)|w_1-w_2|^{2-\hat p}.$$
\end{lemma}
\begin{proof}
We will prove the claim with $\delta=1/30.$
Fix $R,w_1$ and $w_2$ as in our assumptions. If $w_2$ lies on the radial segment through $w_1,$  then we may clearly choose $\gamma$ to be the radial segment between the points
$w_2$ and $w_1.$ Otherwise, consider the additional points $\xi_j=(|w_1|+|w_2-w_1|)\frac {w_j}{|w_j|},$ $j=1,2.$ We let $\gamma_j$ be the radial segment from $w_j$ to $\xi_j$ for
$j=1,2$ and let $\gamma_3$ be the shorter arc on the circle $S(0,|w_1|+|w_2-w_1|)$ from $\xi_1$ to $\xi_2.$
We can estimate the lenghts of these curves by
\[
 \elle(\gamma_2) = |w_2-\xi_2| = |w_1|-|w_2|+|w_2-w_1| \le  2|w_1-w_1|
\]
and
\[
 \elle(\gamma_3) \le \pi |\xi_1-\xi_2| \le 2\pi|w_1-w_2|,
\]
since $w_1,w_2,\xi_1,\xi_2 \in \mathbb D \setminus \overline{B}(0,\frac12)$.
We define $\gamma$ as the concatenation $\gamma = \gamma_1 \ast \gamma_3 \ast \overleftarrow{\gamma}_2$.
Then 
\begin{equation}\label{msp}
\elle(\gamma)\le \elle(\gamma_1)+\elle(\gamma_2)+\elle(\gamma_3)\le |w_2-w_1| + 2|w_2-w_1| +2\pi |w_2-w_1|<10|w_2-w_1|.
\end{equation}
Moreover, clearly $\gamma\cap \overline {B}(0,R)=\emptyset.$ Since $\gamma_3\subset S(0,|w_1|+|w_2-w_1|)$ and $|w_1|>R,$ we have that $\dist(\gamma_3,\partial B(0,R))\ge |w_2-w_1|.$ Furthermore,
$$\dist(\gamma,\partial B(w_1,(1-|w_1))/2))\ge (1-|w_1|)/2- 10|w_2-w_1|\ge 5|w_2-w_1| $$ by \eqref{msp} since $w_1\in \gamma$ and $z_2\in B(z_1,(1-|w_1|)/30.$  
This together with \eqref{msp} yields
$$\int_{\gamma}\dist (z,\partial B(0,R)\cup \partial B(w_1,(1-|w_1|)/2))^{1-\hat p}\, ds(z)$$
$$ \le \int_{\gamma_1}\dist (z,\partial B(0,R))^{1-\hat p}\,ds(z)+\int_{\gamma_2}\dist (z,\partial B(0,R))^{1-\hat p}\,ds+20|w_2-w_1|^{1-\hat p}.$$
The claim follows by integrating since $\gamma_1,\gamma_2$ are radial segments and both are of length no more than $2|w_2-w_1|.$
\end{proof}

\subsection{Hyperbolic metric}\label{sec:hyperbolic}

Recall that the hyperbolic distance between $z_1,\,z_2\in\mathbb D$ is defined to be
$$\dist_h(z_1,\,z_2)=\inf_{\gamma}\int_{\gamma} \frac {2}{1-|z|^2}\, \d s(z),$$
where the infimum is taken over all rectifiable curves $\gamma$ joining $z_1$ to $z_2$ in $\mathbb D$. 
Notice that the density above is comparable to $\frac 1{1-|z|}=\dist(z,\,\partial \mathbb D)^{-1}.$ 
The infimum is achieved by a unique curve joining $z_1,z_2$  that we call  the (geodesic) hyperbolic segment between $z_1$ and $z_2.$ It is an arc of  a (generalized) 
circle that intersects the unit circle orthogonally. Especially, if the hyperbolic segment contains the origin, then it is a Euclidean segment. Conversely, each Euclidean
segment that contains the origin and is contained in $\mathbb D$ is a hyperbolic segment. 
It is not obvious from the  definition that the hyperbolic distance is preserved under conformal self maps of the disk, but this is indeed the case and follows from the fact that
conformal self maps of the disk are M\"obius transformations of $\mathbb D$ onto itself.

The hyperbolic distance in a simply connected domain is defined via a conformal map. Precisely, given a simply connected domain 
$\Omega$ we pick a conformal map $\varphi\colon \mathbb D\to \Omega$ and define, for $x,\,y\in \Omega$,
$$\dist_{h}(x,\,y)=\dist_h(\varphi^{-1}(x),\,\varphi^{-1}(y)) .$$
This is independent of the choice of $\varphi$ since $\varphi$ is unique modulo composition with a M\"obius transformation that maps $\mathbb D$ onto $\mathbb D.$
Equivalently,
$$\dist_h(x,\,y)=\inf_{\gamma}\int_{\gamma} \frac {2|g'(z)|}{1-|g(z)|^2}\, \d s(z),$$
where $g=\varphi^{-1}$ and the infimum is taken over all rectifiable curves that join $x$ to $y$ in $\Omega.$ 
Hyperbolic segments in $\Omega$ are then both minimizers of this integral and images of hyperbolic segments in the unit disk. Even though the hyperbolic metric
is defined via conformal maps, one can estimate it without knowing the map in question. For this one uses the following Koebe distortion theorem.

\begin{lemma}[{\cite[Theorem 2.10.6]{AIM2009}}]\label{koebe}
Suppose that $\varphi$ is conformal in a domain $\Omega\subsetneqq \mathbb C$ with $\varphi(\Omega)=\Omega'\subsetneqq \mathbb C$. Let $z_0\in \Omega$. Then
$$\frac 1 4 |\varphi'(z_0)|\dist(z_0,\,\partial \Omega)\le \dist(\varphi(z_0),\,\partial \Omega') \le |\varphi'(z_0)|\dist(z_0,\,\partial \Omega). $$
\end{lemma}




By the Koebe distortion theorem, the density in the definition of the hyperbolic distance is comparable to $\dist(z,\partial \Omega)^{-1}$ with absolute constants.
For example, in the upper half-plane $\mathbb H$ the hyperbolic metric has the density $y^{-1}$ at
the point $(x,\,y)\in \mathbb H$, and the hyperbolic geodesics are circular arcs perpendicular to the real axis (contained in half-circles with
center on the real axis) and segments of vertical lines ending at the real axis. 
See
\cite[Chapter 2]{AIM2009} for more information on the hyperbolic metric.

We will also need the hyperbolic distance in the complement of the closed unit disk and in complementary Jordan domains. Towards this, we recall that the hyperbolic
distance in the punctured disk $\mathbb D\setminus \{0\}$ is defined via the density $\rho(z)=\frac 1 {|z|\log(1/|z|)}$ and this time taking the infimum over curves in
$\mathbb D\setminus \{0\}.$ For the exterior of the closed unit disk, we transform this density and the hyperbolic distance via the (conformal) M\"obius transformation $\varphi(z)=\frac 1 z.$
Then the density of the hyperbolic distance is still controlled from above by an absolute constant multiple of $\frac 1 {|z|-1}=\dist(z,\,\partial \mathbb D)^{-1}$ (and also from below when $z\in B(0,\,10)$). 

Recall that a Jordan curve divides the plane into two domains, the
boundary of each of which equals to this curve;
we refer to the bounded one as a Jordan domain $\Omega.$ Then the Jordan domain $\Omega$ is conformally equivalent to the unit disk and the corresponding unbounded domain 
$\wz \Omega=\mathbb R^2\setminus \overline {\Omega}$ is conformally equivalent to 
$\mathbb R^2\setminus \overline {\mathbb D}.$  We define the hyperbolic distance and the corresponding density in $\wz \Omega$ via the conformal map and our hyperbolic distance and density
in the exterior of the closed unit disk. This does not depend on the choice of the conformal map in question since any two conformal maps from the exterior domain of the unit circle onto our Jordan domain 
can only differ by a precomposition with a rotation. This follows since the composition of the inverse of the second map with the first one would be a conformal self map of the exterior domain of the unit circle.
Each such a map is a rotation. This can be seen e.g. by pre- and postcomposing with the M\"obius transformation $\varphi(z)=\frac 1 z$ so as to obtain a conformal self map of the punctured disk, noticing that
the origin is a removable singularity and mapped to the origin by the extension. Thus we obtain a conformal self map of the unit disk that maps 0 to 0. Such maps are rotations.

Given a Jordan domain $\Omega$ and a conformal 
map
$\varphi\colon \mathbb D\to \Omega$ or $\varphi\colon \mathbb R^2\setminus
\overline{\mathbb D}\to  \mathbb R^2\setminus \overline{\Omega},$ our map
$\varphi$ extends homeomorphically up to the boundary by the Carath\'eodory-Osgood 
theorem \cite[Theorem 4.9, Page 445]{P1991}. 
Then the {\it hyperbolic ray in $\Omega,$}
ending at $z\in \partial
\Omega,$ is the image under $\varphi$ of the radial ray from the origin to
$\varphi^{-1}(z)$, or in $\mathbb R^2\setminus \overline \Omega$ the 
image under $f$
of the radial half-line starting from $f^{-1}(z).$  Most of the hyperbolic rays in a Jordan domain $\Omega$ have finite length 
in the sense that 
\begin{equation}\label{lyhyita}
\elle (\varphi([0,w)))<\infty
\mbox{ for a.e.}\ w\in S^1=\partial \mathbb D.
\end{equation} 
This follows since
$$\int_0^{2\pi}\int_0^1r|\varphi'(re^{i\theta})|\, dr\, d\theta=\int_{\mathbb D}|f'(z)|\le \pi^{1/2}\left(\int_{\mathbb D}|f'(z)|^2\right)^{1/2}=\pi^{1/2}\left(\int_{\mathbb D}J_{\varphi}(z)\right)^{1/2}$$
by the H\"older inequality and the Cauchy-Riemann equations; the integral of $J_{\varphi}$ over $\mathbb D$ is the area of $\Omega$ and hence finite
and $|f'(z)|$ is bounded in $\overline B(0,1/2)$ since $f$ is smooth.



The following lemma provides us with estimates on the oscillation of $|\varphi'(z)|$ in terms of the hyperbolic metric.

\begin{lemma}\cite[ Theorem 2.10.8]{AIM2009}\label{linearmap}
Suppose that $\varphi$ is conformal in $U,$ where $U$ is the unit disk $\mathbb D$ or $U=\R^2\setminus
\overline {\mathbb D},$ and let $z,\,w\in U.$ Then
$$\exp{(-3\dist_h(z,\,w))}|\varphi'(w)|\le |\varphi'(z)| \le\exp{(3\dist_h(z,\,w))}|\varphi'(w)|. $$
\end{lemma}



We also record the following estimates, referred to as the 
Gehring-Hayman
inequalities e.g. in  {\cite[Theorem 4.20, Page 88]{Ch1992}}. They show the significance of hyperbolic segments. 

\begin{lemma}[\cite{GH1962}]\label{lma:map and geodesic}
Let $\varphi\colon\mathbb D \to \Omega$ be a conformal map.
Given a pair of points
$x,\,y\in \overline{ \mathbb D}$, 
denoting the corresponding hyperbolic segment in $\mathbb D$ by $\Gamma_{x,\,y}$, 
and by $\gamma_{x,\,y}$ any 
curve connecting $x$ and $y$ in $ \mathbb D$, we have
$$\elle(\varphi(\Gamma_{x,\,y}))\le C \elle(\varphi(\gamma_{x,\,y}))$$
and
$$\diam(\varphi(\Gamma_{x,\,y}))\le C \diam(\varphi(\gamma_{x,\,y})),$$
where $C$ is an absolute constant. 
\end{lemma}

We close this subsection with the following lemma that employs hyperbolic segments. In general, the internal distance
between two boundary points of a Jordan domain can well be infinite.  However, such points cannot be obtained as 
Euclidean limits of sequences of points with uniformly bounded internal distances.

\begin{lemma} \label{tatatarvitaan}
Let $\Omega$ be a Jordan domain and $x_j,y_j\in \Omega,$ $j\in \mathbb N,$ be points so 
that $x_j\to x\in \overline {\Omega}$ and $y_j\to y\in \overline {\Omega}$ with $y\neq x.$
Then $$\elle(\Gamma_{x,\,y})\le C\liminf_{j\to \infty}\dist_{\Omega}(x_j,y_j)$$ for the hyperbolic segment $\Gamma_{x,\,y}$ between
$x$ and $y$ in $\Omega,$ where $C$ is an absolute constant. Especially, if $\dist_{\Omega}(x_j,y_j)\le M<\infty$ for
all $j,$ then  $\dist_{\Omega}(x,y)\le CM.$
\end{lemma}
\begin{proof}
Let $x_j,y_j\in \Omega$ be points so that $x_j\to x\in \overline {\Omega}$ and $y_j\to y\in \overline {\Omega}$ with $y\neq x.$
We may assume that $x_j\neq y_j.$
Pick rectifiable curves $\gamma_j$ joining $x_j$ to $y_j$ in $\Omega$ so that
\begin{equation}\label{tamatod1}
\elle(\gamma_j)\le 2\dist_{\Omega}(x_j,y_j).
\end{equation}
Fix a conformal map $\varphi\colon\mathbb D\to \Omega.$ By the Carath\'eodory-Osgood theorem we may extend $\varphi$ homeomorphically
up to the boundary. We refer also to this extension by $\varphi.$ Write $z_j=\varphi^{-1}(x_j), w_j=\varphi^{-1}(y_j).$ By Lemma \ref{lma:map and geodesic}
we conclude that
\begin{equation}\label{tamatod2}
\elle(\varphi(\Gamma_j))\le C \elle(\gamma_j)
\end{equation}
for the hyperbolic segment $\Gamma_j:=\Gamma_{z_j,w_j}$ between $z_j,w_j$ in $\mathbb D$ with an absolute constant $C.$
Since $\varphi$ is uniformly continuous, $x_j\neq y_j$ and $\lim_{j\to \infty}x_j=x\neq y=\lim_{j\to \infty}y_j,$ there exists $\delta>0$
so that $|z_j-w_j|\ge \delta$ for every $j\ge 1.$ Because each $\Gamma_j$ is a hyperbolic segment and hence an arc of a (generalized)
circle that intersects the unit circle orthogonally, we deduce from this the existence of a $\delta'>0$ and points $\xi_j\in \Gamma_j$ so
that $|\xi_j|\ge 1-\delta'$ for all $j\ge 1.$ Define
$$T_j(z)=\frac {z-\xi_j}{1+\overline \xi_j z},$$
where $\overline \xi_j$ is the complex conjugate of $\xi_j.$ Then $T_j$ is a conformal (M\"obius) self map of $\mathbb D$ and maps $\xi_j$ to the origin.
Now $T_j\circ \Gamma_j$ is a hyperbolic segment that contains 0 and hence a Euclidean line segment. On the other hand, a subsequence of the points $\xi_j$ converges
to some $\xi$ with $|\xi|\ge 1-\delta'$ and a subsequence of the corresponding $T_j\circ \Gamma_j$ converges to some curve $\alpha$ by Lemma \ref{arsela}.
Clearly $\alpha$ is a Euclidean line segment that contains the origin and hence also a hyperbolic segment.  Define
$$T(z)=\frac {z-\xi}{1+\overline \xi z}.$$
Then $T$ is a conformal self map of the disk and hence $T^{-1}\circ \alpha$ is a hyperbolic segment between $\varphi^{-1}(x)$ and $\varphi^{-1}(y).$ Consequently, $\Gamma:=\varphi\circ T^{-1}\circ \alpha$ is a hyperbolic segment
in $\Omega$ with end points $x,y.$

We are left to estimate the length of $\Gamma.$ By switching to a subsequence in the beginning of our proof, we may assume by \eqref{tamatod1} and \eqref{tamatod2} that
$$\liminf_{j\to \infty}\elle(\varphi(\Gamma_j))\le 2C \liminf_{j\to \infty} \dist_{\Omega}(x_j,y_j).$$
Since $\varphi$ is uniformly continuous and for the above subsequence  $T_j\circ \Gamma_j$ converge to $\alpha$ and $T_j^{-1}$ converge to $T^{-1},$ we have that $\varphi\circ \Gamma_j$ converge to $\varphi\circ T^{-1}\circ \alpha=\Gamma.$
Hence we may deduce from Lemma \ref{arsela} that
 $$\elle(\Gamma)\le 2C \liminf_{j\to \infty}\dist_{\Omega}(x_j,y_j).$$
\end{proof}

\subsection{Whitney-type set}

A \emph{dyadic square} in $\R^2$ refers to any set
\[
 [m_i 2^{-k},\, (m_i + 1)2^{-k}]\times[m_j2^{-k},\, (m_j + 1)2^{-k}], 
\]
where $m_i,\,m_j,\,k \in \mathbb Z$.
We denote by $\ell(Q)$ the side length of the given square $Q$. 
The barycenter of a set $A\subset \mathbb R^2$ with positive and finite Lebesgue-measure will be denoted by $x_A$ 
and for $c>0$ the dilation of $A$ by factor $c$ by
\[
 cA = \{c(x-x_A)+x_A\,:\,x \in A\}.
\]
We will use these concepts in particular for dyadic squares $Q$.

Recall that any open set in $\R^2,$ different from the entire  $\R^2,$
admits a 
Whitney decomposition; see e.g.\ \cite[Chapter VI]{stein1970}.

\begin{lemma}[Whitney decomposition]\label{lma:whitney}
 For any open set $U \neq \R^2$ there exists a collection
 $W=\{Q_j\}_{j\in\N}$ of countably many closed dyadic squares such that

 (i) $U=\cup_{j\in\N}Q_j$ and $(Q_k)^\circ\cap (Q_j)^\circ=\emptyset$ for all $j,\,k\in\N$  with $j\ne k$;

 (ii) $ \ell(Q_k)\le \dist(Q_k,\,\partial U)\le 4\sqrt{2} \ell(Q_k)$ for all $k\in\N$;

 (iii) $\frac 1 4 \ell(Q_k)\le  \ell(Q_j) \le 4 \ell(Q_k)$ whenever $k,\,j\in \N$ and  
$Q_k\cap Q_j\ne\emptyset$.
\end{lemma}

The above squares $Q_j$ are called Whitney squares of $U.$  We will also need the following more general concept since the
image of a Whitney square under a conformal map need not be a Whitney square.

\begin{definition}\label{whitney-type set}
 A bounded  connected set $A \subset U\neq \mathbb R^2$ is said to be of $\lambda$-Whitney type in $U$ 
(with constant $\lambda\ge 1$) if the following holds. 

(i) There exists a disk with radius $\frac {1}{\lambda}\diam(A)$ contained in $A$;

(ii) $ \frac {1}{\lambda} \diam(A)\le \dist(A,\,\partial U)\le {\lambda } \diam(A)$. 
\end{definition}
For example, the Whitney squares in Lemma~\ref{lma:whitney} are 
$4\sqrt 2$-Whitney-type sets. Conversely, each $\lambda$-Whitney-type set 
$A\subset U$ intersects at most $N(\lambda)$ Whitney squares of $U:$
by  (ii) of Lemma \ref{lma:whitney} and (ii) of Definition 
\ref{whitney-type set} we have that
$$Q\subset B(x,C(\lambda)\dist(x,\partial U))$$
with $C(\lambda)= \sqrt 2(\lambda +1)+\lambda$
for  any $x\in A$ and any Whitney square $Q$ of $U$ that intersects $A,$
and that
$$\ell(Q)\ge (5\sqrt 2)^{-1}\dist(A,\partial U)$$ for any such $Q.$

Observe that for a $\lambda$-Whitney-type set $A$ in $U$ and any $x\in A$, by the triangle inequality 
and (ii) of Definition 
\ref{whitney-type set} we have
\begin{equation}\label{equat90}
\dist(A,\,\partial U)\le \dist(x,\,\partial U)\le (1+\lambda)\dist(A,\,\partial U).
\end{equation}
Thus, if a pair $A_1,\,A_2$ of $\lambda$-Whitney-type sets has non-empty intersection, 
then 
\begin{equation}\label{equat91}
\diam(A_1)\sim \diam(A_2)
\end{equation}
with the constant depending only on $\lambda$.  

In terms of hyperbolic metric, Whitney-type sets have uniformly bounded 
diameter in the following sense.

\begin{lemma}\label{constant diam}
Let $\Omega$ be a Jordan (or exterior Jordan) domain in $\mathbb R^2$
and $A\subset \Omega$ a $\lambda$-Whitney-type set with $\lambda\ge 1$. Then 
\begin{equation}
\label{hyperbolic diameter}
\dist_h(x,\,y) \le C(\lambda)
\end{equation} 
for all $x,\,y\in A$. 
\end{lemma}
\begin{proof}
Let $x,y \in A$ be fixed.
Let $r = \frac1{10}\dist(A,\partial \Omega)$ and consider the cover $\{B(z,r)\}_{z \in A}$ of the set $A$. By the $5r$-covering lemma, there exists a pairwise disjoint subcollection $\{B(z_i,r)\}_{i=1}^N$ so that $A \subset \bigcup_{i=1}^NB(z_i,5r)$. For every $i$ we have $$B(z_i,r) \subset B(x,\diam(A)+r) \subset B(x,11r).$$
Hence, since $B(z_i,r)$ are pairwise disjoint, we have $N \le (11\lambda)^2$. By the fact that $A$ is connected, there exists a sequence $\{i(k)\}_{k=1}^M$ with $M \le N$ so that $x \in B(z_{i(1)},5r)$, $y \in B(z_{i(M)},5r)$, and $B(z_{i(k)},5r)\cap B(z_{i(k+1)},5r) \ne \emptyset$ for all $k = 1,\dots, M-1$. Then
the curve
\[
 \gamma = [x,z_{i(1)}]\ast [z_{i(1)},z_{i(2)}]\ast \cdots \ast 
 [z_{i(M-1)},z_{i(M)}]\ast [z_{i(M)},y]
\]
satisfies $\elle(\gamma) \le 10(N+2)r$ and $\dist(\gamma,\partial\Omega) \ge \dist(A,\partial\Omega) - 5r = 5r$.

Since the density of the hyperbolic metric is bounded from above by $\dist(z,\partial\Omega)^{-1}$ up to a multiplicative constant, we have
\[
 \dist_h(x,y) \le C\int_\gamma \dist(z,\partial\Omega)^{-1} \,ds(z) \le C\elle(\gamma)\dist(\gamma,\partial\Omega)^{-1} \le 2C(N+2),
\]
concluding the proof.

\end{proof}

Given a $\lambda$-Whitney-type set $A\subset \mathbb D$, one has $\dist_h(z,\,w)\le C(\lambda)$ for all $z,\,w\in A$ by \eqref{hyperbolic diameter}. Hence 
Lemma~\ref{linearmap} implies $|\varphi'(z)|\sim|\varphi'(w)|$ with a constant 
depending only on $\lambda$. 
By this (applied to suitable disks) in combination with  Lemma \ref{koebe}, one can prove that the images of Whitney squares of a simply connected domain $\Omega$ under a conformal map of
$\Omega$ onto $\Omega'$ get mapped  to Whitney type sets.
Following the idea from \cite[Theorem 11]{G1962}, a more general version of this can be proven with the help of Lemma \ref{koebe} and \cite[Theorem 18.1]{V1971}.
Since we will use it later on, let us recall that \cite[Theorem 18.1]{V1971}  gives the following: There exists a universal increasing (continuous) function $\Theta\colon (0,\,1) \to \mathbb R$ such that $\lim_{x\to 0^+} \Theta(x)=0$, $\lim_{x\to 1^-} \Theta(x)=\infty$ and for every conformal map $\varphi\colon \Omega  \to \Omega'$ with  domains $\Omega,\,\Omega'\subsetneqq \mathbb R^2$ and each point $x\in \Omega$, we have
\begin{equation}\label{quasi}
\frac{|\varphi(x) -\varphi(y)|}{\dist(\varphi(x),\,\partial \Omega')}\le \Theta\left(\frac{| x  - y|}{\dist( x ,\,\partial \Omega)}\right)
\end{equation}
for every $y$ with $0<|x-y|<\dist( x ,\,\partial \Omega). $


\begin{lemma}\label{whitney preserving} 
Suppose that $\varphi\colon \Omega \to \Omega'$ is conformal, 
where $\Omega,\Omega'\subsetneqq \R^2$  
are domains 
and $A\subset \Omega$ is a $\lambda_1$-Whitney-type set. Then 
$\varphi(A)\subset \Omega'$ is a $\lambda_2$-Whitney-type set with 
$ \lambda_2=\lambda_2(\lambda_1)$.  \end{lemma}

Our next estimate shows that a conformal map from the unit disk (or from the exterior domain of the disk) onto a simply connected domain (to an exterior domain) is $C_\lambda$-bi-Lipschitz modulo a scaling factor on each $\lambda$-Whitney-type set.

\begin{lemma} \label{bilipominaisuus}
Let $\varphi\colon U\to \Omega$ be conformal, where $U=\mathbb D$ or $U=\mathbb R^2\setminus \overline {\mathbb D}.$
If $A\subset U$ is of $\lambda$-Whitney type and $z_0,z_1,z_2\in A,$ then 
$$C_{\lambda}^{-1}|\varphi'(z_0)||z_2-z_1|\le |\varphi(z_2)-\varphi(z_1)|\le C_{\lambda}|\varphi'(z_0)||z_2-z_1|,$$
where $C_{\lambda}$ only depends on $\lambda.$
\end{lemma}
\begin{proof}
Fix $z_0,z_1,z_2\in A,$ where $A\subset  U$ is of $\lambda$-Whitney type.  
As in the proof of Lemma \ref{constant diam}, let $r=\frac {1}{10}\dist (A,\,\partial U).$  Then $\dist_h(z,z_1)\le C$ for an absolute constant when $z_2\in B(z_1,5r).$   Hence 
Lemma \ref{linearmap} gives us the estimate
\begin{equation}\label{tidea}
|\varphi(z_2)-\varphi(z_1)|\le \exp(C)|\varphi'(z_1)||z_2-z_1|
\end{equation}
if $z_2\in B(z_1,5r).$  
Let us assume that $z_2\notin B(z_1,5r).$ Then $|z_2-z_1|\ge 5r.$ Let $\gamma$ be the curve from the proof of Lemma \ref{constant diam} for the pair $z_1,z_2.$ Then $\elle(\gamma)\le (11\lambda)^2r$
and $\dist_h(z,z_1)\le 2C(2+(11\lambda)^2)=:C_1$ for each $z\in \gamma.$  Hence
Lemma \ref{linearmap} gives
\begin{equation}\label{tidea2}
|\varphi(z_2)-\varphi(z_1)|\le \int_{\gamma}|\varphi'(z)|\, ds(z)\le \exp(C_1)|\varphi'(z_1)|\elle(\gamma)\le \frac 1 5 (11\lambda^2)\exp(C_1)|\varphi'(z_1)||z_2-z_1|.
\end{equation}
By combining \eqref{tidea} and \eqref{tidea2}, Lemma \ref{constant diam} together with Lemma \ref{linearmap} allow us to further deduce that
\[
|\varphi(z_2)-\varphi(z_1)|\le C(\lambda)|\varphi'(z_1)||z_2-z_1|\le C'(\lambda)|\varphi'(z_0)||z_2-z_1|.
\]


Towards the opposite inequality, notice first that $\dist_h(\hat w,\hat z)\le C$ with an absolute constant whenever $\hat z\in \Omega$ and $\hat w\in B(\hat z,\dist(\hat z,\partial \Omega)/2).$ This holds since the density
of the hyperbolic distance in this disk is bounded from above by a fixed multiple of $\dist (\hat z,\partial \Omega)^{-1}.$ Especially, $\dist_h(\varphi^{-1}(\hat \xi),\varphi^{-1}(\hat z))\le C$ for each $\hat \xi \in B(\hat z,\dist (\hat z,\partial \Omega)/2).$ By Lemma \ref{linearmap} we conclude that
\begin{equation}\label{deriala}
|\varphi'(\varphi^{-1}(\hat z))|\le C_3|\varphi'(\varphi^{-1}(\hat \xi))|
\end{equation}
for all $\hat \xi \in B(\hat z,\partial \Omega)/2).$  
Since $\varphi^{-1}(\hat \xi)=\frac 1 {\varphi'(\varphi^{-1}(\hat \xi))},$ we deduce from \eqref{deriala} the estimate
\begin{equation}\label{deriyla}
|(\varphi^{-1})'(\hat \xi)|\le C_3 |(\varphi^{-1})'(\hat z)|.
\end{equation}
Let $I$ be the Euclidean line segment between $\hat z\in \varphi(A)$ and given $\hat w\in B(\hat z,\dist(\hat z,\partial \Omega)/2).$ Then $I\subset B(\hat z,\dist(\hat z,\partial \Omega)/2).$ By integrating the estimate \eqref{deriyla} over 
$I$ we
conclude that
\begin{equation}\label{kaanteisosk}
|\varphi^{-1}(\hat w)-\varphi^{-1}(\hat z)|\le \int_I |(\varphi^{-1})'(\hat \xi)|\, ds(\hat \xi)\le C_3|(\varphi^{-1})'(\hat z)||\hat w -\hat z|\le C_1|\varphi'(\varphi^{-1}(\hat z))|^{-1}|\hat w-\hat z|.
\end{equation}
Especially, if $\hat w,\hat z\in \varphi(A)$ satisfy
\begin{equation}\label{wzlahella}
|\hat w-\hat z|\le \frac 1 2 \max\{\dist (\hat z,\partial \Omega),\dist (\hat w, \partial \Omega)\},
\end{equation}
then by \eqref{kaanteisosk}, Lemma \ref{linearmap} and  Lemma \ref{constant diam} we get
$$|\varphi^{-1}(\hat w)-\varphi^{-1}(\hat z)|\le C_3\max\{|\varphi'(\varphi^{-1}(\hat \xi))|^{-1}:\ \hat {\xi}\in \varphi(A)\}|\hat w-\hat z| \le C_1(\lambda)C_3|\varphi'(z_0)|^{-1}|\hat w -\hat z|.$$

We are left to consider the case where \eqref{wzlahella} fails. By Lemma \ref{whitney preserving} we know that $\varphi(A)$ is of $C_4(\lambda)$-Whitney type. Hence
\begin{equation}\label{whitneydiamet}
\diam \varphi(A)\le C_5(\lambda)\dist(\varphi(A),\partial \Omega)\le C_5 \dist(\hat z,\partial \Omega)
\end{equation}
for each $\hat z\in A.$ If \eqref{wzlahella} fails, then $\dist(\hat z,\partial \Omega)\le 2|\hat w-\hat z|$ and we conclude that
$$\diam(\varphi(A))\le 2 C_5(\lambda) |\hat w-\hat z|$$
and further that
$$|\varphi^{-1}(\hat w)-\varphi^{-1}(\hat z)|\le \diam(A)=\diam(A)|\hat w-\hat z|^{-1}|\hat w -\hat z|$$
$$\le 2C_5(\lambda)\diam(A)\diam(\varphi(A))^{-1}|\hat w-\hat z|.$$
It only remains to be noticed that $|\varphi'(z_0)|$ is comparable to $\dist(\varphi(z_0),\partial \Omega)/\dist(z_0,\,\partial U)$ with absolute constants by Lemma \ref{koebe}, that 
$C_5(\lambda)\dist(\varphi(z_0),\partial\Omega)\ge \diam(\varphi(A))$ by \eqref{whitneydiamet} and that $\diam(A)\le C_6(\lambda)\dist(z_0,\,\partial U)$ since $A$ is of 
$\lambda$-Whitney type in $U$ with $z_0\in A.$
\end{proof}

\subsection{Conformal capacity}

Let $\Omega\subset \mathbb R^2$ be a  domain. For a given pair of disjoint continua $E,\,F \subset \overline \Omega$, 
define the {\it conformal capacity between $E$ and $F$ in $\Omega$} as
 $${\rm Cap}(E,\,F,\,\Omega)=\inf\{\|\nabla u\|^2_{L^{2} (\Omega)}: \ u\in\Delta(E,\,F,\,\Omega )\},$$
 where $\Delta(E,\,F,\,\Omega )$ denotes the class of all $u\in W^{1,\,2}_{\loc}(\Omega) \cap C(\Omega \cup E \cup F)$ that satisfy $u=1$ on $E$, and $u=0$ on $F$. 
We remark that in general one has to be careful when defining the capacity of continua that are allowed to intersect the boundary of the domain. However, for our purposes the definition above is enough, so we can avoid considering the prime-end compactification of $\Omega$ and the more subtle definitions of capacity.
 The conformal capacity is by definition increasing both with respect to the sets $E,F$ and with respect to
$\Omega.$  We refer to this by the monotonicity of conformal capacity. Also notice that ${\rm Cap}(E,\,F,\,\Omega)={\rm Cap}(F,\,E,\,\Omega),$ 
as  seen by switching a given test function $u$ to $v=1-u.$

Let us introduce the properties of conformal capacity which will be used in the rest of the paper; we refer to e.g.\ \cite[Chapter 1]{V1971} for more properties. We remark that, even though \cite{V1971} (as well as some other references below) states estimates for ``modulus'',   ``modulus'' is equivalent with
conformal capacity in our setting below (see e.g. \cite[Theorem 2.6]{heko90},
\cite[ Proposition 10.2, Page 54]{R1993}).

\begin{lemma}
The conformal capacity is conformally invariant, that is, for  domains $\Omega$ and $\Omega'$ in $\mathbb R^2$, a conformal (onto) map $\varphi \colon \Omega \to \Omega'$ and disjoint continua $E$ and $F$ in $\Omega$, we have
\begin{equation}\label{cap inv}
{\rm Cap}(\varphi(E),\,\varphi(F),\,\Omega') = {\rm Cap}(E,\,F,\,\Omega).
\end{equation}
Moreover, if $\varphi$ has a homeomorphic extension, still denoted by  $\varphi$, $\varphi\colon \overline{\Omega}\to \overline{\Omega'}$, then \eqref{cap inv} also holds for disjoint continua in $\overline{\Omega}$. Especially this is the case if both $\Omega$ and $\Omega'$ are Jordan. 
\end{lemma}

In what follows, whenever we mention the conformal invariance of conformal capacity, we always refer to the above lemma.

We have the following 
estimate for the conformal capacity in the unit disk $\D$ (and in its exterior domain $\mathbb R^2\setminus \overline{\D}$). Let $E$ and $F$ be 
disjoint 
continua in $\overline{\D}$. Then 
\begin{equation}
\label{condition of lower bound}
{\rm Cap}(E,\,F,\,\D) \ge c \log \left(1+\frac{\min\{\diam (E),\,\diam (F)\}}{\dist(E,\,F)}\right)
\end{equation}
where $c>0$ is a universal constant. Moreover, the analogous inequality holds for  $E,\,F\subset \mathbb R^2 \setminus {\mathbb \D}.$  For these results see e.g.\ 
\cite[Lemma 7.38]{V1988} that gives \eqref{condition of lower bound} for the entire plane
instead of $\D$ and \cite[Remark 2.12]{GM1985 2}, \cite[Theorem 2.6, Theorem 2.8]{heko90}
that allow us to deduce the desired estimates from the global one.

We call a domain $A\subset\mathbb R^2$ a {\it ring domain} if its complement has exactly two components and at least one of the components is compact. If the exterior components of $A$ are $U_0$ and $U_1$, then we write $A=R(U_0,\,U_1)$.
It follows from topology that also $\partial A$ has two components, $V_0=U_0\cap \overline{A}$ and $V_1=U_1\cap \overline{A}$.  If $U_0$, $V_0$ and $V_1$ are compact, we have
\begin{equation}\label{same capacity}
  {\rm Cap}( V_0,\, V_1 ,\,A)= {\rm Cap}( U_0,\, V_1 ,\, A\cup U_0 ); 
\end{equation}
indeed, ``$\le$'' directly follows from the definition and ``$\ge$'' follows by extending each $u\in \Delta(V_0,\,V_1,\,A)$ as constant $1$ to $U_0\setminus V_0$, see also \cite[Theorem 11.3]{V1971} (and its proof). 
Furthermore, we have the following estimate for the capacity of the boundary components of a 
ring domain. 

\begin{lemma}
Let $A=R(U_0,\,U_1)\subset\mathbb R^2$ be a ring domain with $U_1$ unbounded. Assume that $V_0=U_0\cap \overline{A}$ and $V_1=U_1\cap \overline{A}$ are compact.  There exist two universal increasing functions $\phi_i \colon (0,\infty) \to (0,\infty),\, i=1,\,2,$ so that $\lim_{t\to 0+} \phi_i(t) = 0$ and $\lim_{t\to \infty}\phi_i(t) = \infty$, and so that  
\begin{equation}\label{capacity of balls}
\phi_1\left(\frac{ \diam (U_0) }{\dist(U_0,\, U_1 )}\right) \le {\rm Cap}( V_0,\, V_1 ,\,A)\le \phi_2\left(\frac{ \diam (U_0) }{\dist(U_0,\, U_1)}\right). 
\end{equation}
\end{lemma}

%

We need the fact that the inner distance satisfies the triangle inequality 
 \cite[Lemma 2.3]{Br06}.

\begin{lemma}\label{inner triangle}
Let $\Omega$ be a Jordan domain and $z_1,\,z_2,\,z_3\in \overline{\Omega}$ be 
three distinct points. Then 
$$\dist_{\Omega}(z_1,\,z_3)\le \dist_{\Omega}(z_1,\,z_2)+\dist_{\Omega}(z_2,\,z_3).$$
\end{lemma}

We record the following estimate, which states a kind of converse to  \eqref{condition of lower bound}. It builds on \cite[Lemma 2.2]{KZ2016}. 
(Recall from Subsection~\ref{subsection:curves} that $\diam_{\Omega}$ and $B_\Omega(z,r)$ refer to the diameter and ball in the inner distance with respect to $\Omega$.)

\begin{lemma}{\label{inner capacity}}
Let $\Omega$ be a domain and  $E,\,F\subset \Omega$ be a pair of disjoint continua. 
Then if ${\rm Cap}(E,\,F,\, \Omega)\ge \delta_0>0$, we have
\begin{equation}\label{inequat 1}
{\min\{\diam_{\Omega}(E),\,\diam_{\Omega}(F)\}}\gtrsim {\dist_{\Omega}(E,\,F)},
\end{equation}
where the constant only depends on $\delta_0$. Especially
$${\min\{\diam_{\Omega}(E),\,\diam_{\Omega}(F)\}}\gtrsim {\dist (E,\,F)}, $$
and if $\Omega=\mathbb R^2$
\begin{equation}\label{inequat 10}
\min\{\diam(E),\,\diam(F)\}\gtrsim {\dist (E,\,F)}.  
\end{equation}
If we further assume that $\Omega$ is Jordan, then \eqref{inequat 1} also holds if $E\subset \overline \Omega$ and $F\subset\overline\Omega$  are disjoint continua with ${\rm Cap}(E,\,F,\, \Omega)\ge 
\delta_0$. 
\end{lemma}
\begin{proof}
\noindent{\bf Step 1:}
We begin with the case where $E,\,F\subset \Omega$. 
By switching $E$ and $F$ we may assume that $\diam_{\Omega}(E)\le \diam_{\Omega}(F).$ We may also assume that
$2\diam_{\Omega}(E) \le \dist_{\Omega}(E,\,F)$; otherwise the claim holds trivially. 
Fix $z\in E$, and write $\frac{\dist_{\Omega}(E,\,F)}{\diam_{\Omega}(E)}=M$. 
We define
$$ u(x)=
 \begin{cases}
   1, & \text{if } \dist_{\Omega}(x,\,z)\le \diam_{\Omega}(E) \\
  0, & \text{if }  \dist_{\Omega}(x,\,z)\ge \dist_{\Omega}(E,\,F)\\
      \frac {\log (\dist_{\Omega}(E,\,F))-\log (\dist_{\Omega}(x,\,z))}
      {\log(M)},
      &\text{otherwise }
 \end{cases}.$$

Then $u$ is locally Lipschitz and 
$$|\nabla u(x)|\le (\log M)^{-1} \dist_{\Omega}(x,\,z)^{-1}$$
for all $x \in \Omega$ with $\diam_{\Omega}(E) \le \dist_{\Omega}(x,\,z)\le \dist_{\Omega}(E,\,F)$, and
$|\nabla u(x)|=0$ elsewhere.
Write
$$R=B_{\Omega}(z,\,\dist_{\Omega}(E,\,F))\setminus B_{\Omega}(z,\,\diam_{\Omega}(E)),$$
and for $i\ge 1$
$$A_i=B_{\Omega}(z,\,2^{i}\diam_{\Omega}(E))\setminus B_{\Omega}(z,\,2^{i-1}\diam_{\Omega}(E)),$$
where $B_{\Omega}(z,\,r)$ is the disk centered at $z$ with radius $r$ with respect to the inner distance. 
The assumption ${\rm Cap}(E,\,F,\, \Omega)\ge \delta_0>0$ and a direct calculation via our dyadic annular decomposition with respect
to the inner distance give
\begin{align*}
\delta_0\le & \int_{\Omega} |\nabla u|^2\, dx  
\leq   (\log M)^{-2}\int_{R} \dist_{\Omega}(x,\,z)^{-2} \, dx\\ 
\leq &   (\log M)^{-2} \sum_{i=1}^{\infty}\int_{R\cap A_i } 2^{2-2i}\diam_{\Omega}(E)^{-2} \, dx\\
\leq &   2(\log M)^{-2} \sum_{i=1}^{[\log M]+1} 4\pi\\
\ls &  (\log M)^{-2} \log M \lesssim  {(\log M)}^{-1}, 
\end{align*}
where  $[\log M]$ denotes the integer part of $\log M$, and in the third inequality we used the fact that $B_{\Omega}(z,\,r)\subset B(z,\,r)$.
Hence $M\le C(\delta_0)$, which means that 
$\dist_{\Omega}(E,\,F) \lesssim \diam_{\Omega}(E)$.

\noindent{\bf Step 2:} We continue with the case where at least one of the sets $E,\,F$ intersects $\partial\Omega$. 
We cannot directly use a test-function defined like the function $u$ from the previous step since it would not necessarily be continuous in $\Omega\cup E\cup F.$

To begin, let $\varphi\colon \overline{\mathbb D} \to \overline{\Omega}$ be a 
homeomorphism, conformal in $\mathbb D,$ given by the Riemann Mapping and 
Carath\'eodory-Osgood 
theorems. 
Since 
$${\rm Cap}(E,\,F,\, \mathbb R^2)\ge {\rm Cap}(E,\,F,\, \Omega)\ge \delta_0$$
by monotonicity,
we conclude by \eqref{inequat 10}, which is given by Step 1, that neither $E$ nor $F$ is a singleton.

Suppose that $E\subset \partial \Omega.$
Then $\varphi^{-1}(E)$ is a closed nondegenerate arc contained in the unit circle and hence \eqref{lyhyita} provides us with $w\in \varphi^{-1}(E)$ for which
$\elle(\varphi([0,w]))<\infty.$  We choose $0<t_E<1$ so that 
$$\diam_{\Omega}(\varphi([t_Ew,w])\cup E)\le 2 \diam_{\Omega}(E)$$ and 
$$\dist_{\Omega}(\varphi([t_Ew,w])\cup E,F)\ge \frac 1 2 \dist_{\Omega}(E,F).$$  Set $E':=\varphi([t_Ew,w]))\cup E=\varphi([t_Ew,w]\cup \varphi^{-1}(E)).$  Then $E'$ is compact and connected. Notice that $E'$ intersects $\Omega$  and $E\subset E'.$ 

If $E$ intersects $\Omega,$ we simply let $E'=E.$
We construct $F'$ in the analogous manner, considering now the internal distance to $E'.$ Then $E',F'\subset \overline \Omega$ are continua,
$$\diam_{\Omega}(E')\le 2 \diam_{\Omega}(E),$$ 
$$\diam_{\Omega}(F')\le 2 \diam_{\Omega}(F)$$
and
$$\dist_{\Omega}(E,F)\le 4\dist_{\Omega}(E',F').$$
Moreover, $E\subset E',F\subset F'$ and hence monotonicity of capacity
together with our capacity assumption ensures that 
$${\rm Cap}(E',\,F,'\, \Omega)\ge \delta_0.$$
Hence it suffices to prove \eqref{inequat 1} for $E',F'$ instead of $E,F.$  For simplicity of notation, we refer to $E',F'$ by $E,F$ in what follows.


By switching the roles of $E$ and $F$ if necessary, it suffices to show that
\begin{equation}\label{riittaa}
C\diam_{\Omega}(E)\ge \dist_{\Omega}(E,F)
\end{equation}
for some constant that may only depend on $\delta_0.$ At this point, it is perhaps worth pointing out that the right-hand side is finite since both $E$ and $F$ intersect $\Omega.$

Towards \eqref{riittaa} we first pick $z_0\in E\cap \Omega$ and choose a new conformal map $\psi\colon\Omega \to \mathbb D$ so that $\psi(z_0)=0.$ This can be done by post-composing $\varphi^{-1}$ with
a suitable conformal self (M\"obius) map of the disk. Given $j\ge 1,$ we define $E_j=\psi^{-1}((1-\frac 1 j)\psi(E))$ and $F_j=\psi^{-1}((1-\frac 1 j)\psi(F)).$ Here $tA$ for a subset of the unit disk is the image
of $A$ under the map $f(z)=tz.$ Then $E_j,F_j\subset \Omega$ are disjoint continua. By conformal invariance and monotonicity
\begin{equation}\label{pitka}
{\rm Cap}(E_j,\,F_j,\, \Omega)={\rm Cap}(\psi(E_j),\,\psi(F_j),\, \mathbb D)\ge {\rm Cap}\left(\psi(E_j),\,\psi(F_j),\, B\left(0,1-\frac 1 j\right)\right).
\end{equation}
Notice that $\psi(E_j)=(1-\frac 1 j)\psi(E)$ and $\psi(F_j)=(1-\frac 1 j)\psi(F)$ are the images of $\psi(E)$ and $\psi(F),$ respectively, under the conformal map $f(z)=(1-\frac 1 j)z$ and $B(0,1-\frac 1 j)$ is the image
of the unit disk under this map. Hence conformal invariance gives
\begin{equation}\label{lyhyt}
{\rm Cap}\left(\psi(E_j),\,\psi(F_j),\, B\left(0,1-\frac 1 j\right)\right)={\rm Cap}(\psi(E),\,\psi(F),\, \mathbb D)={\rm Cap}(E,\,F,\, \Omega).
\end{equation}
By combining \eqref{pitka} and \eqref{lyhyt} we conclude from our capacity assumption on $E,F$ that
$${\rm Cap}(E_j,\,F_j,\, \Omega)\ge \delta_0$$
for all $j\ge 1.$
Hence we are allowed to infer from Step 1 that
\begin{equation}\label{imar}
\dist_{\Omega}(E_j,\,F_j)\le C\diam_{\Omega}(E_j)
\end{equation}
with $C=C(\delta_0)$ and all $j\ge 1.$

We proceed with a limiting process that relies on \eqref{imar}. First, notice that $0\in \psi(E_j)$ for all $j.$ Hence, by Lemma~\ref{inner triangle} there exists $z_j$ in $E_j$ so that $\diam_{\Omega}(E_j)\le 3\dist_{\Omega}(\psi^{-1}(0),z_j).$
Thus 
\begin{equation}\label{amar}
\diam_{\Omega}(E_j)\le 3\elle(\psi^{-1}([0,\psi(z_j)),
\end{equation}
where $[0,\psi(z_j)]$ is the radial segment between the points $0$ and $\psi(z_j)$. Since $\psi(E_j)=(1-\frac 1 j)\psi(E),$ the point $\xi_j=(1-\frac 1 j)^{-1}\psi(z_j)$ belongs to $\psi(E).$ Since also $0\in \psi(E)$ we
deduce that $\diam_{\Omega}(E)\ge \dist_{\Omega}(\psi^{-1}(0),\psi^{-1}(\xi_j)).$
By Lemma \ref{lma:map and geodesic} we conclude that 
\begin{equation}\label{bmar}
\elle(\psi^{-1}([0,\xi_j]))\le C_1 \diam_{\Omega}(E)
\end{equation}
with an absolute constant $C_1$. By the definition of $\xi_j$ we have that $[0,\psi(z_j)]\subset [0,\xi_j]$ and hence \eqref{amar} together with \eqref{bmar} gives the uniform estimate
\begin{equation}\label{cmar}
\diam_{\Omega}(E_j)\le 3C_1 \diam_{\Omega}(E).
\end{equation}
Next, \eqref{imar} and Lemma \ref{lma:map and geodesic} provide us with points $x_j\in E_j,$ $y_j\in F_j$ and corresponding hyperbolic segments $\Gamma_j$ between $x_j$ and $y_j$ in $\Omega$ so
that
\begin{equation} \label{dmar}
\elle(\Gamma_j)\le C_2 \diam_{\Omega}(E_j)
\end{equation}
where $C_2$ depends only on $\delta_0.$
Since $\Omega$ is a Jordan domain, it is especially bounded and hence, by switching to a subsequence if necessary, we may assume that $x_j\to x$ and $y_j\to y,$ where $x,y\in \overline \Omega.$
This together with  \eqref{cmar} and \eqref{dmar} allow us to employ Lemma \ref{tatatarvitaan} so as to conclude that
$$\dist_{\Omega}(x,y)\le C_3 \diam_{\Omega}(E).$$
By construction, $\psi(x_j)\in (1-\frac 1 j)E$ and $\psi(y_j)\in (1-\frac 1 j)F$ and since $\psi$ is homeomorphic up to boundary, we deduce that $x\in E$ and $y\in F.$ 
Thus we have established \eqref{riittaa} and the proof is complete.
\end{proof}

\subsection{John domains}

Let us recall the definition of a John domain. 

\begin{definition}[John domain]\label{def:John}
An open bounded subset $\Omega\subset \mathbb R^2$ is called a John domain provided it satisfies 
the 
following condition:
There exist a distinguished point $x_0 \in \Omega$ and a constant $J>0$ such that, for every 
$x\in\Omega$, 
there is a rectifiable curve $\gamma$ joining $x$ and $x_0$ in $\Omega$ and satisfying
\[
\dist(y,\, \R^2\setminus\Omega)\ge J\elle(\gamma[x,y]) 
\]
for all $y \in \gamma$.
Such a  curve $\gamma$ is called a $J$-John curve, $J$ is called a John 
constant, and we refer to a John domain with a John constant $J$ by a $J$-John domain and to $x_0$ by a John center of $\Omega.$
\end{definition}

We continue with results related to John domains. For the convenience of the reader, we refer to \cite{nava1991} whenever possible, even when the
result in question has a longer history.

As an example, every disk is a 1-John domain with the center as the John center and radial segments as John curves. In fact, 
one can always choose hyperbolic segments for the John curves in the simply connected situation.

\begin{lemma}\cite[Theorem 5.4]{nava1991}\label{hsjohn}
If $\Omega$ is a simply connected $J$-John domain, then hyperbolic segments from the John center $x_0$ to points in $\Omega$ are $J'$-John curves, 
where $J'$ depends
only on $J$. 
\end{lemma}

\begin{remark}\label{muistutus}
Actually, also the hyperbolic segment $\Gamma$ connecting $x_0$ and 
$y\in \partial \Omega$ is a $J'$-John curve for a simply connected planar 
$J$-John domain 
$\Omega$ with the base point $x_0.$ This follows from the preceding lemma and the definition of a 
hyperbolic segment: the hyperbolic segment between $x_0$ and a point $x$ on this segment is the part of $\Gamma$ between $x_0$ and $x.$ Consequently, any two points $x,y\in \overline \Omega$ are 
rectifiably joinable
and the diameter of a simply connected John domain with respect to the inner distance is
finite.

For further reference let us record the following consequence that deals with integrals as in \eqref{eq:extcharcompl}. Let $1<p<2.$ 

By parametrizing $\Gamma$ via arc length the John
condition and integration gives
$$\int_{\Gamma}\dist(z,\partial \Omega)^{1-p}\, ds(z)\le (J')^{1-p}
\elle(\gamma)^{2-p}\le C(p,J')\dist(x_0,\partial \Omega)^{2-p}.$$
\end{remark}

We move towards explaining the role of John domains in our work.

\begin{definition}
A set $E$ is of \it{bounded turning} if there is a constant $C$ such that any pair of points $z_1,z_2$ can be joined by a curve $\gamma\subset E$ whose diameter satisfies
$\diam(\gamma)\le C|z_1-z_2|.$ We then say that $E$ is of $C$-bounded turning.
\end{definition} 

Recall that quasiconvexity was defined analogously but with length instead of diameter. Hence being of bounded turning is a weaker condition than being quasiconvex.

\begin{lemma}\cite[Theorem 4.5]{nava1991} \label{btchar}
Let $\Omega$ be a bounded simply connected planar domain. Then $\Omega$ is John if and only if $\mathbb R^2\setminus \Omega$ is of bounded turning.  This equivalence is
quantitative in the sense that the John constant and the constant in bounded turning depend only on each other.
\end{lemma}

By Lemma \ref{kvasikonveksi} the complement of a bounded simply connected domain whose complement satisfies \eqref{eq:extcharcompl} is $C'$-quasiconvex and hence
also of $C'$-bounded turning with a constant that only depends on the exponent $p$ and the constant $C$ in \eqref{eq:extcharcompl}. Hence we obtain the following important
corollary to the preceding lemma.

\begin{corollary}\label{komplementtionjohn}
Let $\Omega$ be a bounded simply connected domain whose complement satisfies \eqref{eq:extcharcompl}.
Then $\Omega$ is $J$-John with a constant $J$ that only depends on the exponent $p$ and the constant $C$ in \eqref{eq:extcharcompl}.
\end{corollary}

We will need the fact that the boundaries of bounded simply connected John domains are of area zero.  

\begin{lemma}\label{bdyzero}
If $\Omega$ is a bounded simply connected planar John domain, then the Lebesgue area of $\partial \Omega$ is zero.
\end{lemma}

Conformal maps from the unit disk onto a John domain behave nicely with respect to the inner distance.  In order to
state this quantitatively we need a definition.

We say that a homeomorphism $\varphi\colon \mathbb D\to \Omega $ is 
{\it{quasisymmetric with respect to the
inner distance}} if there is a homeomorphism $\eta\colon[0,\infty)\to [0,\infty)$ so that
\begin{equation}\label{qsdef}
|z-x|\le t|y-x|\mbox{ implies }\dist_{\Omega}(\varphi(z),\varphi(x))\le 
\eta(t)\dist_{\Omega}(\varphi(y),\varphi(x))
\end{equation}
for each triple $z,x,y$ of points in $\mathbb D.$

It follows from the definition that the inverse of a quasisymmetric map is also quasisymmetric: 
$$\dist_{\Omega}(z,x)\le t\dist_{\Omega}(y,x) \mbox{ implies }|\varphi^{-1} (z)-\varphi^{-1}(x)|\le \eta(1/t)^{-1}|\varphi^{-1}(y)-\varphi^{-1}(x)|$$
when $z,x,y\in \Omega.$

Roughly speaking, the definition means that the homeomorphism $\varphi$ maps round 
objects to essentially round objects (with respect to the inner distance).  The following result will be an important technical tool for us.

\begin{lemma}\cite[Theorem 7.2]{nava1991}\label{quasisymmetry}
Let $\Omega\subset \mathbb R^2$ be a simply connected domain, and let
$\varphi\colon \mathbb D\to \Omega $ be a 
conformal map. Then $\Omega$ is John  if and only if $\varphi$ 
is quasisymmetric with 
respect to the inner distance. This statement is quantitative in the sense
that the John constant and the function $\eta$ in quasisymmetry depend
only on each other and on $\diam(\Omega)/\dist(\varphi(0),\partial\Omega).$
\end{lemma}

\begin{remark}\label{uptobdy}
Notice that quasisymmetry is a strong version of uniform continuity of the conformal map from the unit disk onto 
$\Omega$ equipped with the inner distance. Hence the quasisymmetry condition extends up to the boundary:
one is allowed to use \eqref{qsdef}, when correctly interpreted, for triples of points in $\overline {\mathbb D}.$ For example, when $\Omega$ is a slit disk, 
every other point of the slit than the tip
corresponds to two different points with respect to the completion of $\Omega$ in the inner metric. In what follows, 
we will
only use the quasisymmetry condition up to the boundary in situations where $\Omega$ is a Jordan John domain where this is not an issue. However,
we will later employ the fact that a conformal map from the unit disk onto a bounded simply connected John
domain extends continuously to $\partial \mathbb D,$ with respect to the Euclidean distances. This follows since
quasisymmetry of a quasisymmetric map from the unit disk implies uniform continuity with respect to the
Euclidean distances.
\end{remark}

Recall that hyperbolic segments in the unit disk are arcs of (generalized) circles perpendicular to the unit circle.
Hence they are essentially the shortest connecting curves and stay ``optimally away from the boundary". 
The following definition gives an analog of this property.

\begin{definition}[Inner uniform domain]\label{inneruniform}
A domain $\boz$ is called   {\it inner uniform} if there exists a
positive constant $\ez_0$ such that for any pair of points $x,\,y\in\boz$, there exists a 
rectifiable curve $\gz\subset\boz$ joining $x$, $y$ and satisfying
\begin{equation}\label{e1.1}
\elle(\gz)\le \frac1{\ez_0}\dist_{\Omega}(x,\,y) \ \ \mbox{and} \  \ \dist(z,\, \partial \boz)\ge\ez_0 \min\{\elle(\gamma[x,z]),\,\elle(\gamma[z,y])\}\ \mbox{for all}\ z\in\gz.
\end{equation}
\end{definition}

Since a conformal map of the unit disk onto a John domain is quasisymmetric, the definition of hyperbolic segments and Lemma~\ref{lma:map and geodesic} suggest that
each bounded simply connected John domain should be inner uniform. This is indeed the case.

\begin{lemma} \cite[Theorem 2.29, Example 2.18 (2)]{nava1991} \label{johninner}
Let $\Omega$ be a bounded simply connected $J$-John domain. Then $\Omega$ is inner uniform with an associated  constant $\ez_0$ that only depends on $J.$
Moreover, the curves in the definition may be chosen to be hyperbolic segments. 
\end{lemma}

We continue by relating the inner and Euclidean diameters of boundary arcs of a Jordan John domain.

\begin{lemma}\label{inner diameter diameter}
Let $\Omega$ be a Jordan $J$-John domain and let $\gamma\subset \partial \Omega$ be a subarc containing its endpoints. Then we have
$$\diam(\gamma)\le \diam_{\Omega}(\gamma)\le C\diam(\gamma),$$
where $C$ depends only on $J.$
\end{lemma}
\begin{proof}
We only need to show that
$$\diam_{\Omega}(\gamma)\le C\diam(\gamma),$$
since the first inequality is trivial. Pick $x,\,y\in \gamma$ such that
$$\diam_{\Omega}(\gamma)\le 2\dist_{\Omega}(x,\,y).$$
By the definition of inner distance, the hyperbolic segment $\Gamma$ 
joining $x,\,y$ satisfies
$$\dist_{\Omega}(x,\,y)\le \elle(\Gamma).$$
Let $z$ be the midpoint (in the sense of length) of $\Gamma$. Then since $\Omega$ is a John domain and $\Gamma$ is a hyperbolic segment, by applying Lemma \ref{johninner} to pairs of points on $\gamma$ converging to the end points of $\gamma$, we deduce that 
$$\elle(\Gamma)\le C(J) \dist(z,\,\partial \Omega). $$
Hence we have 
\begin{equation}\label{inequ 201}
\diam_{\Omega}(\gamma)\le C(J) \dist(z,\,\partial \Omega).
\end{equation}

Fix a conformal map $\varphi\colon\mathbb D\to \Omega.$ Since $\Omega$ is Jordan,
$\varphi$ extends to a homeomorphism (still denoted $\varphi$) of $\overline
{\mathbb D}$ onto $\overline \Omega.$
Let $B$ be the closed disk of radius $\frac 1 8 |1-\varphi^{-1}(z)|,$ 
tangent to the circular arc $\varphi^{-1}(\Gamma)$ at $\varphi^{-1}(z),$ 
and contained in the Jordan domain enclosed by $\varphi^{-1}(\Gamma)$ and 
$\varphi^{-1}(\gamma)$; recall that $\varphi^{-1}(\Gamma)$ is a hyperbolic 
segment in $\mathbb D$ and hence a circle that meets the unit circle orthogonally at two end points of the arc $\varphi^{-1}(\gamma)$ of the unit circle. Since $B$ is a $3$-Whitney-type set, by Lemma~\ref{whitney preserving}, $Q'=\varphi(B)$ is 
a $\lambda'$-Whitney-type set. Here $\lambda'$ is an absolute constant. 
Let $\alpha$ be the radial projection of $B$ to
$\partial \mathbb D.$ Then $\alpha\subset \varphi^{-1}(\gamma)$ and  $\diam(B)=\frac 1 4 |1-\varphi^{-1}(z)|\le \diam(\alpha).$
Hence
$$
\diam(\varphi^{-1}(\gamma))
\ge  \diam(B)
\ge \frac 1 4 \dist(B,\,\alpha)\ge \frac 1 4 \dist(B,\,\varphi^{-1}(\gamma)).
$$
Consequently, by \eqref{condition of lower bound} 
$${\rm Cap}( B,\,\varphi^{-1}(\gamma),\,\mathbb D)\ge \delta>0$$
for an absolute constant $\delta.$
By the conformal invariance of capacity and monotonicity, 
$$ \delta \le {\rm Cap}( Q' ,\, \gamma,\,\Omega) \le {\rm Cap}( Q' ,\, \gamma,\,\mathbb R^2),$$
which with Lemma~\ref{inner capacity} implies 
\begin{equation} \label{ddg}
  {\dist(Q',\, \gamma)}\le C(\delta) {\diam(\gamma)} . 
\end{equation}
Since $Q'$ is of $\lambda'$- Whitney-type and $z\in Q'$ we conclude via
\eqref{equat90} and \eqref{ddg} that
$$\dist(z,\,\partial \Omega)\sim \diam(Q')\ls \dist(Q',\, \gamma)\ls \diam(\gamma) $$
where all the constants are absolute. 
This together with \eqref{inequ 201} gives
$$\diam_{\Omega}(\gamma) \ls \dist(z,\,\partial \Omega)\ls \diam(\gamma) $$
with constants depending only on $J$ as desired. 
\end{proof}

Based on the above lemma, one would expect $\partial \Omega$ to be compact
with respect to the inner metric for each Jordan John domain. This is
indeed the case by \cite[Remark 3.14]{Br06}, also see \cite{H2012}.


We close this subsection with a subinvariance property.

\begin{lemma} \cite[Theorem 1]{hlpw}  \label{John subdomain}
Let $\Omega\subset \mathbb R^2$ be a simply connected domain, and let
$\varphi\colon \mathbb D\to \Omega $ be a 
conformal map. Suppose that $\Omega$ is $J$-John  with John center $\varphi(0).$
Then $\varphi$ maps every $J'$-John domain 
$G\subset\mathbb D$ with John center $z_0$ to a $c(J,\,J')$-John domain 
$G'=\varphi(G)$ with John center $\varphi(z_0)$. 
\end{lemma}

\subsection{Conformal geometry of the exterior domain}\label{sec:confexterior}

Let us fix our notation for this subsection.
Let  $\Omega\subset \mathbb R^2$ be a Jordan domain, and let a homeomorphism 
$\varphi\colon \mathbb R^2\setminus \mathbb D \to 
\mathbb R^2\setminus \Omega$ be conformal in $\mathbb R^2 \setminus
\overline{\mathbb D}$. 
For $z_1
\in \partial \Omega$, define
\begin{equation}\label{eq:annulusdef}
A(z_1,\,k):=\{x\in \mathbb R^2\setminus \overline{\mathbb D} 
\mid 2^{k-1}< |x-\varphi^{-1}(z_1)|\le 2^{k}\},  
\end{equation}
for 
$k\in \mathbb Z$. 
Furthermore, let $\Gamma(z_1)\subset \mathbb R^2\setminus \overline{\Omega}$ 
be the hyperbolic ray corresponding to $z_1,$ 
and set
\[
\Gamma_{k}:=\varphi(A(z_1,\,k))\cap \Gamma(z_1).
\]

We call $\varphi(\{x \in \mathbb R^2 \setminus \mathbb {D}\mid 2^{k-1} = |x-\varphi^{-1}(z_1)|\})$ the inner boundary of $\varphi(A(z_1,\,k))$  and $\varphi(\{x \in \mathbb R^2 \setminus \mathbb {D}\mid 2^{k} = |x-\varphi^{-1}(z_1)|\})$ the outer boundary of
$\varphi(A(z_1,\,k)).$
See Figure~\ref{fig:annuli}
for an illustration of our notation.

The following technical lemma is a version of a step in the proof of an 
analog of
Lemma~\ref{lma:map and geodesic} in \cite{BKR1998}.

\begin{lemma}\label{lengthtransfer}
With the notation introduced in the beginning of this subsection,
let $z_2\in \Gamma(z_1),$ and let
$\gamma\subset \mathbb R^2 \setminus \Omega$ 
be
any curve connecting $z_1$ and $z_2.$
Let $k \in \mathbb Z$ be such that $2^k\le |\varphi^{-1}(z_1)-\varphi^{-1}(z_2)|$ and let $\gamma_{k}$ 
be any subcurve of $\gamma$ in $\varphi(A(z_1,\,k))$
joining the inner boundary and the outer boundary of
$\varphi(A(z_1,\,k)).$
Then 
\begin{equation}\label{equ 1002}
\elle(\Gamma_{k})\sim \dist(\Gamma_{k},\,\partial 
\Omega)\sim \diam(\Gamma_{k})
\end{equation}
 and 
$$\elle(\gamma_{k})\gtrsim 
\elle(\Gamma_{k}).$$
Here all the constants are absolute and especially independent of $\Omega$ and the choice
of $\varphi,z_1,\gamma,z_2,k.$ 
\end{lemma}

\begin{figure}
 \centering
\includegraphics[width=0.7\textwidth]{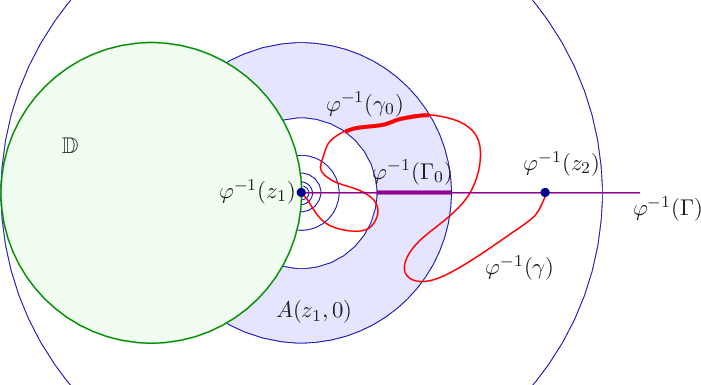}
    \caption{An illustration of the annular parts $\varphi^{-1}(\Gamma_k)$ and $\varphi^{-1}(\gamma_k)$, for $k = 0$, that are considered in Lemma \ref{lengthtransfer}.
    }
   \label{fig:annuli}
 \end{figure}

\begin{proof}
The fact that 
$\elle(\Gamma_k)\sim \dist(\Gamma_k,\,\partial \Omega)\sim \diam(\Gamma_k)$ immediately
follows  from Lemma~\ref{whitney preserving} and 
Lemma~\ref{bilipominaisuus}, since by definition 
$\varphi^{-1}(\Gamma_k)$ is contained in a 2-Whitney-type disk 
in $\mathbb R^2\setminus \overline{\mathbb D}$. 

Hence we only need to prove that $\elle(\gamma_k)\gtrsim \elle(\Gamma_k)$. 
Observe that, since $\gamma_k$ by definition joins the inner and outer 
boundaries of
$\varphi(A(z_1,\,k)),$ then
\begin{equation}\label{eqn1}
\elle(\varphi^{-1}(\gamma_k))\ge \diam(\varphi^{-1}(\gamma_k))\ge \diam(\varphi^{-1}(\Gamma_k))= \elle(\varphi^{-1}(\Gamma_k))=\dist(\varphi^{-1}(\Gamma_k),\,\partial \mathbb D).
\end{equation}
We next argue by case study. 

\noindent{{\bf Case 1}: $\dist(\varphi^{-1}(\gamma_k),\,\varphi^{-1}(\Gamma_k))< \frac 1 3 \dist(\varphi^{-1}(\Gamma_k),\,\partial \mathbb D)$.}
Write $r=\dist(\varphi^{-1}(\Gamma_k),\,\partial \mathbb D) = 2^{k-1}$ and pick $w\in \varphi^{-1}(\Gamma_k)$ so that $\dist(w,\varphi^{-1}(\gamma_k))<\frac r 3.$ Then
$B(w,\frac r 2)$ contains a subcurve $\alpha$ of $\varphi^{-1}(\gamma_k)$ of 
length at least $r/6$. 
Since $\varphi^{-1}(\Gamma_k)\cup \alpha$ is contained in
the 3-Whitney-type set $B(w,\frac{r}2)\cup \varphi^{-1}(\Gamma_k)$ and $6\elle(\alpha)\ge \elle(\varphi^{-1}(\Gamma_k)),$  
Lemma~\ref{bilipominaisuus} gives
\begin{align*}
\elle(\gamma_k) & \ge \elle(\varphi( \alpha))
 = \int_\alpha|\varphi'(z)|\,ds(z)
\ge \frac1C \elle(\alpha)|\varphi'(w)| \\
& \ge
\frac1C \elle(\varphi^{-1}(\Gamma_k))|\varphi'(w)|
\ge \frac1C \int_{\varphi^{-1}(\Gamma_k)}|\varphi'(z)|\,ds(z)
= \frac 1 C\elle(\Gamma_k) 
\end{align*}
for an absolute constant $C$.

\noindent{{\bf Case 2}: $\dist(\varphi^{-1}(\gamma_k),\,\varphi^{-1}(\Gamma_k))\ge \frac 1 3 \dist(\varphi^{-1}(\Gamma_k),\,\partial \mathbb D)$.}  
Let $\alpha'\subset \mathbb R^2\setminus \overline{\Omega}$ be a curve that joins $\gamma_k$ and $\Gamma_k.$
Since $\varphi^{-1}(\Gamma_k)$ is contained in a (2-Whitney-type) disk $B,$ 
$\varphi^{-1}(\alpha')$ contains a subcurve $\alpha\subset \frac 3 2 B$ of length at 
least $\frac 1 6  \dist(\varphi^{-1}(\Gamma_k),\,\partial \mathbb D).$ Since
$\frac 3 2 B$ is of 6-Whitney type, we may again apply Lemma~\ref{bilipominaisuus} to conclude
that
$$\elle(\alpha')\ge \elle(\varphi(\alpha))\ge \frac 1 C\elle(\Gamma_k)$$
with an absolute constant. Hence
\begin{equation}\label{eqn00}
C\dist_{\Omega}( \gamma_k ,\, \Gamma_k )\ge \diam(\Gamma_k). 
\end{equation}

Next, by \eqref{condition of lower bound} for the exterior of the unit disk, \eqref{eqn1}, the fact that $\dist(\varphi^{-1}(\gamma_k),\, \varphi^{-1}(\Gamma_k))\le 2 \dist(\varphi^{-1}(\Gamma_k),\, \partial \mathbb D)$
and the monotonicity of the capacity we obtain
$$c\log(\frac 3 2) \le {\rm Cap}(\varphi^{-1}(\overline{\gamma}_k),\,\varphi^{-1}(\overline{\Gamma}_k),\,\mathbb R^2\setminus \overline{\mathbb D})={\rm Cap}( \overline{\gamma}_k ,\, \overline{\Gamma}_k ,\,\mathbb R^2 \setminus \overline{\Omega})\le {\rm Cap}( \overline{\gamma}_k ,\, \overline{\Gamma}_k ,\,\mathbb R^2).$$
Hence by \eqref{eqn00} and Lemma~\ref{inner capacity} for $\mathbb R^2$ we conclude that
$$\elle(\gamma_k)\ge \diam(\gamma_k)\gtrsim \dist(\gamma_k,\,\Gamma_k)\gtrsim \diam(\Gamma_k)\sim \elle(\Gamma_k)$$
with absolute constants.
\end{proof}

We record another technical result,
see \cite[Corollary 4.18]{Ch1992} and \cite[Proof of Theorem 3.1, Page 645]{BKR1998}.

\begin{lemma} \label{samedistance}
Let 
Let $\Omega$ and $\varphi$ be as in the beginning of this subsection and let 
$0<\sigma\le 1$.
Let $z_0\in \R^2\setminus \overline {\mathbb D}$ and let $I$ be an arc of
$\partial \mathbb D$ with 
$$\elle(I)\ge \sigma (|z_0|-1)$$
and $$\dist(I,z_0)\le \frac {|z_0|-1} {\sigma}.$$ 
Then there is a rectifiable curve $\alpha\subset \R^2\setminus  {\mathbb D}$ 
joining $z_0$ to $I$ so that
$$\elle(\varphi(\alpha))\le C(\sigma)\dist(\varphi(z_0),\,\partial \Omega),$$
where $C(\sigma)$ is independent of $\varphi,z_0,\Omega.$
\end{lemma}
\begin{proof}
Notice first that it is enough to prove the claim for a subarc of $I$. Let $w \in I$ be such that $|w-z_0| \le 2\dist(I,z_0)$. Then, by taking a subarc of $I$ that contains $w$, we may assume that $\elle(I)\le |z_0|-1$,
that $I$ is closed and that $\dist(I,z_0) \le2\frac{|z_0|-1}{\sigma}$. Define $I_t=\{t\xi:\ \xi\in I\}$ for all $1 < t \le |z_0|$.
Then $I_t$ is also a continuum and $\elle(I_t) \ge \elle(I)$.
According to \eqref{condition of lower bound} (for $\mathbb R^2\setminus
\overline {\mathbb D}$), the assumptions on $I$, and the conformal invariance of capacity we have the 
estimate
\begin{equation}
0<\delta(\sigma)\leq {\rm Cap}(I_{|z_0|},\, I_t ,\,\mathbb R^2\setminus 
\overline{\mathbb D})={\rm Cap}(\varphi(I_{|z_0|}) ,\, \varphi(I_t) ,\, \mathbb R^2\setminus \overline{\Omega})
\end{equation}
for all $1 < t < |z_0|$.
Then, by Lemma~\ref{inner capacity}, we conclude that 
$$\dist_{\mathbb R^2\setminus \overline{\Omega}}(\varphi(I_{|z_0|}),\,\varphi(I_t))\le C(\sigma)\diam_{\mathbb R^2\setminus \overline{\Omega}}(\varphi(I_{|z_0|})) \le  C(\sigma) \elle(\varphi(I_{|z_0|})).$$
Hence, we can connect $I_{|z_0|}$ to $I_t$ with a curve $\beta_t$ for which 
$\elle(\varphi(\beta_t)) \le C(\sigma) \elle(\varphi(I_{|z_0|}))$.


By Lemma \ref{arsela}, there exists a sequence $t_i \searrow 1$ so that $\varphi\circ \beta_{t_i}$ converges to a rectifiable curve $\hat \beta \subset \mathbb R^2 \setminus \Omega$ joining $\varphi(I)$ to $\varphi(I_{|z_0|})$ with
\[
 \elle(\hat \beta) \le \liminf_{i \to \infty} \elle(\varphi(\beta_{t_i})) \le 
 C(\sigma) \elle(\varphi(I_{|z_0|})).
\]

Next, take $J$ to be a shortest closed subarc of $\partial B(0,|z_0|)$ containing both $I_{|z_0|}$ and $z_0$. Then, 
\begin{align*}
 \elle(J) & \le \elle(I_{|z_0|}) + \pi\dist(z_0,I_{|z_0|}) \le |z_0|\elle(I) + \pi|z_0-w| + \pi\dist(w,I_{|z_0|})\\
 & \le |z_0|\min((|z_0|-1),2\pi) + 2\pi\frac{|z_0|-1}{\sigma} + \pi(|z_0|-1)\\
 & \le C(\sigma)(|z_0|-1) = C(\sigma)\dist(z_0,\mathbb D).
\end{align*}
Since $J$ is contained in a $\lambda(\sigma)$-Whitney-type set $B = J\cup B(z_0,\frac{|z_0|-1}{2})$, by Lemma \ref{bilipominaisuus} and Lemma \ref{whitney preserving}
\[
\elle(\varphi(I_{|z_0|})) \le \elle(\varphi(J)) \le C(\sigma) \diam(\varphi(B)) \le C(\sigma) \dist(\varphi(B),\partial\Omega).
\]
Now, take $z \in I_{|z_0|}\cap \varphi^{-1}(\hat \beta)$ and define $\alpha = \varphi^{-1}(\hat \beta) \ast J[z,z_0]$. Then $\alpha$ connects $I$ to $z_0$ in $\mathbb R^2 \setminus \mathbb D$ and 
\begin{align*}
\elle(\varphi(\alpha)) &\le \elle(\hat \beta) + \elle(\varphi(J)) \le C(\sigma) \elle(\varphi(I_{|z_0|})) + \elle(\varphi(J))\\
& \le  C(\sigma) \dist(\varphi(B),\partial\Omega) \le C(\sigma)\dist(\varphi(z_0),\partial\Omega).
\end{align*}

\end{proof}

%
%
%

\section{Proof of necessity}\label{sec:nec}
In this section we prove that a bounded simply connected planar 
$W^{1,p}$-extension domain with $1<p<2$
necessarily has the property that
any two points $z_1,z_2 \in \R^2 \setminus \Omega$ can be connected with a curve 
$\gamma \subset \R^2 \setminus \Omega$ satisfying
 \begin{equation*}
  \int_{\gamma}\dist(z,\partial \Omega)^{1-p}\,\d s(z) 
 \le  C(\|E\|,p) |z_1-z_2|^{2-p}.
\end{equation*}

We will first consider the case where $\Omega$ is additionally assumed to be 
Jordan. Under this assumption, we usually denote the complementary domain of $\Omega$ by 
$\wz \Omega.$ 

\begin{theorem}\label{neceJordan} Let $1<p<2$ and let $\Omega$ be a Jordan 
domain. Suppose that 
there exists an extension operator
 $E \colon W^{1,p}(\Omega) \to W^{1,p}(\R^2).$ 
Then, given $z_1,\,z_2\in \wz \Omega\cup \partial \Omega$,  
there is a curve $\gamma\subset  \wz \Omega\cup \partial \Omega$ so that
 \begin{equation}\label{eq:extcharcomplre}
  \int_{\gamma}\dist(z,\partial \Omega)^{1-p}\,\d s(z)
 \le  C(\|E\|,p) |z_1-z_2|^{2-p},
\end{equation}
where 
$C(\|E\|,p)$ depends only on $p$ and the norm of the extension operator.
\end{theorem}
 
   After this, based on inner uniformity (see Definition~\ref{inneruniform} 
below), we prove that, 
if $\Omega$ is a bounded simply connected $W^{1,p}$-extension domain, then,
for $n\ge 2,$ the 
Jordan domains $\Omega_n=\varphi(B(0,\,1-\frac 1 n))$ are also  
$W^{1,p}$-extension domains with 
extension operator norms only 
depending on $p$ and the norm of the extension operator for 
$\Omega.$ 
Here $\varphi\colon \mathbb D\to \Omega$ is a suitable conformal map. 
Finally by a limiting
argument we obtain the result in the general case. 

We remark that, actually, when $z_1,\,z_2\in \wz \Omega$ one can require that the curve $\gamma$ in Theorem~\ref{neceJordan} is contained in $\wz \Omega$. For this see Remark~\ref{remark 1}.

\subsection{Necessity in the Jordan case}

In this section we  prove Theorem~\ref{neceJordan}. Recall that the
existence of our extension operator guarantees that $\Omega$ is a John
domain with a constant $J$ only depending on $p$ and the norm of $E;$ see e.g.\ \cite[Theorem 6.4] {kos1990}, \cite[Theorem 3.4]{gore1990}, \cite[Theorem 4.5]{nava1991} and 
references therein. 
In what follows, $J$ refers to this constant.
Because of technical issues, we first 
consider the case  
$z_1,\,z_2 \in \partial\wz \Omega=\partial \Omega$ with $z_2\neq z_1.$ 

Since $\Omega$ is Jordan,  $\partial \Omega\setminus \{z_1,\,z_2\}$ consists of two open 
arcs $P_1$ and $P_2$. Without loss of generality we assume that $\diam(P_1)\le \diam(P_2).$
For the following four lemmas let $\Omega$, $z_1$, $z_2$, $P_1$ and $P_2$ be fixed.

We need the following general lower bound on the Sobolev norm.

\begin{lemma}\label{lma:lowerbound}
Let $Q$ be a square with sides parallel to the coordinate axes and fix $1\le p<2.$ 
Let $u\in W^{1,1}(Q)$ be absolutely continuous on almost all lines parallel to the 
coordinate axes. Write
$$A_0=\{x\in Q \mid u(x)\le 0\} \quad \text{ and} \quad A_1=\{x\in Q \mid u(x) \ge 1\}.$$
Suppose further that 
$$\max\{\mathscr H^1(\pi_1(A_0)),\, \mathscr H^1(\pi_2(A_0))\}\ge \delta \ell(Q)$$
and
$$\max\{\mathscr H^1(\pi_1(A_1)),\, \mathscr H^1(\pi_2(A_1))\}\ge \delta \ell(Q)$$
for some $\delta>0$, where  $\mathscr H^1$ is the 1-dimensional Hausdorff measure, and $\pi_i$ stands for the projection to the $x_i$-axis for each $i=1,\,2$. Then 
$$ \ell(Q)^{2-p}\le C(\delta,\,p)\int_{Q}|\nabla u|^p \, dx.$$ 
\end{lemma}

\begin{proof}
We may assume that
$$\int_{Q}|\nabla u|^p \, dx <\infty; $$ otherwise the claim is trivial.

Suppose first that $\mathscr H^1(\pi_1(A_0))\ge \delta \ell(Q)$ and $\mathscr H^1(\pi_1(A_1))\ge \delta \ell(Q)$. If for $\mathscr H^1$-almost every $x_1\in\pi_1(A_0)$, there exists some $x_2\in \pi_2(Q)$ such that $u(x_1,\,x_2)\ge \frac 1 3$, then by our absolute continuity assumption and the H\"older inequality,  
$$\frac 1 3 \le \int_{\pi_2(Q)} |\nabla u(x_1,\,t)|\,dt\le \ell(Q)^{\frac {p-1} p}\left(\int_{\pi_2(Q)}|\nabla u(x_1,\,t)|^p\,dt\right)^{\frac 1 p}$$
for $\mathscr H^1-$almost every $x_1\in\pi_1(A_0)$, and our claim follows by Fubini's theorem:
\[
 \int_Q|\nabla u|^p \, dx \ge \int_{\pi_1(A_0)}\int_{\pi_2(Q)}
 |\nabla u(x_1,t)|^p\,dt\,dx_1\ge \mathcal H^1(\pi_1(A_0))\frac{1}{3^p}\ell(Q)^{1-p}\ge \frac{\delta}{3^p}\ell(Q)^{2-p}.
\]
Similarly, the claim holds if for $\mathscr H^1-$almost every $x_1\in\pi_1(A_1)$, there exists $x_2\in \pi_2(Q)$ such that $u(x_1,\,x_2)\le \frac 2 3$. 
If both of the above two conditions fail, we find $x_1\in\pi_1(A_0)$ and $\hat x_1\in\pi_1(A_1)$ such that for all $x_2\in\pi_2(Q)$, $u(x_1,\,x_2)\le \frac 1 3$ and $u(\hat x_1,\,x_2)\ge \frac 2 3$. 
Then by absolute continuity and H\"older's inequality, for $\mathscr H^1$-almost every $x_2 \in \pi_2(Q)$, we have
$$\frac 1 3 \le u(\hat x_1,x_2) - u(x_1,x_2) \le \int_{\pi_1(Q)} |\nabla u(t,\,x_2)|\,dt\le \ell(Q)^{\frac {p-1} p}\left(\int_{\pi_1(Q)}|\nabla u(t,\,x_2)|^p\,dt\right)^{\frac 1 p},$$
and we again conclude by Fubini's theorem that
\[
 \int_Q|\nabla u|^p \, dx \ge \int_{\pi_2(Q)}\int_{\pi_1(Q)}
 |\nabla u(t,x_2)|^p\,dt\,dx_2\ge \mathcal H^1(\pi_2(Q))\frac{1}{3^p}\ell(Q)^{1-p}\ge \frac{1}{3^p}\ell(Q)^{2-p}.
\]


If $\mathscr H^1(\pi_2(A_0))\ge \delta \ell(Q)$ and $\mathscr H^1(\pi_2(A_1))\ge \delta \ell(Q)$, the argument for the previous case gives the asserted estimate after switching the roles of the coordinates $x_1$ and $x_2$.
We are left with the cases where 
$$\mathscr H^1(\pi_1(A_0))\ge \delta \ell(Q)\quad \text{ and} \quad\mathscr H^1(\pi_2(A_1))\ge \delta \ell(Q)$$
and 
$$\mathscr H^1(\pi_2(A_0))\ge \delta \ell(Q) \quad \text{ and} \quad \mathscr H^1(\pi_1(A_1))\ge \delta \ell(Q). $$
By symmetry, it suffices to consider the first one. As above,
if for $\mathscr H^1$-almost every $x_1\in\pi_1(A_0)$, there exists some $x_2\in \pi_2(Q)$ such that $u(x_1,\,x_2)\ge \frac 1 3$, then we get
\[
 \int_Q|\nabla u|^p \, dx \ge \frac{\delta}{3^p}\ell(Q)^{2-p}.
\]
Similarly, if for $\mathscr H^1$-almost every $x_2\in\pi_2(A_1)$, there exists some $x_1\in \pi_2(Q)$ such that $u(x_1,\,x_2)\le \frac 1 3$, then we get
\[
 \int_Q|\nabla u|^p \, dx \ge \frac{\delta}{3^p}\ell(Q)^{2-p}.
\]

Thus, the only case remaining is the one in which there exist $x_1\in\pi_1(A_0)$ and $x_2\in\pi_2(A_1)$ such that for all $t\in \pi_2(Q)$ and $s\in \pi_1(Q)$, $u(x_1,\,t)\le \frac 1 3$ and $u(s,\,x_2)\ge \frac 2 3$, and so that $u$ is absolutely  continuous along these two line segments. This is impossible as these segments intersect.
\end{proof}

We continue with the existence of suitable test functions. Recall that the  curves $P_1$ and $P_2$ are open.

\begin{lemma}\label{lma:testfunction} Let $c_1\ge 1$ and $1 < p < 2$.
With the above notation, there exists a function $\Phi \in W^{1,\,p}({\Omega})$
such that for any $0<\epsilon<\frac 1 9$, we have $\Phi\ge 1-\epsilon$ in 
some neighborhood of
 $P_1\cap B(z_1,c_1|z_2-z_1|)$, $\Phi\le \epsilon$ in some neighborhood of 
$P_2 \cap B(z_1,c_1|z_2-z_1|),$  and
 \[
  \|\nabla \Phi\|^p_{L^{p}({\Omega})} \leq C(p,c_1,J) |z_1-z_2|^{2-p}.
 \]
Here the neighborhoods are defined with respect to the topology of $\overline{\Omega}$. 
\end{lemma}
\begin{proof} 
Let $x_0 \in \Omega$ be a distinguished point as in Definition
\ref{def:John}. 
Denote by $\gamma_1$ the hyperbolic segment from $x_0$ to $z_1$ and by $\gamma_2$ the hyperbolic segment from $x_0$ to $z_2$. By Lemma~\ref{hsjohn} (4), the curves $\gamma_1$ and $\gamma_2$ are John curves. We define $\gamma_0=\gamma_1\cup\gamma_2$. The existence of John
curves is actually only guaranteed by the definition for points inside the domain, but
the general case follows easily from this; see Remark~\ref{muistutus}. 
Let $\varphi\colon\overline{\mathbb D}\to\overline{\Omega}$ be a homeomorphism which is conformal inside and satisfies $\varphi(0)=x_0$. Then it is clear that the preimages of $\gamma_1$ and $\gamma_2$ under $\varphi$ are radial line segments, and $\varphi^{-1}(P_1\cup \gamma_0)$ is a Jordan curve. Hence $P_1\cup \gamma_0$ is also Jordan as $\varphi$ is a homeomorphism. It follows that $P_1\cup \gamma_0$ bounds a Jordan subdomain $\Omega_1\subset\Omega.$


Define a function $\phi\colon \Omega \to \R$ by setting
\[
\phi(x)=\max\left\{\inf_{\gamma(x,\,P_2)} \int_{\gamma(x,\,P_2)} \frac 1 {|\hat z-z_1|} \, \d s(\hat z),\,\inf_{\gamma(x,\,P_2)} \int_{\gamma(x,\,P_2)} \frac 1 {|\hat z-z_2|} \, \d s(\hat z)\right\}, 
\]
for $x\in \Omega$, where the infima are taken over all the rectifiable curves
$\gamma(x,\,P_2)\subset \Omega$ joining $x$ to   $P_2$.

Since $\Omega$ is a Jordan domain, $\gamma_0$, $P_1$ and $P_2$ are pairwise disjoint. 
By the John condition we have
\[
\dist(w,\,\partial \Omega) \ge J |w-z_1| ,
\]
for every $w\in \gamma_1$.
Therefore, for $w\in \gamma_1$, we get
\[
\phi(w) \ge \inf_{\gamma(w,\,P_2)}\int_{\gamma(w,\,P_2)} \frac 1 {|\hat z-z_1|} \, \d s(\hat z)
  \ge  \frac {\dist(w,\, \partial \Omega)} {\dist(w,\, \partial \Omega)+|w-z_1|} \ge \frac {J} {J+1} =: c_0,
\]
where we have used the fact that $\gamma(w,\,P_2)$ necessarily
exits $B(w,{\dist(w,\partial \Omega)})$
and that for points $\hat z \in B(w,{\dist(w,\partial \Omega)})$ we have, by the triangle inequality,
\[
 |\hat z-z_1| \le |\hat z - w| + |w-z_1| \le  \dist(w,\,\partial \Omega)+|w-z_1|.
\]
The same estimate follows similarly for $w\in \gamma_2$.  Hence for any point $w\in \Omega_1$, we have $\phi(w)\ge c_0$ as $\Omega_1$ is Jordan and $P_2$ is outside $\Omega_1$; any curve $\gamma(w,\,P_2)\subset \Omega$ must cross $\gamma_0$ by the Jordan curve theorem.  See Figure \ref{fig:testf}. 

Fix $0<\ez<\frac 1 9.$ We claim that we have $\phi\le\ez$ in some neighborhood of $P_2$.  
Indeed  for any $x\in P_2$ there is a radius $R_x>0$ such that $B(x,\,2R_x)\cap P_1=\emptyset$.  
Let $0<r_x\le R_x$. 
Then for any $y\in B(x,\, r_x)\cap \Omega$ there is a point $z\in P_2\cup\{z_1,\,z_2\}$ such that 
$$|y-z|=\dist(y,\,P_2)=\dist(y,\,\partial \Omega)\le r_x$$
while
$$ \dist(y,\,P_1)\ge R_x$$
via the triangle inequality.  By choosing $r_x$ sufficiently small,  we conclude, via letting $\gamma(w,\,P_2)$ be the line segment joining $y$ and $z$  in the definition of $\phi$,  that
$$\phi(y)\le r_x R_x^{-1} \le \ez.$$
 Hence by taking the union of $B(x,\,r_x)\cap \overline{\Omega}$ over $x\in P_2$ we obtain a neighborhood of $P_2$ in which $\phi\le \ez$.

Recall that $c_1\ge 1$. We define a cut-off function by setting
\[
 \alpha(z)=
 \begin{cases}
   1, & \text{if } |z-z_1|<{c_1|z_1-z_2|} \\
   \log_2 \frac{2c_1|z_1-z_2|}{|z-z_1|}, &\text{if } c_1|z_1-z_2|\le|z-z_1|\le 2 c_1|z_1-z_2| \\
   0, & \text{otherwise}
 \end{cases}
\]
for $z\in \Omega$. 
Using this cut-off function 
we define
\[
\Phi(z)= \alpha(z)\min\left\{\frac 1 {c_0} \phi(z),\,1\right\} 
\]
when $z\in \Omega$. We also define $\Phi(x)=1$ for $x\in P_1\cap B(z_1,\, c_1|z_1-z_2|)$, and $\Phi(x)=0$ when $x\in P_2$. Then by the properties of $\phi$ we know that, for 
any $0<\epsilon<\frac 1 9$, $\Phi\ge 1-\epsilon$ in some neighborhood of
the set  $P_1\cap B(z_1,\, c_1|z_1-z_2|)$, and $\Phi\le \epsilon$ in some 
neighborhood of 
$P_2 \cap B(z_1,\, c_1|z_1-z_2|)$.

 We claim that $\phi$ is locally Lipschitz in $\Omega$ with 
\begin{equation}\label{eq:phiestimate}
 |\nabla \phi(z)|\le \frac 3 2 \max\left\{{|z-z_1|}^{-1},\,{|z-z_2|}^{-1} \right\}
\end{equation}
for almost every $z \in \Omega$. 
Indeed, for any $y\in B(z,\, 3^{-1} \dist(z,\,\partial \Omega))$, we have, by the definition of $\phi$ and the fact that $z_1,\,z_2\in \partial \Omega$,
\begin{align*}
|\phi(y)-\phi(z)|&\le \max\left\{\int_{[y,\,z]} |\hat z-z_1|^{-1}\, \d s(\hat z),\, \int_{[y,\,z]} |\hat z-z_2|^{-1}\, \d s(\hat z)\right\}\\
&\le \frac 3 2 \max\left\{{|z-z_1|}^{-1},\,{|z-z_2|}^{-1} \right\} |y-z|,
\end{align*}
where $[y,\,z]$ is the line segment joining $y$ and $z$. Thus our claim follows. Furthermore, by applying the Leibniz rule we obtain
\begin{align*}
&\|\nabla \Phi\|^p_{L^{p}({\Omega})}
\lesssim \|\nabla \alpha\|^p_{L^{p}({\Omega})}+\|\nabla \phi\|^p_{L^{p}({\Omega\cap B(z_1,\,2c_1|z_1-z_2|)})}\\
&\lesssim \int_{B(z_1,\,2c_1|z_1-z_2|)\setminus B(z_1,\,|z_1-z_2|)} {|\hat z-z_1|}^{-p} \, d\hat z+ \int_{B(z_1,\,2c_1|z_1-z_2|)} {|\hat z-z_1|}^{-p}+{|\hat z-z_2|}^{-p}   \, d\hat z \\
&\leq C(p,c_1, J)  |z_1-z_2|^{2-p},
\end{align*}
by calculating in polar coordinates with $1<p<2$. Thus $\Phi\in W^{1,\,p}(\Omega)$ with the desired properties since $\|\Phi\|_{L^{\infty}(\Omega)}\le 1$ and $\Omega$ is bounded. 
\end{proof}

 \begin{figure}
 \centering
 \includegraphics[width=0.6\textwidth]{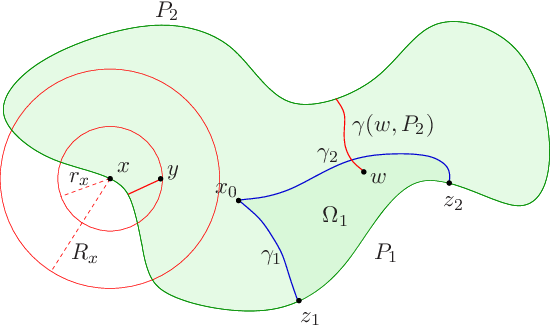}
     \caption{The function $\phi$ is seen to have large value in $\Omega_1$ by observing that any curve $\gamma(w,P_2)$ connecting a point $w \in \Omega_1$ to $P_2$ in $\Omega$ must
 intersect $\gamma_0 = \gamma_1\cup\gamma_2$. In order to see that $\phi$ has small value near $P_2$ 
     one observes that $\phi$ near $x \in P_2$ can be estimated by integrating $\frac{1}{R_x}$ along a curve with length at most $r_x \le \epsilon R_x$.
     }
   \label{fig:testf}
 \end{figure}

Let $\wz \varphi: \R^2\setminus \overline {\mathbb D}\to \R^2\setminus\overline {\Omega}$ be a conformal
map. Since $\Omega$ is Jordan, $\wz \varphi$ extends homeomorphically up to the boundary by the Carath\'eodory-Osgood theorem.
We refer to this extension also by $\wz \varphi.$ 
For our fixed 
$z_1,\,z_2 \in \partial \wz \Omega,$ let $\Gamma_k$ be the hyperbolic ray starting at 
$\wz \varphi^{-1}(z_k),$ where $k=1,2.$ Pick $y_k\in \Gamma_k$ with 
$$|\wz \varphi^{-1}(z_k)-y_k|=|\wz \varphi^{-1}(z_2)-\wz \varphi^{-1}(z_1)|,$$
let $\alpha_c$ be a shorter one of the two
circular arcs from $y_1$ to $y_2.$
Define
\begin{equation}\label{eq:alphagammadef}
 \alpha = [\wz \varphi^{-1}(z_1),y_1]\ast \alpha_c \ast[y_2, \wz \varphi^{-1}(z_2)] \qquad \text{and} \qquad \gamma=\wz \varphi(\alpha).
\end{equation}
See Figure~\ref{fig:gamma}.
We will establish the curve condition \eqref{eq:extcharcomplre} for $\gamma.$
The reason for using $\gamma$ instead of the corresponding hyperbolic segment 
is partially that this is technically easier.

Let $\wz W$ be a Whitney decomposition of $\wz \Omega$  given by Lemma \ref{lma:whitney} and set
$$\wz W_{\gamma}=\left\{\wz Q_{i}\in \wz W \mid \wz Q_{i} \cap \gamma \neq \emptyset \right\}. $$
We index the squares in $W_{\gamma}$ according to side length:
$\wz Q_{i1},\cdots,\wz Q_{in_i}$ are those with side length $2^i$ when $i\in \mathbb Z,$ if there are
such squares. Notice that since $\partial\wz \Omega$ is bounded, each $n_i$ is necessarily finite. 

We start with a simple observation on the Whitney squares that intersect the circular part of the curve $\alpha$.
\begin{lemma}\label{lma:finitelymanyoncircle}
 The number of squares $\wz Q_{ij} \in\wz W$ for which $\varphi^{-1}(\wz Q_{ij}) \cap \alpha[y_1,y_2] \ne \emptyset$ is bounded from above by a universal constant.
\end{lemma}
\begin{proof}
 By Lemma~\ref{whitney preserving} we have that $\wz \varphi^{-1}(\wz Q_{i,j})$ is a $\lambda$-Whitney-type set with a universal constant $\lambda$. If $\varphi^{-1}(\wz Q_{ij}) \cap \alpha[y_1,y_2] \ne \emptyset$, we have 
 \begin{align*}
 |\wz\varphi^{-1}(z_2) - \wz\varphi^{-1}(z_1)| & =
  \dist(\partial \mathbb D, \alpha[y_1,y_2]) \\
  & \le 
  \dist(\partial \mathbb D,\wz\varphi^{-1}(\wz Q_{ij}))
  + \diam(\wz\varphi^{-1}(\wz Q_{ij}))\\
  &\le (\lambda+1)\diam(\wz\varphi^{-1}(\wz Q_{ij})),
 \end{align*}
 and so $\varphi^{-1}(\wz Q_{ij})$ contains a disk of radius 
 $\frac{1}{\lambda(\lambda+1)}|\wz\varphi^{-1}(z_2) - \wz\varphi^{-1}(z_1)|$. Since for different $\wz Q_{ij}$ these disks are disjoint, and since
 \[
  \dist(\wz\varphi^{-1}(z_1),\wz\varphi^{-1}(\wz Q_{ij}))
  \le \diam(\alpha) \le 3|\wz\varphi^{-1}(z_2) - \wz\varphi^{-1}(z_1)|,
 \]
 the claim follows.
\end{proof}


\begin{figure} 
 \centering
 \includegraphics[width=0.85\textwidth]{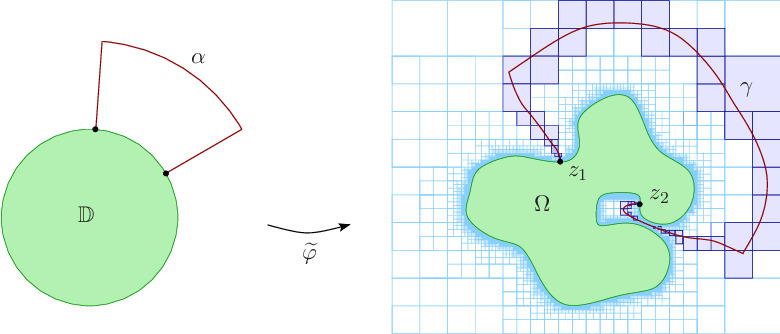}
 \caption{The curve $\gamma$   is obtained as 
the image of
          the curve $\alpha$ under the conformal map 
          $\widetilde\varphi \colon \R^2 \setminus \overline {\D} \to \R^2 
\setminus \overline {\Omega}$.
          In the illustration the Whitney squares in $\wz W_\gamma$ are highlighted.}
 \label{fig:gamma}
\end{figure}

\begin{lemma}\label{turning}
For the curve $\gamma$ defined in \eqref{eq:alphagammadef} and each Whitney square $\wz Q\in \wz W_{\gamma}$, we have
$$\wz Q\subset B(z_1, C|z_1-z_2|),$$
where $C=C(J)$ is independent of $z_1,z_2,\wz \varphi.$
\end{lemma}

\begin{proof}
Since $\Omega$ is John, by Lemma~\ref{btchar} the set $\mathbb R^2 \setminus \Omega$ is of $C(J)$-bounded turning, where $C(J)$ depends only on $J$. Thus, there is a (closed) curve
$\beta\subset \mathbb R^2\setminus \Omega$ that joins $z_1,z_2$ and so that $\beta\subset \overline B(z_1, C(J)
|z_1-z_2|)$. 

Now, if $\wz Q\cap \beta\neq \emptyset$, we have
\[
 \wz Q \subset B(z_1,C(J)|z_1-z_2| + \diam(\wz Q)) \subset
 B(z_1,(C(J)+\sqrt{2}C(J))|z_1-z_2|),
\]
as $z_1\in \partial \Omega$. 

Suppose then that $\wz Q\cap \beta = \emptyset$.
We have
\begin{equation}\label{inequ 36}
\diam(\wz \varphi^{-1}(\beta))\ge |\wz \varphi^{-1}(z_1)-\wz \varphi^{-1}(z_2)|
\end{equation} 
since $z_1,z_2\in \wz \varphi^{-1}(\beta)).$ 
Next, $\wz \varphi^{-1}(\wz Q)$ is a $\lambda-$Whitney-type set by 
Lemma~\ref{whitney preserving} with a univeral constant $\lambda$
and $\wz \varphi^{-1}(\wz Q)\cap \az\neq \emptyset.$ 
Hence the definition of  $\az$ ($\alpha$ is an inner uniform curve for the exterior domain of the unit disk) gives
\begin{align*}
\dist(\wz \varphi^{-1}(\wz Q),\,\wz \varphi^{-1}(\beta))&\le \min\{\dist(\wz \varphi^{-1}(\wz Q),\,\wz \varphi^{-1}(z_1)),\,\dist(\wz \varphi^{-1}(\wz Q),\,\wz \varphi^{-1}(z_2))\}\\
&\le C\diam(\wz \varphi^{-1}(\wz Q))\le C |\wz \varphi^{-1}(z_1)-\wz \varphi^{-1}(z_2)|.
\end{align*}
This together with \eqref{inequ 36} shows that
$$C\min\{\diam(\wz \varphi^{-1}(\wz Q)),\,\diam(\wz\varphi^{-1}(\beta))\}\ge \dist(\wz \varphi^{-1}(\wz Q),\,\varphi^{-1}(\beta)). $$
Then the version of \eqref{condition of lower bound}  for $\mathbb R^2\setminus \overline{\mathbb D}$ and conformal invariance of capacity give
$$0 < \delta(C) \le {\rm Cap}(\wz \varphi^{-1}(\wz Q),\, \wz \varphi^{-1}(\bz) ,\,\mathbb R^2\setminus \overline{\mathbb D})={\rm Cap}(\wz Q ,\, \beta ,\, \wz \Omega)\le {\rm Cap}(\wz Q ,\, \beta ,\, \mathbb R^2), $$
where in the last inequality we used the monotonicity of capacity. 

Hence Lemma~\ref{inner capacity} shows that $\dist(\wz Q,\,\beta)\le C \diam(\beta),$ and since $z_1\in \beta$ and $$\diam(\beta)\le C(J)|z_1-z_2|,$$
 we conclude that
 \begin{align*}
  \wz Q &\subset B(z_1, \dist(z_1,\wz Q) + \diam(\wz Q))
  \subset B(z_1, (1+\sqrt{2})\dist(z_1,\wz Q))\\
  &\subset B(z_1, (1+\sqrt{2})(\diam(\beta) + \dist(\beta,\wz Q)))\\
  & \subset B(z_1, C\diam(\beta)) \subset B(z_1, C(J)|z_1-z_2|).\qedhere
 \end{align*}
 \end{proof}

We apply the preceding four lemmas to prove the following estimate for 
$\wz W_{\gamma}.$ Recall that $n_i$ stands for the number (if any) of $\wz Q_{ij}
\in \wz  W_{\gamma}$
of side length $2^{i}$ and $\|E\|$ stands for the norm of the homogeneous extension operator.

\begin{lemma} \label{sarja}
We have 
$$\sum_i n_i 2^{i(2-p)} \le C(\|E\|,\,p) |z_1-z_2|^{2-p}. $$
\end{lemma}

\begin{proof}
We claim that there exists a constant $c_0$  such that, 
for every $\wz Q_{ij}\in  \wz W_{\gamma}$, 
\begin{equation}\label{eqn05}
c_0 \wz Q_{ij}\cap P_1\neq \emptyset \neq c_0 \wz Q_{ij}\cap P_2.
\end{equation}
Towards this, suppose first that $\wz \varphi^{-1}(Q_{ij})\cap 
[\wz \varphi^{-1}(z_k),y_k] \neq \emptyset$ for $k=1$ or for $k=2,$
where the points $y_k \in \alpha$ are from the definition of $\alpha$ and $\gamma.$
Pick $z_0\in \wz \varphi^{-1}(\wz Q_{ij})\cap 
[\wz \varphi^{-1}(z_k),y_k].$
Then $\elle(\wz \varphi^{-1}(P_k))\ge |z_0|-1$ and $\dist(\varphi^{-1}(P_k),z_0)\le |z_0-\wz\varphi^{-1}(z_k)| \le |z_0|-1$ for $k=1,2.$ Hence Lemma~\ref{samedistance}, applied to both $\wz \varphi^{-1}(P_1)$ and $\wz \varphi^{-1}(P_2)$, gives a curve $\az'$ connecting $P_1$ and $P_2$ and passing 
through $z_0$ such that 
$$\elle(\wz \varphi(\az'))\le C_0\dist(\wz \varphi(z_0),\,\partial \wz \Omega).$$ 
Since $\wz Q_{ij}$ is a Whitney square, it follows that 
$\wz \varphi(\az')\subset c_0'\wz Q_{ij},$
where $c_0'=c_0'(C_0)\ge 1,$
and we conclude \eqref{eqn05} for our $\wz Q_{ij}$. 

We are left with the case where $\wz Q_{ij}$ only intersects the image of the 
circular  part of $\az.$ By Lemma~\ref{lma:finitelymanyoncircle} there are only uniformly finitely 
many such $\wz Q_{ij}$ and so there exists a constant $c''_0$ such that
$$\ell(\wz Q')\le c''_0\ell(\wz Q_{ij}) \quad \text{ and }\quad \dist(\wz Q_{ij},\,\wz Q')\le c''_0 \ell(\wz Q_{ij})$$
 for each such $\wz Q_{ij}$ and some $\wz Q'$ from our first case. 
By setting $c_0=c_0'c_0''$ we  obtain \eqref{eqn05} also in this case. 

Next,  Lemma~\ref{turning} allows us to infer that 
\begin{equation}\label{sisaltyy}
2c_0 \wz Q_{ij}\subset B(z_1,
2c_0C|z_1-z_2|)
\end{equation}
 for each $\wz Q_{ij}\in \wz W_{\gamma}.$ Here $C=C(J)=C(p,\,\|E\|).$

Let $\Phi$ be defined as in Lemma \ref{lma:testfunction} for the choice
$c_1=2c_0C,$ where $c_0C$ is from \eqref{sisaltyy}.
Let $s=\frac {1+p} 2$. Then $1<s<p$. 

Since $\Omega$ is a
$W^{1,\,p}-$extension domain, we have $E\Phi \in W^{1,\,p}(\mathbb R^2)$, 
where $E$ is the corresponding extension operator. 
Therefore, by denoting  the Hardy-Littlewood maximal operator by $M,$ and by using the boundedness of $M\colon L^{p/s}(\mathbb R^2)\to L^{p/s}(\mathbb R^2)$ applied to the function $|\nabla E\Phi|^s$, we obtain
\begin{align}\label{haluttuey}
&\sum_i \sum_{j=1}^{n_i}  |\wz Q_{ij}|^{1-\frac p s} \left(\int_{2c_0 \wz Q_{ij}} |\nabla E\Phi(x)|^s \, dx\right)^{\frac p s}\\
&\le C(c_0,\,p)   \sum_i \sum_{j=1}^{n_i} |\wz Q_{ij}| \left(\bint_{2c_0 \wz Q_{ij}} 
|\nabla E\Phi(x)|^s \, dx\right)^{\frac p s}\\
&\le C(c_0,\,p)  \sum_i  \sum_{j=1}^{n_i} \int_{\wz Q_{ij}} |M(|\nabla E\Phi|^s)(x)|^{\frac p s} \, dx\\
&\le C( c_0,\,p)   \int_{\wz \Omega} |M(|\nabla E\Phi|^s)(x)|^{\frac p s} \, dx \\
&\le C(c_0,\,p)   \int_{\mathbb R^2} |\nabla E\Phi(x)|^p \, dx\\
&\le C(c_0,\,\|E\|,\,p)\int_{\Omega}|\nabla \Phi(x)|^p \, dx
 \le C(c_0,\,c_1,\,\|E\|,\,p) |z_1-z_2|^{2-p}.
\end{align}
Notice that for any $\wz Q_{ij}\in \wz W_{\gamma}$, 
$$\diam(\gamma_1) \sim_{c_0}  \ell(\wz Q_{ij}) \sim_{c_0} \diam(\gamma_2) $$
for subcurves $\gamma_1\subset  2c_0 \wz Q_{ij}$ of $P_1$
and $\gamma_2\subset  2c_0 \wz Q_{ij}$ of $P_2$
 by \eqref{eqn05}, \eqref{sisaltyy} and the definition of $c_1$. 
Then, by Lemma \ref{lma:lowerbound} (with $p=s$ there) applied to a representative
of $E\Phi$ that is absolutely continuous on almost every
line segment parallel to the coordinate axes, relying on the values of $\Phi$ on $P_1,P_2$ from
Lemma 
\ref{lma:testfunction}, we have
\[
\ell(2c_0\wz Q_{i,j})^{2\frac{p}{s}-p} \le C(c_0,p)\left(\int_{2 c_0 \wz Q_{ij}} |\nabla E\Phi(x)|^s \, dx\right)^{\frac p s},
\]
which, by summing over all the squares $\wz Q_{ij}$, gives
the estimate
\begin{eqnarray*}
C(c_0,p)\sum_i \sum_{j=1}^{n_i}  |2c_0 \wz Q_{ij}|^{1-\frac p s} \left(\int_{2 c_0 \wz Q_{ij}} |\nabla E\Phi(x)|^s \, dx\right)^{\frac p s} \ge  \sum_i n_i 2^{i(2-p)}.  
\end{eqnarray*}
Therefore \eqref{haluttuey} yields the asserted inequality. 
\end{proof}

\begin{proof}[Proof of Theorem~\ref{neceJordan}]
We establish the result via a case study.

\noindent{{\bf Case 1}: $z_1,\,z_2\in \partial \Omega.$} 
Let $\gamma$ be the curve constructed in \eqref{eq:alphagammadef} for the pair $z_1,z_2$. Then
$\wz \varphi^{-1}(\gamma)=\alpha.$ 
Since each $\wz Q_{ij}\in \wz W_{\gamma}$ is a Whitney square, its diameter is 
comparable to 
$\dist(\wz Q_{ij},\,\partial \Omega)$, 
 which means for the points 
$w \in \gamma \cap \wz Q_{ij}$ that
\begin{equation}\label{equat41}
\dist(w,\, \partial\Omega)\sim \diam(\wz Q_{ij})
\end{equation}
with absolute constants.

 Since $\alpha$ consists of two line segments and a circular arc, we have
\[
 \mathcal H^1(\varphi^{-1}(\wz Q_{ij})\cap \alpha) \le (2+\pi) \diam(\varphi^{-1}(\wz Q_{ij})),
\]
and by Lemma~\ref{whitney preserving}, the set $\wz \varphi^{-1}(\wz Q_{ij})$ is of $\lambda$-Whitney type with an absolute constant $\lambda$.
Thus, by Lemma~\ref{bilipominaisuus}, we get
\begin{equation}\label{equat40}
\mathcal H^{1}(\wz Q_{ij}\cap \gamma)\le C \ell(\wz Q_{ij}),
\end{equation}
for some absolute constant $C$. 

By combining the claim of Lemma \ref{sarja} with \eqref{equat41} and 
\eqref{equat40} we arrive at
\begin{align*}
\int_{\gamma} \dist(z,\,\partial \Omega)^{1-p} \, ds &\le \sum_{\wz Q_{ij}\in \wz W_{\gamma}}\int_{\gamma\cap \wz Q_{ij}} \dist(z,\,\partial \Omega)^{1-p} \, ds \\
&\le C(p) \sum_{\wz Q_{ij}\in \wz W_{\gamma}} \dist(\wz Q_{ij},\,\partial \Omega)^{2-p} \le C(\|E\|, \, p) |z_1-z_2|^{2-p}.
\end{align*}
Hence we have proven the existence of the desired curve when $z_1,z_2\in \partial \Omega.$

\noindent{{\bf Case 2}}: $z_1,z_2\in \wz \Omega\cup \partial \Omega$ and at 
least one of the points belongs to $\wz \Omega.$ By swapping $z_1$ and $z_2$, if needed, we may assume that $z_2\in \wz \Omega$ and that $\dist(z_1,\partial \Omega)\le \dist(z_2,\partial \Omega).$

Suppose first that 
\begin{equation}\label{sisalla}
|z_1-z_2|\le \dist (z_2,\partial \Omega).
\end{equation}
Then we may choose $\gamma$ to be the line segment $[z_1,\,z_2]$ 
between $z_1,z_2,$   and the curve condition \eqref{eq:extcharcomplre} is satisfied as $1<p<2$: 
\begin{equation}\label{sisallahyvin}
\begin{split}\int_{[z_1,\,z_2]} \dist(z,\,\partial \Omega)^{1-p}\,ds(z)
& \le \int_{[z_1,\,z_2]}  \dist(z,\,\partial B(z_2,\dist(z_2,\partial \Omega)))^{1-p}\,ds\\
&\le C(p) |z_1-z_2|^{2-p}. 
\end{split}
\end{equation} 

Assume now that \eqref{sisalla} fails. Choose $z_3,z_4\in \partial \Omega$ so
that 
$$|z_i-z_{i+2}|=\dist(z_i,\partial \Omega)$$
for $i=1,2.$ Then 
$$|z_1-z_3|+|z_2-z_4|< 2|z_1-z_2|$$ and 
\begin{equation}\label{eikaukana}
|z_3-z_4|\le 3|z_1-z_2|.
\end{equation}
Let $\gamma'$ be the curve connecting $z_3$ and $z_4$ obtained from Case 1, and define $\gamma = [z_1,z_3]\ast \gamma' \ast [z_4,z_2]$. Then, by \eqref{sisallahyvin} for $[z_1,z_3]$ and $[z_4,z_2]$, the Case 1, and \eqref{eikaukana}, we get
\begin{align*}
\int_{\gamma} \dist(z,\,\partial \Omega)^{1-p}\,ds(z) 
& \le C(\|E\|,p)\left(|z_1-z_3|^{2-p} + |z_3-z_4|^{2-p} + |z_4-z_2|^{2-p}\right)\\
&\le C(\|E\|,p)|z_1-z_2|^{2-p},
\end{align*}
concluding the proof also in this case.
\end{proof}

\begin{remark}\label{remark 1}
 Let $z_1,\,z_2\in \wz \Omega$. Even though the curve joining $z_1,\,z_2$ 
which we constructed in the proof above may touch  the boundary 
$\partial \Omega$, it can be modified so as to be contained in $\wz \Omega$. 
 
To begin, we may again assume that 
$$\dist(z_1,\partial \Omega)\le \dist(z_2,\partial \Omega)$$ and that
\eqref{sisalla} fails. Consider the points $z_3,z_4\in \partial \Omega$ from 
the proof 
above and let $w_i=\wz \varphi^{-1}(z_i)$ for $i=3,4.$ Since $\wz \varphi$ is 
continuous up to the boundary and \eqref{eikaukana} holds, 
we find $\epsilon>0$ so that 
$$|\wz \varphi(w)-\wz \varphi(w')|<4|z_1-z_2|$$ whenever
$w,w'\in \partial \mathbb D$ satisfy $|w-w_3|+|w'-w_4|<\epsilon.$
Recall that the curve $\gamma$ in the above proof in Case 1 is the image of the
curve $\alpha$ that consists of two radial segments and a circular arc. See
Figure \ref{fig:combination}. Suppose
that $w_3\neq w_4.$ Then we may choose $w,w'$ as above so that the 
corresponding curve $\alpha$ between $w,w'$ intersects the preimages of the
line segments between $z_1,z_3$ and between $z_2,z_4.$  This allows us to 
reroute our curve so that it does not intersect the boundary. The case of
$w_3=w_4$ is similar; choose $w,w'$ from ``different sides'' of $w_3.$

\begin{figure} 
   \includegraphics[width=0.8\textwidth]{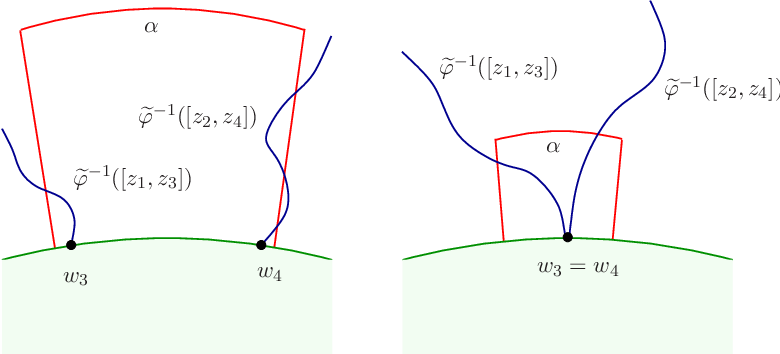}
   \caption{The curve constructed in Theorem \ref{neceJordan} can be modified so as to
travel inside $\widetilde\Omega$ by perturbing slightly the starting point and
the 
endpoint of the intermediate curve $\widetilde\varphi(\alpha)$ and by
disregarding the unnecessary parts of the concatenated curves. On the
left we have the case where the selected points $z_3$ and $z_4$ differ,
and on the right the case where they agree.}
   \label{fig:combination}
\end{figure}

 
\end{remark}

\begin{remark}
 The inequality in Lemma \ref{sarja} is actually equivalent to \eqref{eq:extcharcomplre} for 
our $\gamma$. One of the directions was shown above. For the other direction, first we note that each Whitney square has at most 20 neighboring squares, which tells us that we can distribute the squares in
$\wz W_{\gamma}$ into  no more than 21 subcollections $\{\wz W_k\}_{k=1}^{21}$ so that each of them consists of pairwise disjoint squares.
Then  for any two distinct $\wz Q_i,\,\wz Q_j\in \wz W_k$,  by Lemma~\ref{lma:whitney} we have
$$\frac {11}{10} \wz Q_i\cap \frac{11}{10} \wz Q_j =\emptyset. $$
Notice that for each $\wz Q_{ij}\in \wz W_{\gamma}$, by definition, we have
$$\mathcal H^{1}(\frac {11}{10} \wz Q_{ij}\cap \gamma)\ge \frac {1}{10}\ell(\wz Q_{ij}).$$
Thus by applying the estimate
$$\ell(\wz Q_{ij}) \le \dist(\wz Q_{ij},\,\partial \Omega)\le 4\sqrt 2 \ell(\wz Q_{ij}), $$
 we have
\begin{align*}
\sum_{\wz Q_{ij}\in \wz W_\gamma} \dist(\wz Q_{ij},\,\partial \Omega)^{2-p} & \le C(p) \sum_{k=1}^{21}\sum_{\wz Q_{ij}\in \wz W_k}\int_{\gamma\cap \wz Q_{ij}} \dist(z,\,\partial \Omega)^{1-p} \, ds \\
& \le C(p) \int_{\gamma} \dist(z,\,\partial \Omega)^{1-p} \, ds\le C(\|E\|, \, p) |z_1-z_2|^{2-p},
\end{align*}
which gives the other direction. 
\end{remark}

\subsection{Inner extension}

We prove the following inner extension theorem in this subsection. 

\begin{theorem}{\label{inner extension}}
Let $\varphi \colon \mathbb D \to \Omega$ be a conformal map, where $\Omega\subset \mathbb R^2$ is a simply connected 
John domain with John constant $J$. Suppose that $\varphi(0)$ is the 
distinguished point in the definition of a John domain.
Set $\Omega_{\epsilon}=\varphi (B(0,\,1-\epsilon))$ for 
$0<\epsilon\le \frac 1 2$ and let  $1< p<\infty.$  
Then there exists an extension 
operator $E_\epsilon\colon W^{1,\,p}(\Omega_\epsilon)\to W^{1,\,p}(\Omega)$ such that 
$\|E_{\epsilon}\|\le C(p,\,J).$ 
\end{theorem}

Fix $\ez$, and notice that $\Omega_{\ez}$ is a Jordan domain. 
Let $\Omega'_{\epsilon}=\mathbb R^2 \setminus \overline \Omega_{\epsilon}$, 
and $\widetilde\Omega_{\epsilon}=\Omega'_{\epsilon}\cap \Omega$. 
Let $\varphi$ be a conformal map as in Theorem \ref{inner extension}, with $\varphi(0)$
a John center of $\Omega.$ This map will
be fixed through this subsection.
By Lemma \ref{John subdomain}, $\varphi$ is $\eta$-quasisymmetric with respect to the 
inner distance, where $\eta$ depends only on $J.$ Moreover, by Remark \ref{uptobdy}, 
we may extend $\varphi $ continuously to the boundary $\partial \mathbb D;$ we 
denote the extended map 
still by $\varphi $. 

We are going to modify the method of P.W. Jones from \cite{jo1981} to prove 
Theorem \ref{inner extension}. 
%
%
%
%
We will construct a suitable cover for $\widetilde\Omega_{\epsilon}$ inside 
$\Omega$ and an associated partition of unity. Towards this, recall that
$\Omega$ is John and that, by Lemma~\ref{John subdomain}, so is 
$\Omega_{\ez},$ with a constant only depending on $J.$ 
Thus by Lemma \ref{johninner}, we have that $\Omega_{\epsilon}$ is inner uniform and  we may use hyperbolic segments of
$\Omega_{\epsilon}$ as curves referred to in the definition of inner uniformity, with constant 
$\ez_0$
only depending on $J.$

Fix $k_0\in \mathbb N$ with $2^{-k_0-1}<\ez \le 2^{-k_0}$. 
We begin by constructing a decomposition of the preimage $A=\mathbb D\setminus \overline{B(0,\,1-\epsilon)}$,
of $\widetilde\Omega_{\epsilon}$ under 
$\varphi$, and then obtain a decomposition of $\wz \Omega_\ez$ with the help of the map 
$\varphi$.  See Figure~\ref{fig:innerext}.

For $k\in \mathbb N$ let
$$A_k=B(0,\,1-\epsilon+2^{-k}\epsilon)\setminus B(0,\,1-\epsilon+2^{-k-1}\epsilon). $$
For each $k\ge 0$,  the collection of the $2^{k+k_0}$ radial rays obtained by
dividing the polar angle $2\pi$ evenly and by starting with the zero angle 
subdivides $A_k$ into closed  (with respect to $\mathbb D$) 
sets. Run this process for all $k\in \mathbb N$. We refer to these closed sets by 
$\wz Q_i$.  They satisfy the version  
\begin{equation}\label{uusiwhitney}
\frac 1 \lambda \diam(\wz Q_i)\le \dist(\wz Q_i,\,\partial \left(B(0,1-\ez) ) \right) \le \lambda \diam(\wz Q_i)
\end{equation}
of (i) in Definition~\ref{whitney-type set} with $\lambda=16\pi$. 

Set $\wz S_i=\varphi(\wz Q_i)$ and let $\wz W= \{ \wz S_i\}$. 
We claim that  each $\wz S_i$ is a $\lambda$-Whitney-type set with respect to 
the inner
distance of $\Omega$ and $\partial \Omega_{\ez}$  in the following sense.   

\begin{lemma}
There 
exists a constant $0<c=c(J)<1$ such that
\begin{equation}\label{equat20}
B_{\Omega}(w_i,\,c\diam_{\Omega}(\wz S_i))\subset \wz S_i
\end{equation}
for some $w_i\in \wz S_i,$ 
\begin{equation}\label{equat11}
c  \diam_{\Omega}(\wz S_i)\le  \dist_{\Omega}(\wz S_i,\,\partial \Omega_{\ez})\le \frac 1 c  \diam_{\Omega}(\wz S_i) 
\end{equation}
and
\begin{equation}\label{equat200}
c\diam_{\Omega}(\wz S_i)\le \diam_{\Omega}(\wz S_j) \le \frac 1 c  \diam_{\Omega}(\wz S_i)
\end{equation}
whenever 
$\wz S_i\cap \wz S_j \neq\emptyset.$
Here $B_{\Omega}(x,\,r)$ denotes the open disk centered at $x$ with radius $r$ with 
respect to the inner distance.
\end{lemma}
\begin{proof}
Let $\wz S_i$ be fixed. By the construction of $\wz Q_i = \varphi^{-1}(\wz S_i)$ there is a disk $B(z_0,\,c_0\diam(\wz Q_i))$ contained in $\wz Q_i$ for some absolute constant $c_0\le 1$. Let $z_1$ be an arbitrary point on the boundary of $B(z_0,\,c_0\diam(\wz Q_i))$ and let $z_2\in  \wz Q_i$ be such that 
\begin{equation}\label{lisatag0}
\dist_{\Omega}(\varphi(z_2),\,\varphi(z_0))\ge \frac 1 3 \diam_{\Omega}(\wz S_i);
\end{equation}
the existence of such a point follows from the triangle inequality. 
Then 
$$|z_2-z_1|\le c_0^{-1}|z_1-z_0|$$ and hence \eqref{lisatag0} together with 
quasisymmetry gives
\begin{equation}\label{lisatag}
\diam_{\Omega}(\wz S_i)\le 3 \dist_{\Omega}(\varphi(z_2),\,\varphi(z_0)) \le 3 \eta(c_0^{-1}) \dist_{\Omega}(\varphi(z_1),\,\varphi(z_0)). 
\end{equation}
By the arbitrariness of $z_1$ and the fact that $\varphi$ is a homeomorphism, we conclude \eqref{equat20} for a constant $c=c(\eta)=c(J)$. 

Towards \eqref{equat11}, first choose points $z_3\in \partial B(0,1-\ez)$  and $z_4\in 
\wz Q_i$ such that
\begin{equation}\label{equat44}
\dist_{\Omega}(\varphi(z_4),\,\varphi(z_3))\le 
2\dist_{\Omega}(\wz S_i,\,\partial \Omega_{\ez}). 
\end{equation}
Let $z\in \wz Q_i$ be such that 
\begin{equation}\label{lisatag3}
\diam(\wz Q_i)\le 2|z_4-z|.
\end{equation}
By \eqref{uusiwhitney} 
\begin{equation}\label{lisatag1}
|z_4-z|\leq \diam(\wz Q_i)\sim \dist(\wz Q_i,\, \partial B(0,1-\ez))\ls |z_4-z_3|
\end{equation}
with absolute constants.
Now \eqref{lisatag1},  quasisymmetry of $\varphi$ and \eqref{equat44} give
\begin{equation}\label{equat45}
 \dist_{\Omega}(\varphi(z_4),\,\varphi(z))\le {C(\eta)} \dist_{\Omega}(\varphi(z_3),\,\varphi(z_4))\leq{C(\eta)}\dist_{\Omega}(\wz S_i,\,\partial \Omega_{\ez}).
\end{equation}
Let $z_0$ be as in the first paragraph of the proof. By the triangle inequality,
$|z-z_0|\ge \frac 1 4 \diam(\wz Q_i)$ or $|z_4-z_0|\ge \frac 1 4 \diam(\wz Q_i).$
Assume that the latter inequality holds; the other case is handled analogously.
Clearly, $|z_4-z_0|\le \diam(\wz Q_i)\le 2|z_4-z|$ by \eqref{lisatag3}.
Hence quasisymmetry gives
\begin{equation}\label{lisatag4}
\dist_{\Omega}(\varphi(z_4),\,\varphi(z_0))\le 
\eta(2)\dist_{\Omega}(\varphi(z_4),\,\varphi(z)). 
\end{equation}
Let us argue as in the first paragraph of the proof with the help of the point $z_2 \in \wz Q_i$.
Our assumption that
$|z_4-z_0|\ge \frac 1 4 \diam(\wz Q_i) \ge \frac14|z_2-z_0|$
together with quasisymmetry further gives
\begin{equation}\label{lisatag5}
\diam_{\Omega}(\wz S_i)\le 3\dist_{\Omega}(\varphi(z_2),\,\varphi(z_0))\le  3 \eta(4)\dist_{\Omega}(\varphi(z_4),\,\varphi(z_0)).
\end{equation}
We obtain the lower bound of the distance in \eqref{equat11} by combining
\eqref{equat45}, \eqref{lisatag4} and \eqref{lisatag5}.

Towards the upper bound in \eqref{equat11}, 
pick points $z_5\in \partial B(0,1-\ez)$  and $z_6\in \wz Q_i$ such that 
$$|z_5-z_6|=\dist(\wz Q_i,\, \partial B(0,1-\ez)). $$
Let $z_0$ and $c_0$ be as in the first paragraph of our proof.
Then \eqref{uusiwhitney} gives
$$|z_5-z_6|\le c_0^{-1}\lambda |z_0-z_6|,$$
and by quasisymmetry 
$$\dist_{\Omega}(\wz S_i,\,\partial \Omega_{\ez}) \le \dist_{\Omega}(\varphi(z_5),\,\varphi(z_6))\le \eta(c_0^{-1}\lambda) \dist_{\Omega}(\varphi(z_6),\,\varphi(z_0))\leq  \eta(c_0^{-1}\lambda)\diam_{\Omega}(\wz S_i),$$
as desired.

We are left to prove \eqref{equat200}. Since $\wz S_i\cap \wz S_j\neq \emptyset,$ both
$$\dist_{\Omega}(\wz S_j,\,\partial \Omega_{\ez})\le \dist_{\Omega}(\wz S_i,\,\partial 
\Omega_{\ez})+\diam_{\Omega}(\wz S_i)$$
and the analogous inequality with the roles of $i,j$ reversed hold. Hence \eqref{equat200}
follows from \eqref{equat11}.
\end{proof}



Given $\widetilde S_{i}\in \widetilde W$ and $M>1$ that will be selected soon,
define
$$\wz U_{i}:=\{x\in\Omega \mid \dist_{\Omega}(x,\,\widetilde S_{i})< \frac 1 {M} \diam_{\Omega}(\widetilde S_{i})\}. $$
We claim that we can choose $M>1$ depending only on $J$ such that these sets 
$\wz U_i$ have uniformly finite overlaps. 
Notice that $\wz U_{i}\subset \wz \Omega$ whenever $M\ge 2/c$ for the constant
$c$ in \eqref{equat11}.

\begin{lemma} \label{naapurit}
If $\wz S_i\cap \wz S_j=\emptyset,$ then
\begin{equation}\label{equat1}
\max\{\diam_{\Omega}(\wz S_i),\,\diam_{\Omega}(\wz S_j)\}\le C(J)\dist_{\Omega}(\wz S_i,\,\wz S_j),
\end{equation}
Especially, for $M\ge 2C(J)$ in the definition of
the sets  $\wz U_i$ we have 
\begin{equation}\label{equat2}
1 \le \sum_{i}\chi_{\wz U_i}(x)\le 9
\end{equation}
for every $x\in \wz \Omega_{\ez},$ 
where $\chi_{\wz U_i}$ is the characteristic function of $\wz U_i$. 
\end{lemma}
\begin{proof} 
First, observe that  $\wz Q_i\cap \wz Q_j=\emptyset$ gives
$$\dist (\wz Q_i,\,\wz Q_j)\ge C \max\{\diam (\wz Q_i),\,\diam (\wz Q_j)\},$$
where the constant is absolute. 
We apply quasisymmetry to show that $\wz S_i\cap \wz S_j=\emptyset$ implies
$$
\dist_{\Omega}(\wz S_i,\,\wz S_j)\gs\max\{\diam_{\Omega}(\wz S_i),\,\diam_{\Omega}(\wz S_j)\},
$$
where the constant depends only on the John constant; also see \cite[Formula (3.5)]{KZ2016} for a version of this. 
Towards this, choose $w_1\in \wz S_i$ and $w_2 \in \wz S_j $ such that
\begin{equation}\label{inequat 4} 
\dist_{\Omega}(w_1,\,w_2)\le 2 \dist_{\Omega}(\wz S_i,\,\wz S_j),
\end{equation}
and let $w_3\in \wz S_i$ be an arbitrary  point. Then since 
$$|\varphi^{-1}(w_1)-\varphi^{-1}(w_2)|\ge \dist (\wz Q_i,\,\wz Q_j)\gs \diam (\wz Q_i)\ge |\varphi^{-1}(w_1)-\varphi^{-1}(w_3)|$$
with an absolute constant, 
the quasisymmetry of $\varphi$ applied to $\varphi^{-1}(w_1),\,\varphi^{-1}(w_2)$ and $\varphi^{-1}(w_3)$
gives
$$ \dist_{\Omega}(w_1,\,w_2)\gs \dist_{\Omega}(w_1,\,w_3). $$
Thus by the arbitrariness of $w_3$, \eqref{inequat 4} shows that
$$\dist_{\Omega}(\wz S_i,\,\wz S_j)\gs  \diam_{\Omega}(\wz S_i).$$
By replacing $w_3\in \wz S_i$ with $w_3\in \wz S_j$ and $\diam(\wz Q_i)$ with $\diam(\wz Q_j)$ above, respectively, we analogously obtain the inequality 
$$\dist_{\Omega}(\wz S_i,\,\wz S_j)\gs  \diam_{\Omega}(\wz S_j),$$ and \eqref{equat1} follows. 

Regarding \eqref{equat2}, the  lower bound is trivial since $\widetilde W$ 
forms a cover of $\widetilde \Omega_{\epsilon}.$
Since each $\wz S_i$ has at most $8$ neighboring sets, we obtain the upper 
bound in \eqref{equat2} from \eqref{equat1}. 
\end{proof}


We now fix $M=\max\{2C(J),2/c\},$ where the constant $C(J)$ is from 
\eqref{equat1}
and $c$ is from \eqref{equat11}.
Given $\widetilde S_{i}\in \widetilde W$, set
$$\psi_{i}(x)=\max\{1- 2M \diam_{\Omega}(\widetilde S_{i})^{-1} \dist_{\boz}(x,\, \wz S_i),\, 0\}$$
for $x\in \Omega.$ 
Then  $\psi_{i}$ is  locally Lipschitz with bounded and relatively closed 
support
in $\Omega$,  $|\nabla \psi_{i}|\le\, C(J) \diam_{\Omega}(\widetilde S_{i})^{-1}$ and $\psi_i(x)=1$ for any 
$x\in \widetilde S_{i}.$ Moreover, the support of   $\psi_{i}$ is
contained in $\wz U_{i}.$
Define 
$$\phi_j(x)= \frac {\psi_j(x)}{\sum_i \psi_i(x)}$$
for $x\in  \widetilde\Omega_{\epsilon}.$ Then our collection of the
functions $\phi_{j}$ is a partition of unity in $\widetilde\Omega_{\epsilon}:$
$\sum \phi_{j}(x)=1$ in $\widetilde\Omega_{\epsilon}.$
By \eqref{equat2} also  the functions $\phi_{j}$ are
locally Lipschitz, have supports in $U_j,$ and 
\begin{equation}\label{inequ 210}
|\nabla \phi_{j}|\le\, C(J) \diam_{\Omega}(\widetilde S_{j})^{-1}.
\end{equation}


\begin{figure}
  \centering
  \includegraphics[width=1\textwidth]{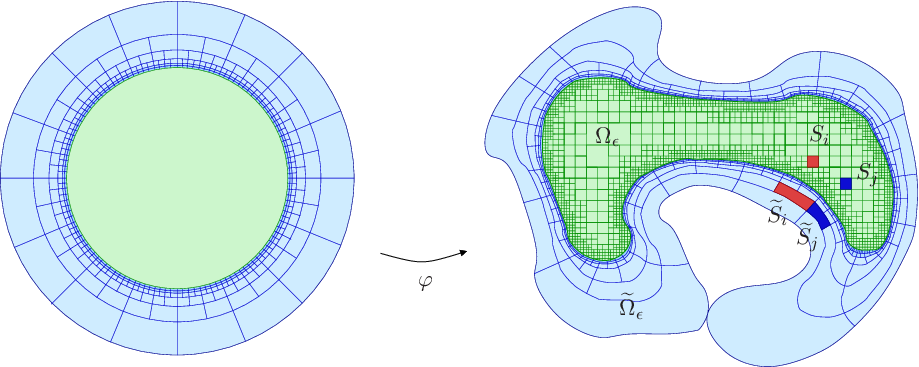}
\caption{In the inner extension the annular region $\wz
\Omega_\epsilon$ is divided into Whitney-type sets that are obtained by 
mapping a Whitney-type decomposition of the annulus inside the disk 
conformally. For the inner part $\Omega_\epsilon$ we use a 
standard Whitney decomposition. Two pairs of sets $(\wz{S}_i,S_i)$ 
and $(\wz{S}_j,S_j)$ are highlighted.}
  \label{fig:innerext}
\end{figure}

In order to construct our extension operator, we associate to  
each
$\widetilde S_{i}\in \widetilde W$ a suitable square $Q\in W,$ 
where $W=\{Q_1,Q_2,\dots \}$ is a fixed Whitney decomposition of $\Omega_{\epsilon}$; 
see Figure~\ref{fig:innerext}. 

\begin{lemma}\label{lma:Sichoice}
Given $\widetilde S_{i}\in \widetilde W$ there is $Q\in W$ such that 
\begin{equation}{\label{choose candidates}}
 \diam(Q)=\diam_{\Omega} (Q) \sim_J \dist_{\Omega}(\wz S_{i},\, Q)\sim_J \diam_{\Omega}(\wz S_i). 
\end{equation}
\end{lemma}
\begin{proof}
To see that a Whitney square of desired size can be chosen, trace back towards 
$\varphi(0)$ along any 
hyperbolic ray of $\Omega$ that intersects $\wz S_{i}$ and let
$Q$ be a first Whitney square of $\Omega_{\ez}$ intersecting this hyperbolic ray such that
\begin{equation}\label{equat92}
\dist(\varphi^{-1}(Q),\,\varphi^{-1}(\wz S_i))\ge \frac 1 {9\lambda} \diam(\varphi^{-1}(\wz S_i)), 
\end{equation}
where $\lambda$ is an absolute constant given by Lemma~\ref{whitney preserving}
so that $\varphi^{-1}(Q)$ is of $\lambda$-Whitney-type with respect to $B(0,1-\ez)$. We show the existence of such a square via Definition~\ref{whitney-type set} and the assumption that $0<\ez\le \frac 1 2$. Towards this, 
if such a square does not exist, then
$$\frac {\dist(\varphi^{-1}(Q),\,\varphi^{-1}(\wz S_i))}{\diam(\varphi^{-1}(\wz S_i))}< \frac 1 {9\lambda}  $$
for all the Whitney squares $Q$ intersecting  our fixed  hyperbolic ray. However, the diameter of $\varphi^{-1}(\wz S_i)=\wz Q_i$  is at most 2, while a $\lambda$-Whitney-type set in $B(0,\, 1-\ez)$ containing the origin has distance to $\partial B(0,\, 1-\ez)$ at least $\frac 1 {4\lambda}$ since $\ez \le \frac 1 2$ and $\lambda \ge 1$. 
Therefore we have
$$  \frac 1 {8\lambda}\le \frac {\dist(\varphi^{-1}(Q),\,\varphi^{-1}(\wz S_i))}{\diam(\varphi^{-1}(\wz S_i))}< \frac 1 {9\lambda}, $$
which leads to a contradiction. 
Then the facts that $Q$ is a first square satisfying \eqref{equat92}, 
$\varphi^{-1}(Q)$ is of
$\lambda$-Whitney type, $\varphi^{-1}(\wz S_i)=\wz Q_i$  satisfies \eqref{uusiwhitney}  and that $\dist_{\Omega}(\varphi^{-1}(Q),\,\varphi^{-1}(\wz S_i))$ is comparable to the length of the segment of our hyperbolic ray between
$\varphi^{-1}(Q)$ and $\varphi^{-1}(\wz S_i)$  allow us to deduce that
\begin{equation}\label{uusiekv}
\diam(\varphi^{-1}(\wz S_i))\sim \dist(\varphi^{-1}(Q), \, \varphi^{-1}(\wz S_i)) \sim \diam(\varphi^{-1}(S)). 
\end{equation}
Next we apply the quasisymmetry of $\varphi$ with respect to the inner distance to establish \eqref{choose candidates}. First of all choose $x_1\in  \wz S_i $ and $x_2\in Q $ such that
\begin{equation}\label{inequat 5} 
\dist_{\Omega}( Q, \,  \wz S_i )\le \dist_{\Omega}(x_1,\,x_2)\le 2\dist_{\Omega}( Q, \,  \wz S_i ), \end{equation}
and let $x_3\in \wz S_i$ be an arbitrary point. 
Since $x_1,x_3\in \wz S_i$ and $x_2\in Q,$ \eqref{uusiekv} implies that
$$|\varphi^{-1}(x_1) - \varphi^{-1}(x_2)|\ge \dist(\varphi^{-1}(Q), \, \varphi^{-1}(\wz S_i))\ge  C^{-1}|\varphi^{-1}(x_1)-\varphi^{-1}(x_3)| $$
with an absolute constant $C$, 
and hence the quasisymmetry of $\varphi$ gives
$$ \dist_{\Omega}(x_1,\,x_3)\le C(J) \dist_{\Omega}(x_1,\, x_2). $$
Thus \eqref{inequat 5} gives
\begin{equation}\label{ey}
 \diam_{\Omega}( \wz S_i )\le C(J)\dist_{\Omega}( Q, \,  \wz S_i )
\end{equation}
according to the arbitrariness of $x_3.$
For the other direction, choose $x_4\in \varphi^{-1}(\wz S_i),\,x_5\in \varphi^{-1}(Q)$ such that
$$\dist(\varphi^{-1}(Q), \, \varphi^{-1}(\wz S_i)) \le |x_4-x_5|\le 2 \dist(\varphi^{-1}(Q), \, \varphi^{-1}(\wz S_i)).$$
Pick $x_6\in \varphi^{-1}(\wz S_i)$ such that
$$\diam(\varphi^{-1}(\wz S_i))\le 2 |x_6-x_4|;$$
the existence of such a point follows from the triangle inequality.   
By \eqref{uusiekv}
$$|x_6-x_4|\ge \frac 1 2 \diam(\varphi^{-1}(\wz S_i))\ge C^{-1}  \dist(\varphi^{-1}(Q), \, \varphi^{-1}(\wz S_i))\ge C^{-1}|x_5-x_4|. $$
Then the quasisymmetry of $\varphi$ gives
$$\diam_{\Omega}(\wz S_i)\ge \dist_{\Omega}(\varphi(x_6),\,\varphi(x_4))\ge  \eta(C)^{-1}\dist_{\Omega}(\varphi(x_5),\,\varphi(x_4))\ge \eta(C)^{-1}  \dist_{\Omega}( Q, \,  \wz S_i ),$$
which together with \eqref{ey} gives the last equivalence in \eqref{choose candidates}. The second 
equivalence follows by an analogous argument by changing the roles of $\wz S_i$ and $Q$,  and the first equality is obvious. 
\end{proof}

We now use Lemma \ref{lma:Sichoice} to pick for each
$\widetilde S_{i}\in \widetilde W$ a square $Q\in W.$
By  Lemma \ref{lma:Sichoice} the collection of those $Q$ that satisfy \eqref{choose candidates} is a nonempty subcollection of $W.$ Recalling that $W=\{Q_1,Q_2,\dots\}$ we simply pick
the one of them that has the smallest index amongst those that belong to our subcollection. For simplicity of notation, we refer to this $Q$ associated to $\wz S_i$ by $S_i.$

Then, by \eqref{choose candidates}, we know that the inner 
distance between  $\wz S_{i}$ and $S_i$  with respect to $\Omega$ is no more
than a constant times $\diam_{\Omega}(\wz S_{i})$. By the triangle inequality 
and \eqref{equat200} it follows that
$$\dist_{\Omega}(S_{i},\, S_{j})\lesssim  \diam_{\Omega}(\wz S_{i})$$ if 
$\widetilde S_i\cap \widetilde S_j\neq\emptyset.$
Given $\widetilde S_i, \widetilde S_j$ with $\widetilde S_i \cap \widetilde S_j \ne \emptyset $, we consider  the hyperbolic segment $\Gamma$ in $\Omega$ joining the barycenters $x_{S_i}$ and $x_{S_j}$ of $S_i,S_j,$
respectively.
From Lemma \ref{lma:map and geodesic} and \eqref{choose candidates} 
we conclude that the Euclidean length of
the 
hyperbolic segment
$\Gamma$ is no more than constant times $\diam_{\Omega}(\wz S_{i}).$ Since
$\Omega_{\epsilon}=\varphi (B(0,\,1-\epsilon))$ and $S_i,S_j\subset \Omega_{\epsilon},$ it follows that the hyperbolic segment $\Gamma$ joining $x_{S_i}$ and $x_{S_j}$ 
is contained in  $\Omega_{\epsilon}.$ We use  Lemma \ref{lma:map and geodesic} 
a second time to conclude that the Euclidean length of a hyperbolic segment $\Gamma_{i,j}$ 
joining $x_{S_i}$ and $x_{S_j}$ that is hyperbolic with respect to $\Omega_{\epsilon}$
 is also bounded from above by a constant 
times $\diam_{\Omega}(\wz S_{i}):$ 
\begin{equation}\label{lisattypituus5}
\elle(\Gamma_{i,j})\ls \diam_{\Omega}(\wz S_{i}).
\end{equation}

Let us define $G(\widetilde S_i, \widetilde S_j)$ 
to be the union of all squares in our fixed Whitney decomposition $W$ 
of $\Omega_{\epsilon}$ that intersect this fixed
geodesic $\Gamma_{i,j}$. 

Next we show that the inner uniformity of $\Omega_{\ez}$
allows us to conclude that there are 
uniformly finitely many Whitney squares in every 
$G(\widetilde S_i, \widetilde S_j)$ with 
$\widetilde S_i\cap \widetilde S_j\neq\emptyset.$ This is a counterpart of \cite[Lemma 2.8]{jo1981} with a similar proof. 

\begin{lemma}\label{lisattypaall}
Let $i,j$ be such that $\widetilde S_i\cap \widetilde S_j\neq\emptyset.$ Then
\begin{equation}\label{equat12}
\#\left\{S_k\in W\mid S_k\in G(\widetilde S_i, \widetilde S_j)\right\} 
\le C(J), 
\end{equation}
where $\#$ denotes the counting measure.
\end{lemma} 
\begin{proof} 
Since $\diam(S_i)\lesssim \diam_{\Omega}(\wz S_{i})$ by 
\eqref{choose candidates} and the curve $\Gamma_{i,j}$ intersects the Whitney square
$S_i$, we conclude
by \eqref{lisattypituus5} that the diameter of each Whitney square of 
$\Omega_{\ez}$ that intersects $\Gamma_{i,j}$ 
is bounded from above by a fixed multiple of $\diam_{\Omega}(\wz S_{i})$.

On the other hand, by \eqref{choose candidates} with \eqref{equat200},
\begin{equation}\label{equat201}
\diam_{\Omega}(S_{i})\sim \diam_{\Omega}(\wz S_{i})\sim \diam_{\Omega}(\wz S_{j})\sim \diam_{\Omega}(S_{j}).
\end{equation}
Hence, the second condition of \eqref{e1.1} together with 
\eqref{equat201} tells us that
$$\dist(Q,\,\partial \Omega_{\ez})\gs \diam_{\Omega}(\wz S_{i})$$
if $Q\cap \Gamma\neq \emptyset.$ 

Thus the diameters of 
$Q\in W$ with $Q\cap \Gamma\neq \emptyset$ are bounded from below and from
above by fixed multiples of $\diam_{\Omega}(\wz S_{i})$, and hence 
\eqref{equat12} follows as $\elle(\Gamma_{i,j})\ls \diam_{\Omega}(\wz S_{i})$. 
\end{proof}

Given $u\in L^1(\Omega_{\epsilon})$, set
$$a_{i}=\bint_{S_{i}} u(x)\,dx=\frac 1 {|S_i|} \int_{S_i} u(x)\, dx,$$
where $S_{i}\in W$ is the square associated to 
$\widetilde S_{i}\in \widetilde W.$  Recall our partition of
unity consisting of the functions $\phi_i,$ see the discussion before \eqref{inequ 210}.
Define
\begin{equation}\label{eq:Eepsilondefinition}
E_{\epsilon}u(x)=\sum_{i}a_{i}\phi_{i}(x),\ x\in \wz \Omega_{\ez}
\end{equation}
for any given  function $u\in W^{1,\,p}(\Omega_{\epsilon})$ which is Lipschitz in $\overline{\Omega}_{\epsilon}$, and set $E_{\epsilon}u(x)=u(x)$ when $x\in \overline{\Omega}_{\epsilon}$. 

\begin{lemma}\label{ketjutus}
Let $E_\epsilon$ be given by \eqref{eq:Eepsilondefinition}.
Given $\widetilde S\in \widetilde W,$ we have the estimate
\begin{equation} \label{inequat 200} \|\nabla(E_{\epsilon}u(x))\|^p_{L^p(\widetilde S)}\ls \sum_{\wz S_{k}\cap \wz S\neq \emptyset}\int_{G(\wz S,\,\wz S_{k})}|\nabla u(x)|^p\, dx
\end{equation}
with a constant that only depends on $p$ and $J.$
\end{lemma}
\begin{proof}
Fix $\widetilde S\in \widetilde W$  and set $a=\bint_{S} u(x)\,dx.$ 
Then 
$$\nabla E_{\epsilon} u(x)=\nabla (E_{\epsilon}u(x)-a)=\nabla\left(\sum_{i}\phi_i(x)(a_i-a)\right) $$
in $\wz S$. 
By \eqref{equat12}, $G(\wz S,\,\wz S_{k})$  consists of  no more than $C(J)$
squares and the side lengths of all of them are comparable to $\diam_{\Omega}(S)$ modulo a multiplicative constant that only depends on $J.$
Hence \eqref{inequ 210}, \eqref{choose candidates} and  the Poincar\'e inequality
 (see e.g. \cite[Lemma 3.1]{jo1981} for the use of the Poincar\'e inequality) 
applied to the chain 
$G(\wz S,\wz S_k)$ of squares give
\begin{eqnarray}
\|\nabla(E_{\epsilon}u(x))\|^p_{L^p(\widetilde S)}
&\lesssim&  \int_{\wz S} \sum_{\wz S_{k}\cap \wz S\neq \emptyset} |a_k-a|^p  |\nabla   \phi_k(x)|^p \, dx \nonumber \\
&\lesssim& \sum_{\wz S_{k}\cap \wz S\neq \emptyset} |a_{k}-a|^p (\diam_{\Omega} (S))^{2-p} \nonumber \\
&\lesssim& \sum_{\wz S_{k}\cap \wz S\neq \emptyset} (\diam_{\Omega} (S))^{2-p} (\diam_{\Omega} (S))^{p-2} \int_{G(\wz S,\,\wz S_{k})}|\nabla u(x)|^p\, dx\nonumber \\
&\lesssim& \sum_{\wz S_{k}\cap \wz S\neq \emptyset} \int_{G(\wz S,\,\wz S_{k})}|\nabla u(x)|^p\, dx. 
\end{eqnarray}
\end{proof}

We are now ready to prove our norm estimate.

\begin{lemma} \label{normiestimaatti}
Let $E_\epsilon$ be given by \eqref{eq:Eepsilondefinition}.
We have 
$$\|E_{\epsilon}u\|^p_{L^p(\widetilde \Omega_\epsilon)}+\|\nabla(E_{\epsilon}u)\|^p_{L^p(\widetilde \Omega_\epsilon)}\ls  \|u\|^p_{L^p(\Omega_\epsilon)}+\|\nabla u\|^p_{L^p(\Omega_\epsilon)}$$
with a constant only depending on $p$ and $J.$
\end{lemma}
\begin{proof}
We begin by estimating the overlaps of $G(\wz S_k,\wz S_i).$ Towards this, for a fixed
$S_i,$ we first
bound the number of distinct $\wz S$ for which $S_i$ is associated to $\wz S.$

To begin with,  
\eqref{choose candidates} implies that, for a fixed $S_i\in W$,  
for every $\wz S\in \wz W$ associated to it we have
\begin{equation}\label{inequat 15}
\dist_{\Omega}(\wz S,\, S_i) \ls \diam_{\Omega}(S_i). 
\end{equation}
We claim that there are no more than $N(J)$ sets $\wz S\in \wz W$ associated to
a fixed $S_i\in W$. Towards this, first note that for any $x\in \Omega$ and 
$0<r<\diam(\Omega)$,  the hyperbolic segment $\Gamma$ of $\Omega$ joining $x$ to a point  $y\in \partial B_{\Omega}(x,\,r)$ satisfies
$$r=  \dist_{\Omega}(x,\,y)\le \elle(B_{\Omega}(x,\,r)\cap \Gamma).$$
Let $z \in \Gamma$ be such that $\elle(\Gamma[x,z]) = \frac{r}2$.
Then since, by Lemma \ref{johninner}, hyperbolic segments of $\Omega$ satisfy \eqref{e1.1} with a constant $0<\epsilon_0=\epsilon_0(J)<1$, we have
$$B\left(z,\,\frac 12 \epsilon_0 r\right)\subset B_{\Omega}(x,\,r).$$
Thus
\begin{equation}\label{equat93}
C(J)r^2 \le  |B_{\Omega}(x,\,r)|\le \pi r^2, 
\end{equation}
where the upper bound comes from 
$$B_{\Omega}(x,\,r)\subset B(x,\,r). $$

By \eqref{equat20}, each $\widetilde S$ associated with a fixed $S_i$ contains a Euclidean disk of radius comparable to $\diam_\Omega(\widetilde S)$, which is turn, by \eqref{choose candidates}, is comparable to $\dist_\Omega(\widetilde S,S_i)$ and to $\diam_\Omega(S_i)$. Hence, the number of such $\widetilde S$ is no more than a constant $N(J)$.
 
Since $\widetilde S_{i}$ has no more than 8 neighbors and the number of the 
sets $\wz S$ associated to any $S\in W$ is no more than $N(J),$ 
by \eqref{choose candidates}, \eqref{equat93} and \eqref{equat12} we conclude that
\begin{equation}\label{peter 3.2}
\sum_{\wz S_i\in \wz W}\sum_{\widetilde S_{i}\cap \widetilde S_k\neq \emptyset} \chi_{G(\wz S_k,\,\wz S_{i})}(x)
\lesssim 1,
\end{equation}
for all $x$; notice that  \eqref{equat12}  with \eqref{choose candidates} and \eqref{equat93} implies that each Whitney square is contained in at most uniformly finitely many chains. 
Inequality \eqref{peter 3.2} is the counterpart of    \cite[Page 80, Formula (3.2)]{jo1981}. 

Now Lemma \ref{ketjutus} together with \eqref{peter 3.2} 
gives
\begin{eqnarray*}
\|\nabla(E_{\epsilon}u)\|^p_{L^p(\widetilde \Omega_\epsilon)}
& = & \sum_{\wz S_i\in \wz W} \|\nabla(E_{\epsilon}u)\|^p_{L^p(\wz S_i)} \\
&\lesssim& \sum_{\wz S_i\in \wz W} \sum_{\widetilde S_{k}\cap \widetilde S_{i}\neq \emptyset} \int_{G(\wz S_{i},\,\wz S_{k})}|\nabla u(x)|^p\, dx \nonumber \\
& = & \int_{\Omega_\epsilon} \sum_{\wz S_i\in \wz W} \sum_{\widetilde S_{k}\cap \widetilde S_{i}\neq \emptyset} \chi_{G(\wz S_k,\,\wz S_{i})}(x)|\nabla u(x)|^p\, dx  \\
&\lesssim& 
\|\nabla u\|^p_{L^p(\Omega_\epsilon)}, 
\end{eqnarray*}
with the constant only depending on $p$ and $J$.

We are left to estimate the integral of $|Eu|^p$ over 
$\widetilde \Omega_\epsilon.$ By the definition of $Eu$ we have
$$\int_{\wz S_i}|Eu|^p\, dx\ls \sum_{\widetilde S_{i}\cap \widetilde S_k\neq \emptyset} 
\int_{S_k}|u|^p\, dx$$
and the desired bound follows via \eqref{peter 3.2}.  
\end{proof}
 
Finally, we prove Theorem~\ref{inner extension}.

\begin{proof}[Proof of Theorem \ref{inner extension}]
Let $E_\epsilon$ be given by \eqref{eq:Eepsilondefinition}.
Let us first show that the above procedure gives us an extension of our 
Lipschitz function $u$ to a function  $E_{\epsilon}u$ in   
$W^{1,\,p}(\Omega),$ 
with the desired norm bound. 
Towards this, we  claim that $E_{\epsilon}u$ is locally Lipschitz in $\Omega.$

According to our construction, $E_{\epsilon}u$ is smooth in $\wz \boz_{\ez}$. Hence to show the local Lipschitz continuity,  we only need to consider the case where $z_1\in \overline{\Omega}_{\epsilon}$ and $z_2\in \wz \boz_{\ez}$ with
$$B(z_2,\, 2|z_1-z_2|)\subset \Omega.$$
Suppose that $z_2\in \wz S$ for some $\wz S\in \wz W$. Then by \eqref{choose candidates} and the Lipschitz continuity of $u$ we have
\begin{eqnarray*}
|E_{\epsilon}u(z_2)-u(z_1)|&\le& \sum_{\widetilde S_{k}\cap \widetilde S\neq \emptyset} \phi_k(z_2)|a_k-u(z_1)|\\
 &\ls& \sum_{\widetilde S_{k}\cap \widetilde S\neq \emptyset} \phi_k(z_2) (\dist(z_1,\,S_{k})+\diam(S_k))\\
&\ls&  \sum_{\widetilde S_{k}\cap \widetilde S\neq \emptyset} \phi_k(z_2) (|z_1-z_2|+\diam_{\Omega}(\wz S_k))\ls |z_1-z_2|, 
\end{eqnarray*}
where in the last inequality we applied the facts that for $\widetilde S_{k}\cap \widetilde S\neq \emptyset$ it holds that
$$\diam_{\Omega}(\wz S_k)\sim \dist_{\Omega}(S,\,{\Omega}_{\epsilon})\sim \dist_{\Omega}(z_2,\,{\Omega}_{\epsilon})\le |z_1-z_2|. $$
Therefore we obtain the local Lipschitz continuity of $E_{\epsilon}u$.

Recall that $\partial \Omega_{\ez}$ has Lebesgue measure zero by Lemma~\ref{bdyzero}. Hence Lemma \ref{normiestimaatti} and the local 
Lipschitz continuity of $E_{\epsilon}u$ give that $Eu\in W^{1,p}(\Omega)$ with
$$\|u\|_{L^p(\Omega)}+ \|\nabla(E_{\epsilon}u)\|_{L^p(\Omega)}\le C(J,\,p)\left(  \| u\|_{L^p(\Omega_\epsilon)}+\|\nabla u\|_{L^p(\Omega_\epsilon)}\right).$$
Consequently $E_{\ez}$ is a bounded linear operator from $W^{1,\,p}(\Omega_{\ez})\cap Lip(\overline{\Omega}_{\ez})$ to $W^{1,\,p}(\Omega)$. 
Next, $W^{1,\,p}(\Omega_{\ez})\cap Lip(\overline{\Omega}_{\ez})$ is 
dense in $W^{1,\,p}(\Omega_{\ez})$ for $1<p<2$:  even $C^{\infty}(\R^2)$ is dense in $W^{1,\,p}(G)$ for $1<p<\infty$ if $G$ is a planar Jordan domain \cite{lewis1987}. This allows us to extend $E_{\ez}$ (uniquely) to a bounded linear operator 
from $W^{1,\,p}(\Omega_{\ez})$ to $W^{1,\,p}(\Omega)$. Thus the claim of Theorem~\ref{inner extension} follows. 
\end{proof}

\subsection{Proof in the general case}\label{ss:general}
In this subsection, we prove the necessity of \eqref{eq:extcharcompl}
in the general case, 
where $\Omega$ is a bounded  simply connected $W^{1,p}$-extension domain. 

First of all, $\Omega$ is necessarily  $J$-John, where the constant $J$ 
depends only on $p$ and the norm of the extension operator $\|E\|$  for 
$\Omega,$ see e.g. \cite[Theorem 3.4]{gore1990}.
 Fix $z_1,\,z_2\in \mathbb R^2\setminus \Omega$. Let $\Omega_n=\varphi(B(0,\,1-\frac 1 n))$ 
for $n\ge 2$, where $\varphi\colon \mathbb D\to \Omega$ is a conformal map with $\varphi(0)$ the John center of $\Omega$. 
Let $\wz \Omega_n$ be the complementary domain of $\Omega_n.$ Then
$$ \bigcap_{n=4}^{\infty} \wz \Omega_n=\mathbb R^2\setminus \Omega. $$ By 
Theorem~\ref{inner extension} we know that each $\Omega_n$ is  a $W^{1,p}$-extension domain 
with the norm of the operator only depending on $p,\,J$, and $\|E\|$.
Hence, by Theorem~\ref{neceJordan}, there is a curve $\gamma_n\subset \wz \Omega_n\cup \partial \Omega_n$ 
connecting $z_1$ and $z_2$ so that 
\begin{equation}\label{33oletus} 
\int_{\gamma_n}  \dist(z,\,\partial \Omega_n)^{1-p} \, ds \le C(J,\,\|E\|, \, p) |z_1-z_2|^{2-p}.
\end{equation}

We now proceed as in the proof of Lemma \ref{arsela} to find a limit curve satisfying \eqref{eq:extcharcompl}. Unlike in Lemma \ref{arsela}, here our integrand is not fixed. 
For this reason we repeat the argument with a small modification.

Notice that we may assume that $\elle(\gamma_n)\le C(J,\,\|E\| ,\,p) |z_1-z_2|:=M$ by Lemma~\ref{kvasikonveksi},  and thus also $\gamma_n \subset \overline{B}(z_1,\,M)$. Therefore by Lemma \ref{arsela}, a subsequence, not relabeled, converges uniformly to a limiting curve $\gamma$. We use the constant speed parametrization from $[0,1]$ for each $\gamma_n$, and by taking a further subsequence if necessary also assume that $\elle(\gamma_n) \to l$ as $n \to \infty$.

Since $\varphi$ is continuous up to the boundary (see Remark~\ref{uptobdy}), we have that $\Omega_n$ converges to $\Omega$ (in the Hausdorff distance with respect to the Euclidean metric). 
Thus, for any $t \in [0,1]$ we have that
\begin{equation}\label{eq:boundarydistance}
 \dist(\gamma(t),\partial\Omega) = \lim_{n\to \infty} \dist(\gamma(t),\partial\Omega_n).
\end{equation}
By the (uniform) convergence of $\gamma_n$ to $\gamma$ we also have
\begin{equation}\label{eq:uniformconvergence}
 \lim_{n\to \infty}|\gamma(t) - \gamma_n(t)| = 0.
\end{equation}
Now,  \eqref{eq:boundarydistance} together with \eqref{eq:uniformconvergence} gives
\[
 \dist(\gamma(t),\partial\Omega)^{1-p} = \lim_{n \to \infty}
 \dist(\gamma_n(t),\partial\Omega_n)^{1-p}.
\]
This, when combined with Fatou's Lemma and \eqref{33oletus}, yields
\begin{align*}
  \int_{\gamma} \dist(z,\partial\Omega)^{1-p}\,ds(z)
  & = \int_0^1 \dist(\gamma(t),\partial\Omega)^{1-p}| \gamma'(t)|\,dt
  \le \int_0^1 \dist(\gamma(t),\partial\Omega)^{1-p}l\,dt\\
  &= \int_0^1 \lim_{n \to \infty}
 \dist(\gamma_n(t),\partial\Omega_n)^{1-p}|(\gamma_{n}'(t))| \,dt\\
  & \le \liminf_{n \to \infty} \int_0^1
 \dist(\gamma_n(t),\partial\Omega_n)^{1-p}|(\gamma_{n}'(t))| \,dt\\
  & = \liminf_{n \to \infty}\int_{\gamma_{n}}\dist(z,\partial\Omega_n)^{1-p}\,\d s(z)\\
  & \le C(J,\,\|E\|, \, p) |z_1-z_2|^{2-p}.
 \end{align*}
 We have completed the proof in the general case.

\section{Proof of sufficiency}\label{sec:suf}

In this section we prove the sufficiency of the condition 
\eqref{eq:extcharcompl} in Theorem \ref{thm:main}, but begin with an auxiliary
version. Namely, let $1<p<\hat p<2$ and suppose that $\Omega$ is a  
Jordan
domain with the property that 
there exists a constant $C$ such that for every pair of points 
$z_1,\,z_2\in \R^2 \setminus \overline{\Omega}$ one can find a curve
$\gamma$ joining them in $\R^2 \setminus \overline{\Omega}$ with
\begin{equation}\label{eq:condition}
\int_{\gamma}\dist(z,\,\partial \Omega)^{1-\hat p} \, \d s(z) \le C |z_1-z_2|^{2-\hat p}.
\end{equation}
We claim that $\Omega$ is a $W^{1,p}$-extension domain. 
Write $\widetilde \Omega$ for
the complementary domain of $\Omega$. 

\begin{theorem}\label{prop:Jordancase}
  Let $1 < p <\hat p < 2$ and let $\Omega \subset \R^2$ be a bounded Jordan domain.
  Suppose that for all
 $z_1,z_2 \in \widetilde  \Omega$ there exists a curve $\gamma \subset \widetilde  \Omega$
 joining $z_1$ and $z_2$ such that \eqref{eq:condition} holds.
 Then $\Omega$ is a $W^{1,p}$-extension domain
 and the norm of the extension operator
 only depends on $p,\,\hat p$ and the constant $C$ in \eqref{eq:condition}.
\end{theorem}

The proof of Theorem \ref{prop:Jordancase} is given in several steps. 
In the first step, in the following subsection, we show that 
\eqref{eq:condition} also holds for initial arcs of hyperbolic rays 
$\Gamma\subset \widetilde  \Omega,$ up to a multiplicative constant. 
In the second subsection we introduce shadows of sets and use them to assign a collection of Whitney squares of the domain 
$\Omega$ to each such Whitney square $\wz Q$ of its complementary domain $\wz\Omega$ that satisfies
$\ell(\wz Q)\le 3\diam(\Omega).$ 
In the third subsection we use these Whitney squares to
 construct our extension operator.  The fourth subsection gives a basic estimate that we will use to
 control the extension, the fifth subsection deals with the construction of additional intermediate squares and estimates for them and the sixth one completes the proof.

Eventually, in the final subsection of this section, we prove Theorem \ref{thm:main} 
via Theorem 
\ref{prop:Jordancase} and an approximation and compactness argument. 
For this, it is crucial that the norm of the extension operator in 
Theorem \ref{prop:Jordancase} only depends on $p,\hat p$ and $C$ in 
inequality \eqref{eq:condition} and that a uniform version of \eqref{eq:condition} for some
$\hat p>p$ and for all of our approximating Jordan domains follows from 
\eqref{eq:extcharcompl} by Lemma \ref{selfimprove}; see 
Lemma~\ref{reduktio} below. 

Since we rely on compactness arguments, we do not obtain an explicit form for the extension operator for Theorem \ref{thm:main}. On the other  hand, once we know that $\Omega$ is indeed a $W^{1,\,p}$-extension domain, an explicit extension operator (a version of the Whitney extension operator) can be given 
\cite{hakotu2008},\cite{hakotu2},\cite{sh2006}. We do not see a way to 
directly show that this kind of a concrete procedure works under our 
assumptions. 


\subsection{Transferring the condition to hyperbolic rays}\label{sec:hypray}


According to the Riemann mapping theorem there is a conformal map 
$\widetilde \varphi\colon \mathbb R^2\setminus \overline{\mathbb D} \to 
\widetilde \Omega.$ 
We will refer to this fixed map 
through Subsection~\ref{suffjordan} by $\widetilde \varphi$. 
Since $\partial\Omega$ is a  Jordan curve, 
the Carath\'eodory-Osgood theorem allows us to extend
$\widetilde \varphi$ continuously to the boundary of $\mathbb D$
as a homeomorphism.
We denote the extension still by $\widetilde \varphi$. 
Recall the definition of a hyperbolic ray from Subsection~\ref{sec:hyperbolic}.

\begin{lemma}\label{lma:2}
Assume that \eqref{eq:condition} holds for $\wz \Omega$ for a bounded Jordan domain
$\Omega.$  
Let $z_1\in \partial \Omega$ and $[z_2,z_3]$ be an arc of the hyperbolic
ray $\Gamma\subset \overline{\wz\Omega}$ corresponding to $z_1.$ 
Then 
\begin{align}\label{curve conditionGamma}
\int_{[z_2,z_3]}\, \dist(z,\,\partial \Omega)^{ 1-\hat p}\,\d s(z) \le C'|z_2-z_3|^{2-\hat p},
\end{align}
where $C'$ depends only on $\hat p$ and the constant $C$ in \eqref{eq:condition}.
\end{lemma}
\begin{proof}
By symmetry we may assume that $z_3$ is 
after $z_2$
on $\Gamma$ when one moves towards infinity.
Suppose first that $z_2\neq z_1$. Let $\gamma$ be a curve from  
\eqref{eq:condition}
for the pair $z_2,z_3.$ 
We use the notation from Subsection~\ref{sec:confexterior}; especially, we let $\gamma_k$ be a subcurve of $\gamma$ that joins
the inner and outer boundaries of  $\widetilde \varphi(A(z_1,\,k)),$ provided 
that
$[z_2,z_3]$ intersects at least three such sets. If  $[z_2,z_3]$  is 
contained in the union
of two of these sets, we claim that \eqref{curve conditionGamma} follows from the bi-Lipschitz estimate 
from Lemma~\ref{bilipominaisuus}. 
Indeed, by the definition of 
hyperbolic rays in $\mathbb R^2\setminus \overline{\mathbb D}$, 
$\wz \varphi^{-1}([z_2,\,z_3])$ is contained in 
$B(\wz \varphi^{-1}(z_3)),\frac 3 4 (|\wz \varphi^{-1}(z_3)|-1))$ if $[z_2,\,z_3]$ is contained in 
the union of two consecutive $\wz \varphi(A(z_1,\,k))$. Write $G=B(\wz \varphi^{-1}(z_3),\frac 4 5 (|\wz \varphi^{-1}(z_3)|-1)).$  Then the version of \eqref{curve conditionGamma} (with $\partial \Omega$ replaced by 
$\partial G$) holds for the radial segment $\wz \varphi^{-1}([z_1,\,z_3])=[\wz \varphi^{-1}(z_2),\,\wz \varphi^{-1}(z_3)]\subset G$ with a constant only depending on $\hat p.$
Hence a change of variable together with Lemma~\ref{bilipominaisuus} applied to $\wz \varphi$ in the set
$G$ ensures us that 
\eqref{curve conditionGamma} holds for $[z_2,\,z_3]$ with a constant 
depending only on $\hat p;$ notice that $\dist(z,\partial \wz \varphi(G))\le \dist(z,\partial \Omega)$ when $z\in \wz \varphi (G)$ since $\wz \varphi$ is a homeomorphism.

Suppose then that $[z_2,z_3]$ intersects at least three of the sets 
$\wz \varphi(A(z_1,\,k)).$
For each $k\in \mathbb Z$ with 
$$|\wz \varphi^{-1}(z_1)- \wz \varphi^{-1}(z_2)|\le 2^{k-1}\le 2^k\le 
|\wz \varphi^{-1}(z_1)- \wz \varphi^{-1}(z_3)|,$$
let
$$Z_{k}= \widetilde \varphi( S^1_{k})\cap \Gamma_{k},$$
where $S^1_{k}$ is the circle centered at $\wz \varphi^{-1}(z_1)$ and with radius 
$3\times 2^{k-2}$.

Fix $k\le 2$ as above. 
According to Lemma~\ref{lengthtransfer},  
\begin{align}\label{apu1}
\dist(\Gamma_{k},\,\partial \Omega)\sim \dist(Z_{k},\,\partial \Omega)
\end{align}
and 
\begin{align}\label{apu2}
\elle(\Gamma_{k})\sim \dist(\Gamma_{k},\,\partial \Omega)
\end{align}
with absolute constants. 
Hence
\begin{align}\label{ineq2}
\int_{\Gamma_{k}}\, \dist(z,\partial \Omega)^{ 1-\hat p}\,\d s(z)\lesssim  \dist(Z_{k},\,\partial \Omega)^{2-\hat p}.
\end{align}

Next we claim that
\begin{equation}\label{ineq1}
\dist(Z_{k},\, \partial\Omega)\gtrsim
\dist(\gamma_{k},\, \partial\Omega)
\end{equation}
for some absolute constant.
Indeed let $B_k=\overline{B}(Z_k,\,\frac 1 4  \dist(Z_{k},\,\partial \Omega))$. 
If $\gamma_k\cap B_k\neq\emptyset$, then by the triangle inequality we obtain 
the claim. For the other case, notice that $B_k$ is a $4$-Whitney-type set, 
and then by Lemma~\ref{whitney preserving},  $\wz\varphi^{-1}(B_k)$ is of 
$\lambda$-Whitney-type for some absolute constant $\lambda$. 
Hence \eqref{equat90} gives
\begin{equation}\label{equat30}
\dist(\wz\varphi^{-1}(Z_k),\,\mathbb S^1)\sim \diam(\wz\varphi^{-1}(B_k))
\end{equation}
with an absolute constant. 
By the geometry of $A(z_1,\,k)$ in $\mathbb R^2\setminus \overline{\mathbb D}$, we have
$$\dist(\wz\varphi^{-1}(Z_k),\,\wz\varphi^{-1}(\gamma_k))\le 4 \dist(\wz\varphi^{-1}(Z_k),\,\mathbb S^1)$$
and 
$$\diam(\wz\varphi^{-1}(\gamma_k))\ge 2 \dist(\wz\varphi^{-1}(Z_k),\,\mathbb S^1). $$
Hence with the version of \eqref{condition of lower bound} for 
$\mathbb R^2\setminus \overline{\mathbb D}$ and \eqref {equat30} we conclude 
that
$${\rm Cap}( \wz\varphi^{-1}(B_k),\,\wz\varphi^{-1}(\gamma_k),\,\mathbb R^2\setminus \overline{\mathbb D})\ge \delta(\lambda)>0, $$
and the conformal invariance of capacity gives
$${\rm Cap}( B_k,\, \gamma_k,\,\wz \Omega)\ge \delta(\lambda). $$
This estimate together with Lemma~\ref{inner capacity} yields
$$\dist(B_k,\, \gamma_k)\le C(\lambda) \diam(B_k). $$
We then conclude \eqref{ineq1} also in this case by the definition of $B_k$ and the triangle inequality; indeed
\begin{align*}
\dist(\gamma_{k},\, \partial\Omega)&\le \dist(B_k,\, \partial\Omega)+ \dist(B_k,\, \gamma_k)\\
&\le  \dist(B_k,\, \partial\Omega) + C(\lambda) \diam(B_k)\\
&\le C(\lambda) \dist(Z_k,\, \partial\Omega).
\end{align*}


By Lemma \ref{lengthtransfer}  
$$\ell( \gamma_{k})\gtrsim \ell(\Gamma_{k})$$
with an absolute constant. 
Then, by \eqref{ineq1}  \eqref{apu1} and \eqref{apu2}, this gives that there is a subcurve $\gamma'_k\subset \gamma_k$ such that
$$\dist(Z_{k},\, \partial\Omega) \gtrsim
\dist(\gamma'_{k},\, \partial\Omega)$$
and 
$$\elle( \gamma'_{k})\sim \elle(\Gamma_{k})$$
with  absolute constants. 
By combining this with \eqref{apu1} and \eqref{apu2}  we conclude that
\begin{align}\label{ineq3}
\int_{\gamma_{k}}\,\dist(z,\partial \Omega)^{ 1-\hat p}\,\d s(z)\gtrsim  \int_{\gamma'_{k}}\,\dist(z,\partial \Omega)^{ 1-\hat p}\,\d s(z)\gs
\dist(Z_{k},\,\partial \Omega)^{2-\hat p}.
\end{align}
Now \eqref{ineq2} and \eqref{ineq3} give us the inequality
\begin{equation}\label{hyvak}
\int_{\Gamma_{k}}\, \dist(z,\partial \Omega)^{ 1-\hat p}\,\d s(z) \le C(\lambda)\int_{\gamma_{k}}\,\dist(z,\partial \Omega)^{ 1-\hat p}\,\d s(z).
\end{equation}

Let us consider the remaining values of $k.$ If $k\ge 2,$ then
$A(z_1,k)$ is a full annulus and of $8$-Whitney type.
Especially, 
\begin{equation} \label{trivlisays}
\dist(z,\partial \Omega)\sim \dist(w,\partial \Omega)
\end{equation} for all $z,w\in \wz \varphi(A(z_1,k))$ since this set is
of $\lambda$-Whitney type for an absolute $\lambda$ by Lemma \ref{whitney preserving}.  Moreover, $\elle(\wz \varphi^{-1}(\gamma_k)\ge \elle(\wz \varphi^{-1}(\Gamma_k)$ since the former
crosses $A(z_1,k)$ and the latter is a radial segment. Hence the bi-Lipschitz estimate 
from Lemma~\ref{bilipominaisuus} gives that $\elle(\Gamma_k)\le C\elle(\gamma_k)$ with an absolute constant and 
\eqref{hyvak} follows from \eqref{trivlisays}.
The only remaining values of $k$ to
consider are those potential $k$ with 
$$2^{k-1}\le |\wz \varphi^{-1}(z_1)- \wz \varphi^{-1}(z_3)| \le 2^k$$
or 
$$2^{k-1}\le |\wz \varphi^{-1}(z_1)- \wz \varphi^{-1}(z_2)| \le 2^k.$$
For such $k,$ \eqref{ineq2} still holds and Lemma~\ref{linearmap} shows
that $\dist(Z_{k},\,\partial \Omega)\sim \dist(Z_{k-1},\,\partial \Omega)$ 
with absolute constants. 
By our assumption, $[z_2,z_3]$  is not contained in the union of two of our 
sets  $\wz \varphi(A(z_1,\,k))$,
and hence these additional integrals over $\Gamma_k$ are controlled by the 
earlier terms for which \eqref{ineq3} holds. 

We conclude from the previous paragraph and \eqref{hyvak} that summing over
$k$ together with the first paragraph of the proof yields 
\eqref{curve conditionGamma} when $z_1\neq z_2.$

Finally if $z_1=z_2$ we deduce \eqref{curve conditionGamma} by picking $w_j\in[z_1,\,z_3]\cap \wz \Omega$ with $w_j\to z_1$ and by applying the conclusion 
from the case $z_1\neq z_2$ (to $[w_j,\,z_3]$) and the monotone 
convergence theorem. 
\end{proof}

\subsection{Shadows of  Whitney-type sets}\label{assign}

Let $\Omega$ be a Jordan domain whose complementary domain $\wz \Omega$ 
satisfies \eqref{eq:condition}. According to \cite{la1985} (also see \cite[Lemma 2.1]{sh2010}), $\wz \Omega$ is then quasiconvex with a constant that
only depends on $\hat p$ and the constant $C$ in \eqref{eq:condition}.
Consequently, by the second part of Lemma \ref{arsela}, also the complement of $\Omega$ is quasiconvex with the same constant.  We conclude
from Lemma \ref{btchar} that 
$\Omega$ is $J$-John, where the 
John constant $J$ depends only on $\hat p$ and the constant $C$ in 
\eqref{eq:condition}. We fix a John center $x_0$ for $\Omega$ and a conformal
map $\varphi \colon \mathbb D \to \Omega$ with $\varphi(0)=x_0.$ 
By the Carath\'eodory-Osgood theorem $\varphi$  extends homeomorphically up 
to the boundary and we refer also to the
extension by $\varphi.$ Our map $\varphi$ will be fixed through 
Subsection~\ref{suffjordan}.
Recall from Subsection~\ref{sec:hypray} that 
$\wz \varphi$ refers to a fixed exterior conformal map. 

We will assign a collection of ``reflected'' squares in the Whitney decomposition 
$W$ of $\Omega$ to
squares $\widetilde Q_i$ in the Whitney decomposition
$\widetilde{W}=\{\widetilde{Q}_i\}$ of the complementary domain $\widetilde{\Omega}.$ 
This will actually only be needed for those  $\wz Q_i$ for which
$\ell(\wz Q_i)\le 3\diam(\Omega).$ The construction of our extension operator
will then rely on these squares. 
We continue under the assumption 
that $\wz \Omega$ satisfies  \eqref{eq:condition} and with the above $\varphi,
\wz \varphi.$
In what follows, we usually use the notation 
$\widetilde{A}$ to indicate that the set in question is contained in 
$\widetilde {\Omega}$.

Given a set $\wz A\subset \wz \Omega$, we consider all the hyperbolic 
rays in 
$\widetilde{\Omega}$ 
passing through $\widetilde A$, and define the {\it shadow} $S_{\wz \Omega}(\wz A)$ as the set of 
all the points where these rays hit the boundary $\partial\Omega$. 
Equivalently, by the invariance of hyperbolic rays
$$S_{\wz \Omega}(\wz A)=\wz \varphi(\pi_r(\wz \varphi^{-1}(\wz A))),$$
where $\pi_r$ is the 
radial projection  to the unit circle. 



Similarly, we define $S_{\Omega}(A)$
for $A \subset \Omega$, with the difference that the hyperbolic rays now
begin from $\varphi(0).$ Then 
$$S_{\Omega}(A)=\varphi(\pi_r(\varphi^{-1}(A))).$$


When there is no risk of confusion, we will drop the subindices and simply 
write $S(\cdot)$ for the respective shadow. 

We have the following properties.

\begin{lemma}\label{shadow estimate}
Let 
$A\subset \Omega$ be a closed 
$\lambda$-Whitney-type set. Then $S(A)$ is connected and 
$$\diam_{\Omega}(S(A ))\sim_{J} \diam(S(A )) \sim_{J,\,\lambda} \diam(A ), $$
where the constant $J$  is the John constant.
Moreover, for any fixed $M\ge 1$ and each closed
$\lambda$-Whitney-type set $\widetilde A $ in the exterior 
domain $\wz \Omega$ of $\Omega$ with
$$\diam(\widetilde A )\leq M \diam(\Omega),$$
$S(\wz A)$ is connected and
$$\diam(S(\wz A))\ge c(\lambda,M) \diam(\wz A).$$
\end{lemma}

\begin{proof} Let us begin with the case $A\subset \Omega.$
By the definition of Whitney-type sets,
$A$ is connected and thus also $\varphi^{-1}(A )$ is 
connected. 
Therefore, 
$\varphi^{-1}(S(A )) =\pi_r(\varphi^{-1}(A )) $ is connected, 
and hence so is $S(A )$. 

Next, by Lemma~\ref{whitney preserving}, $\varphi^{-1}(A)$ is a 
$\lambda'$-Whitney-type set with $\lambda'=\lambda'(\lambda)$.
Moreover,
\[\dist(\varphi^{-1}(A),\varphi^{-1}(S(A)))=\dist(\varphi^{-1}(A),\pi_r(\varphi^{-1}(A))=\dist(\varphi^{-1}(A),\partial \mathbb D)
 \]
and hence the conformal capacity between $\varphi^{-1}(S(A))$ and 
$\varphi^{-1}(A )$ in $\mathbb D$ is bounded from below by a positive constant depending only on $\lambda$; see  
\eqref{condition of lower bound}. By conformal invariance of capacity, also 
$${\rm Cap}(A,\,S(A),\,\boz)\ge \delta(\lambda).$$
Let us prove that
$C(\lambda)\diam(S(A )) \ge \diam(A ).$ 
By the monotonicity of capacity we have
\begin{equation}\label{reduktioo}
\delta(\lambda)\le {\rm Cap}(A,\,S(A),\,\boz) \le {\rm Cap}(A,\,S(A),\, \mathbb R^2),
\end{equation}
which by Lemma~\ref{inner capacity} yields that
\begin{equation}\label{equat33}
{\dist(A,\, S(A))}\le C(\lambda) \diam(S(A)).
\end{equation}
Hence, by the definition of Whitney-type sets,
\begin{equation}\label{uunohdus} 
\diam(A) \ls \dist(A,\,\partial \Omega) \le  {\dist(A,\, S(A))}\lesssim \diam(S(A)),\end{equation}
with  constants depending only on $\lambda$. 

Since $\Omega$ is John, hyperbolic 
rays are John curves by Lemma \ref{hsjohn}. 
Then  for each hyperbolic ray $\Gamma\subset \Omega$ ending at $w\in \partial\Omega$ with $\Gamma\cap A \neq \emptyset$, the property \eqref{equat90} of $\lambda$-Whitney-type sets and the definition of John curves give
$$\dist(w,\,A )\le C(J,\,\lambda) \dist(A,\,\partial \Omega)\le C(J,\,\lambda) \diam(A).$$
Thus
$$\diam(S(A))\le C(J,\,\lambda)\diam(A),$$ 
and hence, by \eqref{uunohdus}, we 
can find another constant $C(J,\lambda)$ 
such that 
\begin{equation}\label{blaaah}
\frac{1}{{C(J,\lambda)}}  \diam(A) \le \diam(S({A})) 
\le {C(J,\lambda)}  \diam(A).
\end{equation}

Finally it follows from Lemma~\ref{inner diameter diameter} that
$$
\diam_{\Omega}(S(A))\sim \diam(S(A))
$$
with constants depending only on $J,$ and the asserted estimate follows by combining this with \eqref{blaaah}. 

The connectivity of $S(\widetilde A)$ follows analogously to the above argument.
Regarding the desired
estimate for  $\diam(S(\wz A)),$ notice first that $\wz A$ contains a disk
$B=B(z_0,r)$ with $r=\frac 1 {\lambda}\diam(\wz A),$ 
since it is of 
$\lambda$-Whitney type. By the monotonicity of capacity we know that
\begin{equation}\label{trivia}
{\rm Cap}(\wz A,\,\partial \Omega\,,\wz \Omega)\ge {\rm Cap}(\partial B,\,\partial \Omega\,,\wz \Omega\setminus \overline B).
\end{equation}
Next, the M\"obius transformation $\phi:\ z\mapsto \frac {r^2} {z-z_0},$ given in complex notation,
 is bi-Lipschitz with a constant only depending on $\lambda$ in 
$B(z_0,Cr)\setminus B(z_0,r)$ for $C=2\lambda(\lambda+1)$ and $B(z_0,Cr)\setminus B(z_0,r)$ contains an arc of $\partial \Omega$ of diameter at least 
$\lambda r/M.$
We conclude that
$$\dist(\phi(\partial B),\phi(\partial \Omega))\le C(\lambda,M) \diam(\partial 
(\phi(\Omega))).$$ 
Hence  \eqref{capacity of balls} (with $U_0=\phi(B)$) gives
\begin{equation}\label{trivia2}{\rm Cap}(\phi(\partial B),\, \phi(\partial \Omega)\,,\mathbb R^2)= 
{\rm Cap}(\partial (\phi(B)),\, \partial (\phi(\Omega))\,,\phi(\mathbb R^2\setminus(\overline B\cup \overline \Omega)))\ge \delta(\lambda,M).
\end{equation}
Monotonicity, conformal invariance of capacity and \eqref{trivia},\eqref{trivia2} 
allow us to conclude that
\begin{equation}\label{trivia3}
{\rm Cap}(\wz A,\,\partial \Omega\,, \mathbb R^2)\ge {\rm Cap}(\wz A,\,\partial \Omega\,,\wz 
\Omega)\ge \delta(\lambda,M).
\end{equation}
Now,
Lemma~\ref{inner capacity} together with conformal invariance of capacity
and \eqref{trivia3}
gives 
$$\dist(\wz \varphi^{-1}(\wz A),\,\partial \mathbb D)\le C(M,\,\lambda).$$ 
Since $\wz \varphi^{-1}(\wz A)$ is of $C(\lambda)$-Whitney-type by 
Lemma~\ref{whitney preserving}, we conclude that
$$\diam(\wz \varphi^{-1}(\wz A))\sim_{\lambda} \dist(\wz \varphi^{-1}(\wz A ),\,\partial \mathbb D)\le C(M,\lambda). $$
This together with the version of \eqref{condition of lower bound} for 
$\mathbb R^2\setminus\overline{\mathbb D}$ and conformal invariance imply that
$$\delta(\lambda,\,M) \le {\rm Cap}(\wz \varphi^{-1}(\wz A ),\,\wz \varphi^{-1}(S(\wz A )),\,\mathbb R^2\setminus\overline{\mathbb D})={\rm Cap}(\wz A,\,S(\wz A ),\,\wz \Omega).$$ By monotonicity of capacity we further conclude that
$$\delta(\lambda,\,M)\le {\rm Cap}(\wz A,\,S(\wz A ),\,\mathbb R^2).$$
This estimate is the analog of \eqref{reduktioo} and hence we may complete
the argument exactly as in the case of $\Omega$ above. 
\end{proof}



The following lemma associates a Whitney square of $\Omega$ to a given 
closed boundary arc. 

 \begin{figure}
   \centering
\includegraphics[width=0.8\textwidth]{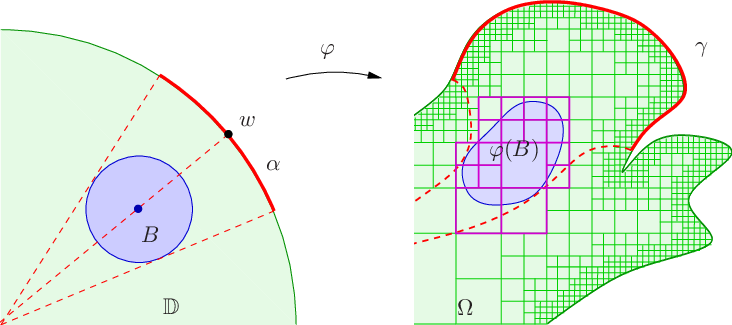}
     \caption{The set $B \subset \mathbb D$ is chosen to be a Whitney-type 
set whose shadow is exactly $\alpha$. Since $\varphi(B)$ is also of 
Whitney-type, 
there are at most a fixed number of Whitney squares intersecting it. Therefore 
one of these squares must have a large shadow.
     }
   \label{fig:shadowball}
 \end{figure}

\begin{lemma}\label{existence}
For each closed nondegenerate
subarc $\gamma\subset \partial \Omega,$ there exists a 
Whitney square $Q\in W$ 
satisfying
\begin{align}\label{sameshadow}
\diam(S({Q})) \le {C(J)} \diam(\gamma) , 
\end{align}

\begin{align}\label{sameinter}
\diam(\gamma) \le  {C(J)}\diam(S(Q)\cap \gamma), 
\end{align}
and
\begin{align}\label{kuutionkokoet}
\dist(Q,\gamma)\le C(J)\diam(\gamma).
\end{align} 
Here $C(J)$ depends only on $J.$
\end{lemma}

\begin{proof}
Given a closed nondegerate subarc $\gamma,$ let
$\alpha= \varphi^{-1}(\gamma).$ 
Suppose first that $\elle(\az)> \frac 1 2. $
By Lemma~\ref{John subdomain}, $\varphi$ is quasisymmetric with respect to the inner distance of $\Omega$ with $\eta$ only depending on $J.$ 
Pick $z_1,\,z_2\in \az$ such that
$$\dist_{\Omega}(\varphi(z_1),\,\varphi(0))=\dist_{\Omega}(\varphi(0),\,\gamma)$$
and
$$|z_1-z_2|=\frac 1 4.$$ Recall that $\varphi(z_i)$ is rectifiably joinable, say, to
$\varphi(0)$ by Remark \ref{muistutus} for $i=1,2.$
Since $\varphi$ is homeomorphic up to the boundary, we may pick points $w_1^j,w_2^j$ along
these rectifiable curves so that $\dist_{\Omega}(w_1^j,\varphi(0))$ tends to $\dist_{\Omega}(z_1,\varphi(0)),$  $\dist_{\Omega}(w_1^j,w_2^j)$ tends to $\dist_{\Omega}(\varphi(z_1),\varphi(z_2)),$  $\varphi^{-1}(w_1^j)$ tends to $z_1$ and  $\varphi^{-1}(w_2^j)$ tends to $z_2.$
Now 
$$|0-\varphi^{-1}(w_1^j)|\le 8|\varphi^{-1}(w_1^j)-\varphi^{-1}(w_2^j)|$$
for all sufficiently large $j$ and then the quasisymmetry of $\varphi$ gives the estimate
$$\dist_{\Omega}(\varphi(0),\,w_1^j)\le \eta(8)\dist_{\Omega}(w_1^j,\,w_2^j).$$
By letting $j$ tend to infinity we conclude that
\begin{equation}\label{inequ 202}
\dist_{\Omega}(\varphi(0),\,\gamma)=\dist_{\Omega}(\varphi(z_1),\,\varphi(0))  \le \eta(8)\dist_{\Omega}(\varphi(z_1),\,\varphi(z_2))\le \eta(8)\diam_{\Omega}(\gamma). 
\end{equation}
By the John property, see Lemma~\ref{hsjohn},
for each hyperbolic ray $\Gamma\subset \Omega$ we have
$$ \dist_{\Omega}(\varphi(0),\,\partial \Omega)\ge J  \elle(\Gamma).$$
Then the triangle inequality gives
\begin{equation}\label{equat3}
\dist_{\Omega}(\varphi(0),\,\gamma)\ge \dist_{\Omega}(\varphi(0),\,\partial \Omega)\ge  \frac J 2 \diam(\Omega). 
\end{equation}
Moreover, Lemma~\ref{inner diameter diameter} implies that
\begin{equation}\label{abba}
\diam(\gamma)\sim_J \diam_{\Omega}(\gamma). 
\end{equation}
By combining \eqref{abba}
with \eqref{inequ 202} and \eqref{equat3} we conclude that
$$\diam(\gamma) \ge \frac 1 {C(J)}\diam(\partial \Omega).$$
 Therefore if one chooses a Whitney square $Q$ containing 
$\varphi(0)$, then its shadow is $\partial \Omega$, and \eqref{sameshadow}  
follows; in this case \eqref{sameinter} holds trivially and 
\eqref{kuutionkokoet} follows from 
\eqref{inequ 202} together with \eqref{abba} since $\varphi(0)\in Q.$

When $\elle(\az)\le \frac 1 2,$ denote the midpoint of $\alpha$ by $w,$ 
let 
$$ r=\frac {\sin\left(\frac {\elle(\az)} {2}\right)}{1+2\sin\left(\frac {\elle(\az)} 2\right)} ,\  z=(1-2r) w$$ 
and set $B=\overline{B(z, r)}.$ See Figure~\ref{fig:shadowball}.
Observe that by the assumption $\elle(\az)\le \frac 1 2,$ the set $B$ satisfies
$$2\dist(B,\,\partial \mathbb D)=2r=\diam(B),$$
and is of $2$-Whitney-type, and the radial projection of $B$ is 
precisely $\az$. Moreover, quasisymmetry of 
$\varphi$ easily gives
\begin{equation}\label{bkoko}
\dist(\varphi(B),\gamma)\le C(J)\diam(\varphi(B)).
\end{equation}

Consider the collection $W_B$ of all Whitney squares in $W$ that intersect $\varphi(B).$ 
Since $\varphi(B)$ is a $\lambda$-Whitney-type set by Lemma~\ref{whitney preserving} for 
some absolute constant $\lambda$,
this collection has no more than $N$ squares where $N(\lambda);$ see the discussion 
after Definition~\ref{whitney-type set}. Since $\varphi$ is homeomorphic up to the 
boundary,  the shadow of  $\varphi(B)$ is 
precisely $\varphi(\az)=\gamma$. We claim that the shadow of one of the Whitney  
squares in $W_B,$ call it $Q$, 
satisfies  $$\diam(S(Q)\cap \gamma)\ge \diam(\gamma)/N.$$

Towards this, since $\varphi(B)\subset \bigcup_{Q'\in W_B}Q' $, we have
$\gamma=S(\varphi(B)) \subset  \bigcup_{Q'\in W_B} S(Q')$. 
Suppose that  for every $Q'\in W_B$ we have
$$\diam(S(Q')\cap \gamma)< \diam(\gamma)/N.$$
Recall that $\gamma$ is an arc and that each $S(Q')$ is  connected and hence also an arc. 
Since
$\gamma \subset  \bigcup_{Q'\in W_B} S(Q')$,  we deduce by the triangle inequality that 
$$\diam\left(\gamma \right)\le \sum_{Q'\in W_B}\diam(S(Q')\cap \gamma)<\diam(\gamma).$$
This gives a contradiction, and hence \eqref{sameinter} follows.

Towards \eqref{sameshadow}, first notice that $\varphi(B)$ is of 
$\lambda$-Whitney type for an absolute $\lambda$ by Lemma~\ref{whitney preserving}. 
Also $Q$ as a Whitney square is of $4\sqrt 2$-Whitney type. Since $Q$ intersects $\varphi(B),$ the property \eqref{equat91} of intersecting
Whitney-type sets ensures that 
\begin{equation}\label{equat32}
\diam(Q) \sim_\lambda\diam(\varphi(B)).
\end{equation}
By Lemma~\ref{shadow estimate},
we further have
\begin{equation}\label{samakoko}
\diam(S(Q))\sim_J \diam(Q)
\end{equation}
and
\begin{equation}\label{tokasamakoko}
\diam(\varphi(B))\sim_J \diam(\gamma)
\end{equation}
since $$S(\varphi(B))=\gamma.$$ 
By combining \eqref{samakoko} and \eqref{tokasamakoko} with 
\eqref{equat32} we conclude that 
$$\diam(S(Q))\le C(J) \diam(\gamma),$$
as desired.

Finally, \eqref{kuutionkokoet} follows by combining \eqref{bkoko} with \eqref{equat32}.
\end{proof}

\bigskip

The definition of our extension operator in Subsection~\ref{jordandef}
will rely on the following existence result.


\begin{lemma}\label{referenssi} 
Let $\Omega$ be a Jordan John domain with constant $J.$ There is a constant $C(J)$ that only depends on $J$ so that the following holds.
Given $\wz Q\in \wz W,$ there exists  $Q\in W$ so that
\begin{equation}\label{heijmaaritelma}
\diam (S(Q))/C(J)\le \diam (S(\wz Q))\le C(J) \diam(S(Q)\cap S(\wz Q))
\end{equation}
and 
\begin{equation}\label{heij2maaritelma}
\dist(Q,S(\wz Q))\le C(J)\diam(S(\wz Q)).
\end{equation} 
Moreover, if 
$\ell(\wz Q)\le 3\diam(\Omega),$  then
\begin{equation}\label{referenssi3}
\diam(\wz Q)\le C(J) \diam(Q).
\end{equation}
\end{lemma}
\begin{proof}
Since $\wz Q$ is of $4\sqrt {2}$ -Whitney type,
Lemma \ref{shadow estimate}
shows that $S(\wz Q)$ is a nondegenerate subarc of $\partial \Omega.$ Thus,
by Lemma~\ref{existence}, there exists a Whitney square $Q\in W$ that 
satisfies both
\eqref{heijmaaritelma} and \eqref{heij2maaritelma}
with constants only depending on $J.$  Finally, \eqref{referenssi3} follows from
these properties of $Q$ together with Lemma \ref{shadow estimate}.
\end{proof}

\begin{figure}
 \centering
 \includegraphics[width=0.70\textwidth]{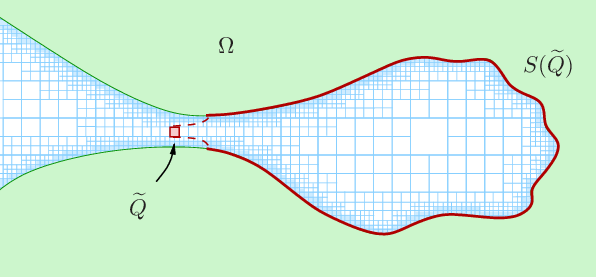}
 \caption{The shadow $S(\widetilde Q)$ of a Whitney square $\widetilde Q$ 
          of the complementary domain $\widetilde \Omega$
          may have much larger diameter than the square in question.}
 \label{fig:bigshadow}
\end{figure}


Notice that a single $Q\in W$ may well satisfy the requirement in Lemma \ref{referenssi} for many distinct 
$\widetilde {Q}$, of different sizes: $S(\wz Q)$ can be much larger in size 
than $\wz Q$; see Figure~\ref{fig:bigshadow}.
We close this subsection with technical lemmas that  will eventually allow us to deal with the distribution of such squares 
$\wz Q.$

\begin{lemma}\label{auxiliary 1}
Let $C\ge 1.$  Suppose that $Q\in W$ and suppose that $\gamma_2,\dots,\gamma_n\subset S(Q)$ are pairwise disjoint arcs so that 
$$\diam(S(Q)) \le {C}\diam(\gamma_j)$$ for each
$1\le j\le n.$
Then $n\le N$, where $N$ depends only on $C$ and the John constant $J$ of $\Omega$. 
\end{lemma}

\begin{proof}
Let $\gamma_1,\dots,\gamma_n$ be pairwise disjoint arcs contained in $S(Q)$ so that 
$\diam(S(Q))\le C\diam(\gamma_j)$ for each $1\le j\le n.$
In order to bound
$n$ it suffices to associate to each $\gamma_j$ a disk $B_j$ of radius 
$r\ge \diam(S(Q))/C'$
so that these disks are pairwise disjoint and all have distance to 
$S(Q)$ no more than
$C'\diam(S(Q)),$ for a constant $C'$ only depending on $C,J.$


Let $w_j$ be the midpoint of   $\varphi^{-1}(\gamma_j),$  
$$ r_j=\frac {\sin\left(\frac {\elle(\varphi^{-1}(\gamma_j))} {2}\right)}{1+2\sin\left(\frac {\elle(\varphi^{-1}(\gamma_j))} 2\right)},\ z_j=(1-2 r_j) w_j$$
 and set $B_j=\overline {B}(z_j,\, r_j).$ Then  the radial projection of $B_j$ is precisely $\varphi^{-1}(\gamma_j) $
 and each $B_j$ is of 4-Whitney type.
Since the arcs $\gamma_j$ are pairwise disjoint, so are also $\varphi^{-1}(\gamma_j)$ and consequently also the sets $B_j$. 
Then the sets $\varphi(B_j)$ are also pairwise disjoint. 
From Lemma~\ref{shadow estimate}   
it follows that 
$$\diam(\varphi(B_j))\ge C(J)\diam(\gamma_j)$$
and  (by \eqref{equat33}) 
$$\dist(\varphi(B_j),S(Q))\le \dist(\varphi(B_j),\gamma_j)\le C(J)\diam(\gamma_j)\le C(J)\diam(S(Q)).$$
The claim follows by recalling that a $\lambda$-Whitney-type set $A$ contains a disk of radius
$\frac {1}{\lambda}\diam(A)$ and that  $C\diam(\gamma_j)\ge \diam(S(Q));$
the sets $\varphi(B_j)$ are of $\lambda$-Whitney type for an absolute $\lambda$  by 
Lemma~\ref{whitney preserving}. 
\end{proof}


For a Whitney-type set $\wz A\subset \wz \Omega$ and a hyperbolic 
ray $\Gamma$ with
$\Gamma\cap \wz A\neq \emptyset,$ corresponding
to a point $z\in \partial \Omega,$ we define the 
{\it tail} of $\Gamma$ with respect to $\wz A $ to be the arc of 
$\Gamma$ between $z$ and $\wz A$, that is $\Gamma_{z,w}\subset \Gamma$ with
$w$ the first point in $\wz A$ when travelled towards infinity from $z.$ 
Denote this set by $T_{\wz \Omega}(\Gamma,\wz A).$


\begin{lemma}\label{sum estimate}
Let $\wz A\subset \wz \Omega$ be a closed $\lambda$-Whitney-type set so that
$\wz \Omega\setminus \wz A$ is connected and let
$\Gamma$ be a hyperbolic ray with $\Gamma\cap \wz A\neq \emptyset.$
Set $\wz W(\wz A,\Gamma)=\{\wz Q_j\in \wz W:\ \wz Q_j\cap T_{\wz \Omega}(\Gamma,\wz A)\neq
\emptyset\}.$
Then
\[
\sum_{\wz Q_j\in \wz W(\wz A,\Gamma)} \ell(\widetilde Q_{j})^{2-\hat p}\leq C\, \diam(S(\wz A))^{2-\hat p}, 
\]
where $C$ depends only on $\hat p,\lambda$ and the constant in \eqref{eq:condition}.
\end{lemma}

In order to prove this, we need an auxiliary lemma and a definition.


We define the tail of $\wz A$ by setting
$$T_{\wz \Omega}(\wz A)=\{y\in \wz \Omega\mid y\in T_{\wz \Omega}(\Gamma,\wz A) \mbox{ for some hyperbolic ray } \Gamma\}.$$
Equivalently, 
$$T_{\wz \Omega}(\wz A)=\wz \varphi(T_{\mathbb R^2 \setminus \overline {\mathbb D}}(\wz \varphi^{-1}(\wz A))).$$
When there is no danger of confusion, we will simply write $T$ 
instead of $T_{(\cdot)}$. 

We need an estimate for the sizes of those Whitney squares that intersect a 
given tail. Such estimates follow rather easily in the complement of the disk,
see Figure \ref{fig:diamcomp}, but our exterior domain case requires work.

\begin{lemma}\label{auxiliary 2}
Let $\wz A\subset \wz \Omega$ be a closed $\lambda$-Whitney-type set with
$\diam(\wz A)\leq 3\diam(\Omega).$ 
Assume additionally that $\wz \Omega\setminus \wz A$ is connected. 
Let
$\wz Q\in \wz W$ satisfy $\wz Q\cap T(\wz A)\neq\emptyset.$ Then 
$$\ell(\wz Q)\le C(\lambda) \diam(S(\wz A)). $$
\end{lemma}

\begin{figure}
   \centering
\includegraphics[width=0.7\textwidth]{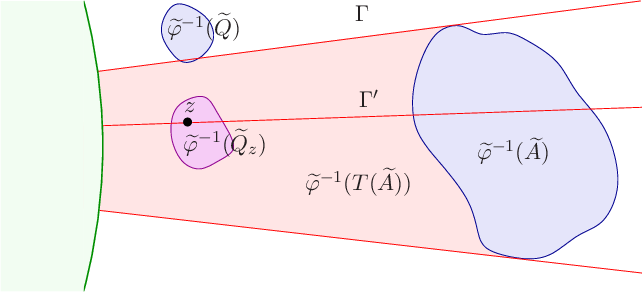}
    \caption{In the case $\diam(\widetilde \varphi^{-1}(\widetilde Q)) < c_1 \diam(\tilde \varphi^{-1}(S(\widetilde A))$ we argue using an extra
    Whitney-type set $\wz \varphi^{-1}(\widetilde Q_z) \subset \widetilde \varphi^{-1}(T(\widetilde A))$ of roughly the same size as $\widetilde \varphi^{-1}(\widetilde A)$ that is also near $\widetilde \varphi^{-1}(\widetilde A)$.
    }
   \label{fig:diamcomp}
 \end{figure}

\begin{proof} 
Fix $\wz Q\in \wz W$ with $\wz Q\cap T(\wz A)\neq\emptyset.$
We may assume that $\lambda \ge 4\sqrt 2$ so that also $\wz Q$ is of
$\lambda$-Whitney type. Let us first prove that 
\begin{equation}\label{equat4}
\diam(\wz Q)\ls\diam(\Omega),
\end{equation}
with a constant depending only on $\lambda.$ 

Towards this claim, recall from the definition of $\lambda$-Whitney-type
that there exists a disk
$$B\left(z_0,\,\frac 1 {\lambda}\diam(\wz A)\right)\subset \wz A. $$
Next, by \eqref{same capacity} we have
\begin{equation}\label{lisays}
{\rm Cap}(\wz A,\, \partial {\Omega} ,\, \wz \Omega) ={\rm Cap}(\partial \wz A,\, \partial {\Omega} ,\, \wz \Omega\setminus \wz A). 
\end{equation}
We continue by arguing as in the proof of Lemma~\ref{shadow estimate}.

Since the M\"obius transformation 
$\phi\colon z\mapsto \frac {\diam(\wz A)^2} {(z-z_0)}$ is 
$C(\lambda)$-bi-Lipschitz  in the set 
$$B(z_0,(2+\lambda)\diam(\wz A))\setminus B(z_0, \diam(\wz A)/\lambda)$$ and 
this set contains both $\partial \wz A$ and an arc of $\partial \Omega$ of
diameter at least $\diam(\wz A)/3,$
we have that 
$$\dist(\phi(\wz A),\,\phi(\partial \Omega))\le 
C'(\lambda)\diam(\phi(\partial \Omega)).$$
Hence \eqref{capacity of balls} (with $U_0=\mathbb R^2\setminus \phi(\wz A)$)
gives the estimate 
$${\rm Cap}(\phi(\partial \wz A),\, \phi(\partial {\Omega}),\, \phi(\wz \Omega\setminus \wz A))\ge \delta(\lambda).$$
Thus 
\begin{equation}\label{inequat 18}
{\rm Cap}(\wz A,\,\partial \Omega,\,\wz \Omega\setminus \wz A)
\ge \delta(\lambda)
\end{equation}
by the conformal invariance of capacity; notice that $\phi$ is conformal in 
the ring domain $\wz \Omega\setminus \wz A$. 

Next, as also $\wz \varphi^{-1}$ preserves conformal capacity, monotonicity together with the 
inequalities
\eqref{lisays} and \eqref{inequat 18} gives
\begin{equation}\label{uusilisays}
{\rm Cap}(\wz \varphi^{-1}(\wz A),\,\partial\mathbb D,\,\mathbb R^2)\ge {\rm Cap}(\wz \varphi^{-1}(\wz A),\,\partial\mathbb D,\,\mathbb R^2\setminus \overline{\mathbb D})\ge \delta(\lambda).
\end{equation}
Hence Lemma~\ref{inner capacity} and the fact that $\varphi^{-1}(\wz A)$ is of $\lambda'$-Whitney-type
by Lemma~\ref{whitney preserving} yield
\begin{equation}\label{yinkorjaus}
\dist(\wz \varphi^{-1}(\wz A),\, \partial \mathbb D)\le C(\lambda).
\end{equation}
By \eqref{equat90} and
the fact that $\wz \varphi^{-1}(T_{\wz \Omega}(\wz A))= T_{\mathbb R^2\setminus 
\overline {\mathbb D}}\wz \varphi^{-1}(\wz A)$ we deduce that 
\begin{equation}\label{lisays2}
\dist(w,\, \mathbb D)\le C(\lambda)
\end{equation}
for every $w\in \wz \varphi^{-1}(T(\wz A)).$ Since 
$\wz Q\cap T(\wz A)\neq \emptyset$ and since $\wz \varphi^{-1}(\wz Q)$ is also 
of $\lambda'$-Whitney type by
Lemma~\ref{whitney preserving}, \eqref{lisays2} gives us the estimate
\begin{equation}\label{lisays3}
\diam(\varphi^{-1}(\wz Q))\le C(\lambda) \dist(\varphi^{-1}(\wz Q),\, \mathbb D)\le C(\lambda).
\end{equation}
Now monotonicity and conformal invariance of capacity together with \eqref{condition of lower bound} and \eqref{lisays3} yield
$${\rm Cap}( \wz Q ,\,\partial\Omega,\,\mathbb R^2)\ge{\rm Cap}( \wz Q ,\,\partial\Omega,\,\wz \Omega)={\rm Cap}(\varphi^{-1}(\wz Q),\,\partial\mathbb D,\,\mathbb R^2\setminus \overline{\mathbb D})\ge \delta(\lambda).$$
Since $\wz Q$ is a Whitney square,  \eqref{equat4} follows from
this by Lemma~\ref{inner capacity}.

Recall again that the preimages of both
$\wz A$ and $\wz Q$ are of $\lambda'$-Whitney-type with  
$\lambda'=\lambda'(\lambda)$. Hence, if $\wz Q\cap \wz A\neq \emptyset,$ then  $\ell(\wz Q)\sim \diam(\wz A)$ by \eqref{equat91}, and
our asserted estimate follows from Lemma \ref{shadow estimate}. Hence we may assume that $\wz Q\cap \wz A= \emptyset.$

We prove the claim of the lemma first under 
the additional assumption that
\begin{equation}\label{uusilisays1}
\diam(\wz \varphi^{-1}(\wz Q))\ge {c_1} \diam( \wz\varphi^{-1}(S(\wz A)))
\end{equation}
where
\begin{equation}\label{equat51}
 c_1=\min\left\{\frac 1 9,\,\frac 1 {6 \lambda'},\,\frac  1 {8 \lambda'^2}\right\}.
\end{equation}
To begin, since $\wz \varphi^{-1}(\wz A)$ is of $\lambda'$-Whitney type,
\eqref{yinkorjaus} together with Lemma~\ref{shadow estimate} gives
\begin{equation}\label{uusilisays2}
\diam(\wz \varphi^{-1}(\wz A))\le C(\lambda,\lambda')\diam (S(\wz\varphi^{-1}
(\wz A))).
\end{equation}

Since $\wz Q\cap T_{\wz \Omega}(\wz A)\neq \emptyset,$ we can pick a point $z\in \wz \varphi^{-1}(\wz Q)\cap \wz \varphi^{-1}(T_{\wz \Omega}(\wz A))=\varphi^{-1}(\wz Q)\cap T_{\mathbb D\setminus 
\overline{\mathbb D}}(\wz \varphi ^{-1}(\wz A)).$ Then the hyperbolic ray (radial line) through $z$ intersects $\wz \varphi^{-1}(\wz A).$ Since $\wz \varphi^{-1}(\wz A)$ is of $\lambda'$-Whitney type,
the length of the segment of this radial line between $\partial \mathbb D$ and $\wz \varphi^{-1}(\wz A)$ is no more than $C(\lambda')\dist(\wz \varphi^{-1}(\wz A),\,S(\wz \varphi^{-1}(\wz A))).$ Let $I$ be the
subsegment between $z$ and $\wz \varphi^{-1}(\wz A).$
Then 
$$\elle(I)\ls  C(\lambda')\dist(\wz \varphi^{-1}(\wz A),\,S(\wz \varphi^{-1}(\wz A)))=C(\lambda')\dist(\wz \varphi^{-1}(\wz A),\,\partial \mathbb D).$$
Recalling that $\wz \varphi^{-1}(\wz A)$ is of $\lambda'$-Whitney type with $\lambda'=\lambda'(\lambda),$ this in combination with \eqref{uusilisays1} and \eqref{uusilisays2}  
gives
\begin{equation}\label{qhpreestimate}
\elle(I)\ls \diam(\varphi^{-1}(\wz A))\ls \diam(\wz \varphi^{-1}(\wz Q))
\end{equation}
with constants only depending on $\lambda.$ Since $\wz \varphi^{-1}(\wz Q)$ is of $\lambda'$-Whitney type, we deduce that
$\elle(I)\ls |z|-1$. Because $I$ is also a radial segment with $z$ the closest point to $\partial \mathbb D,$ it follows that the number of the Whitney squares of $\mathbb R^2\setminus \overline{\mathbb D}$ 
that intersect $I$ is at most $N=N(\lambda').$ Recalling that $\lambda'=\lambda'(\lambda)$ we conclude that we can join 
$\wz \varphi^{-1}(\wz A)$ and $\wz \varphi^{-1}(\wz Q)$ by a chain of no more than $N(\lambda)$ Whitney squares.
Then by Lemma~\ref{whitney preserving} and the fact that a $\hat \lambda$-Whitney type set intersects at most $N(\hat \lambda)$ Whitney squares, there also exists a chain of no more than $N'=N'(\lambda)$  Whitney squares of $\wz \Omega$ joining $\wz A$ and $\wz Q$.
Since  both $\wz A$ and $\wz Q$ are of $\lambda$-Whitney type, their diameters are comparable by \eqref{equat91} to diameters of those Whitney squares that intersect them and the diameters of any
two consecutive Whitney squares in our chain are comparable. It follows that
$\diam(\wz Q)\le C(\lambda) \diam(\wz A).$ 
By Lemma~\ref{shadow estimate} and the assumption that
$$\diam(\wz A)\le 3\diam(\Omega),$$
we conclude that
$\diam(\wz Q)\le C(\lambda) \diam(\wz A) \le C(\lambda) \diam(S(\wz A)).$

We are left to consider the case where 
\begin{equation}\label{equat57}
\diam(\wz \varphi^{-1}(\wz Q))<  c_1
\diam( \wz\varphi^{-1}(S(\wz A))).
\end{equation} 
If $\wz Q\subset T(\wz A)$ then by 
Lemma~\ref{shadow estimate} with \eqref{equat4}   we have
$$\diam(\wz Q)\lesssim \diam(S(\wz Q))\lesssim\diam(S(\wz A)).$$
If $\wz Q$ is not contained in $T(\wz A)$, let $d=\diam(\wz \varphi^{-1}(\wz Q))$. Let $w_1,w_2$ be the end points of $\wz \varphi^{-1}(S(\wz A)).$ By \eqref{equat57} and \eqref{equat51}, we have that
\begin{equation}\label{equat60}
d<6 \lambda'  d\le\diam ( \wz\varphi^{-1}(S(\wz A)))
\end{equation} 
and it follows that $\wz\varphi^{-1}(\wz Q)$ intersects only one of the hyperbolic rays from $w_1,w_2$ to infinity.
Let $\Gamma$ be this hyperbolic ray. Also let $\Gamma'$ be the hyperbolic ray in 
$\mathbb R^2\setminus \mathbb D$ which intersects 
$\wz \varphi^{-1}(T(\wz A))$ and satisfies 
\begin{equation}\label{equat52}
\dist(\Gamma,\,\Gamma')=2\lambda' d;
\end{equation} the existence of $\Gamma'$ follows from \eqref{equat60}.   Let $z$ be the point on $\Gamma'$ 
with $|z|=1+d.$ See Figure \ref{fig:diamcomp}.
Let $\wz Q_z$ be a Whitney square so that $z\in \wz \varphi^{-1}(\wz Q_z).$
Then $\wz \varphi^{-1}(\wz Q_z)$ is also of $\lambda'$-Whitney type as 
$\wz Q_z$ is of $4\sqrt 2$-Whitney type  and we assumed that 
$\lambda\ge 4\sqrt 2.$
Hence by Definition~\ref{whitney-type set} of $\lambda'$-Whitney type sets, \eqref{equat51} and \eqref{equat57} we conclude that 
\begin{equation} \label{raduga}
\diam( \wz \varphi^{-1}(\wz Q_z ))+\dist( \wz \varphi^{-1}(\wz Q_z ) ,\,\partial \mathbb D)\le \lambda' d+d < \frac 1 {4\lambda'}\diam ( \wz\varphi^{-1}(S(\wz A))),
\end{equation}
where we used the fact that $c_1\le \frac  1 {8 \lambda'^2}\le  \frac  1 {4 \lambda'(\lambda'+1)}.$
 
Next, $\wz \varphi^{-1}(S(\wz A))=S(\varphi^{-1}(\wz A))=\pi_r(\wz \varphi^{-1}(\wz A)),$ where $\pi_r$ is the radial projection. Since $\pi_r$  is a contraction, $\diam(\wz \varphi^{-1}(\wz A))>0$
and $\wz \varphi^{-1}(\wz A)$ is of $\lambda'$-Whitney type,
we have
$$\frac 1 {4\lambda'} \diam ( \wz\varphi^{-1}(S(\wz A)))< \frac 1 {\lambda'} \diam(\wz 
\varphi^{-1}(\wz A)) \le \dist(\wz \varphi^{-1}(\wz A),\,\partial \mathbb D).$$
By combining this estimate with \eqref{raduga} we conclude that, for any point $x\in \wz \varphi^{-1}(\wz Q_z ),$
\begin{equation}\label{equat61}
\dist(x,\,\partial \mathbb D)\le \diam( \wz  \varphi^{-1}(\wz Q_z ) )+\dist( \wz  \varphi^{-1}(\wz Q_z ) ,\,\partial \mathbb D)<\dist(\wz \varphi^{-1}(\wz A),\,\partial \mathbb D);
\end{equation}
especially
$$\wz  \varphi^{-1}(\wz Q_z )\cap \wz \varphi^{-1}(\wz A)=\emptyset.$$
Furthermore, since $$\diam(\wz  \varphi^{-1}(\wz Q_z ))\le \lambda' d,$$
by \eqref{equat60} and \eqref{equat52} we know that $\wz  \varphi^{-1}(\wz Q_z )$ does not 
intersect either of our two hyperbolic rays in $\mathbb R^2\setminus \mathbb D$ from the end points $w_1,w_2$ of 
$\wz \varphi^{-1}(S(\wz A))$. 
This implies that
$$ S(\wz  \varphi^{-1}(\wz Q_z ))  \subset \wz \varphi^{-1}(S(\wz A)).$$
This, together with \eqref{equat61}, yields
$ \wz  \varphi^{-1}(\wz Q_z )  \subset \wz \varphi^{-1}(T(\wz A))$, or equivalently
$\wz Q_z\subset T(\wz A).$
Since $\wz  Q_z\subset T(\wz A)$ and $\diam(\wz Q_z)\lesssim \diam(\Omega)$ by
\eqref{equat4}, Lemma~\ref{shadow estimate} gives
\begin{equation}\label{tropics}
\diam(\wz Q_z)\lesssim \diam(S(\wz Q_z))\lesssim \diam(S(\wz A)).
\end{equation}

Pick $\hat z\in \wz \varphi^{-1}(\wz Q)\cap \Gamma.$ Since $\wz \varphi^{-1}(\wz Q)$ is of $\lambda'$-Whitney type, we have that $|\hat z|-1\sim d$ with a constant only depending on $\lambda'.$
Let $z_1$ be the point on $\Gamma$ with $|z_1|=1+2d$ and let  $z_2$ be a point on
$\Gamma'$ with $|z_2|=1+2d.$
Consider the curve $\gamma$ obtained by concatenation from the part of $\Gamma$ between $\hat z,z_1,$ part of $\Gamma'$ between $z,z_2$ and a shorter one of the circular arcs on $S(0,1+2d)$
joining $z_1,z_2.$ Then the number of Whitney squares of $\mathbb R^2\setminus \overline {\mathbb D}$ intersecting $\gamma$ is at most $N(\lambda').$ We again rely on Lemma \ref{whitney preserving}
and the fact that a $\hat \lambda$-Whitney type set intersects at most $N(\hat \lambda)$ Whitney squares to conclude there also exists a chain of no more than $N'=N'(\lambda)$  Whitney squares of $\wz \Omega$ joining $\wz Q_z$ to $\wz Q$. It follows that $\diam(\wz Q_z)\sim \diam(\wz Q)$ and hence the desired estimate follows from \eqref{tropics}.
\end{proof} 


\begin{proof}[Proof of Lemma~\ref{sum estimate}]
Let $\Gamma$ be a hyperbolic ray that intersects $\wz A.$
Denote by $\Gamma_0$ the tail of $\Gamma$ with respect to $\wz A.$

We claim that $\elle(\Gamma_0)\leq C \diam(S(\wz A))$ with a constant that only
depends on our data: 
$\hat p$ and the constant $C$ in \eqref{eq:condition}. To begin, suppose that 
$\wz Q\in \wz W$ 
intersects $\Gamma_0.$ Then $\wz Q\cap T(\wz A)\neq \emptyset,$ and hence
Lemma~\ref{auxiliary 2} gives
\begin{equation}\label{equat3434}
\ell(\wz Q) \leq C \diam(S(\wz A))
\end{equation}
with a constant that only depends on $\lambda.$ 
%
Next, \eqref{equat3434} yields that
\begin{equation}\label{equat3444} 
\dist(z,\,\partial \Omega)\leq 4\sqrt 2\, C\, \diam(S(\wz A))
\end{equation}
whenever $z\in \Gamma_0.$

By \eqref{equat3444} and  Lemma \ref{lma:2} we have
\begin{align}\label{yksicurve}
\diam(S(\wz A))^{1-\hat p}\elle(\Gamma_0)\leq C \int_{\Gamma_0} \dist(z,\,\partial \Omega)^{1-\hat p} \, \d s(z)\leq C_1  
\ell(\Gamma_0)^{2-\hat p}, 
\end{align}
where $C$ only depends on $\lambda$ and $C_1$ depends only on $\hat p,\lambda$ and on the constant in 
\eqref{eq:condition}.
This together with the assumption that $\hat p>1$ results in
\begin{align}\label{pituus}
\elle(\Gamma_0) \leq C_1^{1/(\hat p-1)} \diam(S(\wz A)). 
\end{align}

By combining \eqref{yksicurve} with  \eqref{pituus} we conclude that 
\begin{align}\label{onecurve}
\int_{\Gamma_0} \dist(z,\,\partial \Omega)^{1-\hat p} \, \d s(z)\leq C_1^{(2-\hat p)/(\hat p-1)}  
\diam(\wz S(\wz A))^{2-\hat p}. 
\end{align}
We now employ \eqref{onecurve} to prove our claim.

Recall that  $\wz W(A,\Gamma)$ consists of those $\wz Q_j\in \wz W$ that
intersect $\Gamma_0.$ Since each Whitney square has at most 20 neighboring
squares, we can distribute the squares in $\wz W(A,\Gamma)$ 
into no more than 21 subcollections $\{\wz W_k\}_{k=1}^{21}$ such that in
each of the subcollections the squares are pairwise disjoint.
Next,  for any two distinct $\wz Q_i,\,\wz Q_j\in \wz W_k$,  
by Lemma~\ref{lma:whitney} we have
$$\frac {11}{10} \wz Q_i\cap \frac {11}{10} \wz Q_j =\emptyset. $$
Clearly, for each $ \wz Q_{j}\in  \wz  W(A,\Gamma)$, we have
$$\mathcal H^{1}\left(\frac {11}{10} \wz Q_{j}\cap \Gamma_0\right)\ge \frac {1}{10}\ell( \wz Q_{j}),$$
where $\mathcal H^{1}$ denotes the $1$-dimensional Hausdorff measure. 
Recall that
$$\ell( \wz Q_{j}) \leq \dist( \wz Q_{j},\,\partial \Omega)\leq 4\sqrt 2\, 
\ell( \wz Q_{j}) .$$ Hence \eqref{onecurve} gives
\begin{align*}
\sum_{\wz Q_{j}\in \wz W(\wz A,\Gamma)} \ell(\wz Q_{j})^{2-\hat p} \, & \ls \sum_{k=1}^{13}\sum_{\wz Q_{j} \in \wz W_k}\int_{\Gamma_0\cap  \frac {11}{10}\wz Q_{j}} \dist(z,\,\partial \Omega)^{1-\hat p} \, ds \\
& \ls \int_{\Gamma_0} \dist(z,\,\partial \Omega)^{1-\hat p} \, ds\ls 
\diam(S(\wz A))^{2-\hat p}.
\end{align*}
\end{proof}

\subsection{Definition of the extension operator in the Jordan case}\label{jordandef}

Recall from Subsection~\ref{assign} that our conformal map $\varphi\colon \mathbb D\to \Omega$ satisfies
$\varphi(0)=x_0,$ 
where $x_0$ is a fixed John center of $\Omega$. 
Let 
$$B_{\Omega}=B(x_0,\,\diam(\Omega)).$$
Then $\overline \Omega\subset B_{\Omega}.$
Recall from Lemma~\ref{lma:whitney} that 
$$\ell(\wz Q)\le \dist(\wz Q,\,\partial \Omega)$$
for each $\wz Q\in \wz W,$ the Whitney decomposition of $\wz \Omega.$ 
Then, if $\wz Q \cap B_{\Omega}\neq \emptyset$, we obtain by definition that
$$\ell(\wz Q)\le \dist(\wz Q,\,\partial \Omega)\le \diam(\Omega).$$ Also,
if $\wz Q'\in \wz W$ is a neighbor of $\wz Q$ with  $\wz Q \cap B_{\Omega}\neq \emptyset,$
then
$$\ell(\wz Q') \le \dist(\wz Q',\,\partial \Omega)\le (1+\sqrt 2)\dist(\wz Q,\,\partial \Omega)\le 3 \diam(\Omega).$$
Hence the side lengths of all the Whitney squares $\wz Q$ that intersect $B_{\Omega}$ and
of all their neighbors are at most $3\diam(\Omega).$ 

Let $C(J)$ be the constant from Lemma \ref{referenssi}.
For each $\wz Q_i\in \wz W$ with $\ell(\wz Q_i)\le 3\diam(\Omega)$ we 
consider the collection $W_i$ of all squares $Q\in W$ that satisfy the conclusions of  Lemma~\ref{referenssi} for this value of $C(J).$ 
Then this collection is non-empty. We have to choose one $Q$ from this collection. Since any choice will work, we may proceed as follows. Recall that  $W$ can be written as $\{Q_1,Q_2,\dots\}.$ We pick the $Q_j\in W_i$ of smallest index $j$ and define
$$\rf(\wz Q_i)=Q_j.$$ 

It may happen that $\rf(\wz Q_i)=\rf(\wz Q_k)$ even when $k\neq i$ and there may well be squares $Q\in W$ for which there is no $\wz Q_i$ with
$\rf(\wz Q_i)=Q.$ In fact, the number of distinct $\wz Q_i$ with $\rf(\wz Q_i)=Q$ is always finite (Lemma \ref{ketjujenpaall} in Subsection~\ref{sectionintermediate})
but we do not have a uniform bound on the number of them. 
Nevertheless, Lemma~\ref{sum estimate} with work would allow us to to control the sum of $\ell(\wz Q_i)^{2-\hat p}$ for the $\wz Q_i$ that
satisfy $\rf(\wz Q_i)=Q.$ However, this would not suffice  for our final estimate, as certain intermediate Whitney squares also come into the estimate. To overcome this, we will eventually prove Lemma~\ref{strong sum estimate}  that takes into consideration also these intermediate squares.

 
Pick a collection of functions $\phi_{i}\in C^{\infty}(\widetilde \Omega)$ so that
each $\phi_i$ is compactly
supported in $\frac {11}{10}\widetilde Q_{i}$,
$|\nabla \phi_{i}|\lesssim \ell(\widetilde Q_{i})^{-1}$, and
\[
\sum_{i}\phi_{i}(x)=1 
\]
for all $x\in \widetilde \Omega.$ 
Then the support of $\phi_i$ and that of $\phi_j$ have no intersection 
unless $\wz Q_i\cap \wz Q_j\neq\emptyset$. 
See \cite{jo1981} for the existence of such a partition of unity $\{\phi_i\}$.

Given $u\in W^{1,\,p}({\Omega})$ and $\wz Q_i\in \wz W$ with 
$\ell(\wz Q_i)\le 3\diam(\Omega),$ we set
\[
a_i=\bint_{\rf (\wz Q_i)}u(z)\,dz =\frac 1 {|\rf(\wz Q_{i})|}\int_{\rf (\wz Q_i)} u(z)\,dz, 
\]
and we define 
$Eu(x)=u(x)$ in $\Omega$
and
\begin{equation}\label{eq:jordanextdefinition1}
E u(x)=\sum_{i}a_i\phi_{i}(x) 
\end{equation}
for $x\in B_{\Omega}\setminus \overline{\Omega}$. 
Here the sum runs over those $i$ for which $\ell(\wz Q_i)\le 3\diam(\Omega).$
We will prove that $\|E u\|_{W^{1,\,p}(B_{\Omega}\setminus {\overline{\Omega}})}
\lesssim\|u\|_{W^{1,\,p}( {\Omega})}$. We have not yet defined $Eu$ on 
$\partial \Omega.$  Since $\partial \Omega$ is of area zero by Lemma \ref{bdyzero}, this is not an issue, and we simply let $Eu(x)=0$ for points in $\partial \Omega.$

\begin{remark} A reader familiar with the extension operator employed in \cite{jo1981} perhaps wonders why we have chosen $\rf(\wz Q_i)$ via the shadow of $\wz Q_i$ instead
of picking a Whitney square $Q$ of diameter comparable to that of $\wz Q_i$ and at distance comparable to the diameter of $\wz Q_i.$ Actually, such a square can be found as $\Omega$ is 
John, but we have not been able to establish useful estimates for the difference of averages over reflections of neighboring squares under this kind of a choice. One should view our
construction of $\rf$ as a kind of reflection via harmonic measure. In fact, the Jordan case in the setting of \cite{jo1981} is that of a quasidisk and for them our choice of $\rf(\wz Q_i)$  
can be checked to conform with the one used in \cite{jo1981}. We will control the above difference of averages via suitable John subdomains of $\Omega.$ In our setting, these subdomains
may well have bad overlaps contrary to what happens in \cite{jo1981} and in our adaptation of this technique in Section~\ref{sec:nec}, see \eqref{peter 3.2}. The key point in what follows will be to obtain control on the overlaps in terms of the squares $\wz Q_i.$ 
\end{remark}

\subsection{Basic estimate}\label{sectionestimate}

In order to estimate $|\nabla Eu|$ for the operator defined in \eqref{eq:jordanextdefinition1} we need control on the differences of the averages of $u$ over pairs of Whitney squares.
Towards this, denote by $\widehat{|\nabla u|}$ the zero extension of $|\nabla u|$, and by $M$ the Hardy-Littlewood
maximal operator. 

\begin{lemma}\label{special case}
Given distinct Whitney squares $Q,\,Q'\subset \Omega$ such that
\begin{equation} \label{oletus} 
\dist_{\Omega}(S(Q),\,S(Q'))\lesssim  \ell(Q)\sim \ell(Q'),
\end{equation}
we have 
 \[
\left|\bint_{{Q}}u(z)\,dz-\bint_{{Q'}}u(z)\,dz\right| \le C_0 \ell(Q)^{-1} \int_{Q} M(\widehat{|\nabla u|})(z) \, dz. 
 \]
Here $C_0$ only depends on $J$ and the constants
in \eqref{oletus}. 
\end{lemma}

\begin{proof}
Since $\Omega$ is John and $\varphi(0)$ is a John center of $\Omega,$ $\varphi$ is 
$\eta$-quasisymmetric with respect 
to the inner distance by Lemma \ref{John subdomain}, where $\eta$ depends only on the
John constant $J.$ 
Next, $\dist(A,\,\partial \mathbb D)=\dist(A,\, S(A))$ for each $A\subset \mathbb D.$ Since  
$\varphi^{-1}(Q),\,\varphi^{-1}(Q')$ are of $\lambda$-Whitney-type for some absolute constant $\lambda$  by Lemma~\ref{whitney preserving},
we conclude that 
\begin{equation}\label{110}
\dist(\varphi^{-1}(Q),\,\varphi^{-1}(S(Q)))\le C(\lambda) \diam(\varphi^{-1}(Q)) 
\end{equation}
and 
$$ \dist(\varphi^{-1}(Q'),\,\varphi^{-1}(S(Q')))\le C(\lambda) \diam(\varphi^{-1}(Q')).$$

Let us show that quasisymmetry of $\varphi$ 
allows us to translate \eqref{110} and its analog for 
$Q'$ to $\Omega.$
Pick $z_1\in  \varphi^{-1}(Q)$ and $z_2 \in \varphi^{-1}(S(Q))$ such that
\begin{equation}\label{tor1}
\dist( \varphi^{-1}(Q),\, \varphi^{-1}(S(Q)))=|z_1-z_2|,
\end{equation}
and let $z_3\in \varphi^{-1}(Q)$ be a point such that
\begin{equation}\label{tor2} 
\diam(\varphi^{-1}(Q)) \le 2 |z_1-z_3|.
\end{equation}
Recall that $\varphi(z_2)$ is rectifiably joinable, say, to $\varphi(0)$ by 
Remark \ref{muistutus}. Pick points $w_j$ along this rectifiable curve so that 
$w_j$ tend to $\varphi(z_2)$ and $\dist_{\Omega}(\varphi(z_1),\,w_j)$ tends to $\dist_{\Omega}(\varphi(z_1),\varphi(z_2)).$ Since $\varphi$ is homeomorphic up to
boundary, it follows that $\varphi^{-1}(w_j)$ tend to $z_2.$ Hence, 
by \eqref{110},\eqref{tor1},\eqref{tor2} we have
$$|z_1-\varphi^{-1}(w_j)|\le C(\lambda) |z_1-z_3|$$
when $j$ is sufficiently large.
Then the quasisymmetry of $\varphi$  gives
$$\dist_{\Omega}( \varphi(z_1),\, w_j)\le C(J,\,\lambda) \dist_{\Omega}( \varphi(z_1),\, \varphi(z_3))$$
for all sufficiently large $j.$ Since $\dist_{\Omega}(\varphi(z_1),\,w_j)$ tends to 
$\dist_{\Omega}(\varphi(z_1),\varphi(z_2))$ we deduce that
$$\dist_{\Omega}( \varphi(z_1),\, \varphi(z_2))\le C(J,\,\lambda) \dist_{\Omega}( \varphi(z_1),\, \varphi(z_3)).$$
Hence 
\begin{equation}\label{111}
\dist_{\Omega}( Q,\, S(Q))\ls \diam_{\Omega}( Q)\sim \ell( Q)
\end{equation}
with constants depending only on $\lambda$ and $J.$
Similarly 
\begin{equation}\label{111toka}
\dist_{\Omega}( Q',\, S(Q'))\ls \diam_{\Omega}( Q')\sim \ell( Q').
\end{equation}

By the triangle inequality (see Lemma~\ref{inner triangle}), \eqref{111}, \eqref{111toka}, 
Lemma~\ref{shadow estimate} and  \eqref{oletus} we conclude that 
\begin{align*}
\dist_{\Omega}(Q,\,Q') \lesssim & \dist_{\Omega}(Q,\,S(Q))+\diam_{\Omega}(S(Q))+\dist_{\Omega}(S(Q),\,S(Q'))\\
&+\diam_{\Omega}(S(Q'))+\dist_{\Omega}(Q,\,S(Q'))\\ \ls& \ \ell(Q)
\end{align*}
with constants depending only on $\lambda$ and $J$. 
By Lemma \ref{lma:map and geodesic} we deduce from this that the length of the hyperbolic segment 
$\Gamma$
between the centers 
of $Q$ and $Q'$ is no more than a constant (only depending on the constants in 
\eqref{oletus} and the John constant $J$)
multiple of $\ell(Q).$

Next, we construct  a John subdomain $\Omega_{Q,\,Q'}\subset \Omega\cap CQ$ of 
diameter no more than 
$C \ell(Q),$
containing both $Q$ and $Q'$, where $C$ only depends on the John constant 
$J$. Towards this, set
 $$\Omega_{Q,\,Q'}= Q \cup Q' \cup\bigcup_{z\in \Gamma} B\left(z,\,3^{-1}\dist(z,\,\partial \Omega)\right),$$
where $\Gamma$ is the above hyperbolic segment between the centers 
of $Q$ and $Q'.$ To see that $\Omega_{Q,\,Q'}$ is John, let $z_0$ be the middle point (in the sense of length)
of $\Gamma$ and
consider, for a given
$z\in \Omega_{Q,\,Q'}$, the following curve $\gamma$: the first part of the 
curve is a line segment from  $z$ to $z_1\in \Gamma$, where 
$z\in B\left(z_1,\,3^{-1}\dist(z_1,\,\partial \Omega)\right)$ or $z_1$ is 
the center of $Q$ (or $Q'$) if $z\in Q$ (or $z\in Q'$), and the second part 
coincides with $\Gamma[z_1,\,z_0].$ Since a simply connected John domain 
$\Omega$ is 
(quantitatively) inner uniform and we can use
hyperbolic segments as the curves required in  \eqref{e1.1}
(see Definition \ref{inneruniform} and Lemma \ref{johninner}), it 
follows that the above curve is a John curve of $\Omega_{Q,\,Q'}$ between $z$ 
and $z_0,$ with a 
constant only depending on $J.$

By letting 
$$a= \bint_{\Omega_{Q,\,Q'}} u\,dz,\, a_Q= 
\bint_{{Q}}u(z)\,dz,\, a_{Q'}= 
\bint_{{Q'}}u(z)\, dz,$$
the Poincar\'e inequality on $\Omega_{Q,\,Q'}$ from \cite{B1997} (with a constant depending only on $J$) and   \eqref{oletus}  imply
\begin{align*}
|a_{Q}-a_{Q'}|&\le \left|a_{Q}-a\right|+\left|a_{Q'}-a\right|\lesssim \bint_{Q} \left|u-a\right|\, dz+ \bint_{Q'} \left|u-a\right|\, dz\\
&\lesssim \ell(Q)^{-1}\int_{\Omega_{Q,\,Q'}}|\nabla u(z)|  \, dz  \lesssim \ell(Q)  \bint_{CQ}\widehat{|\nabla u|}(z)\, dz  \\
&\lesssim \ell(Q) \bint_{Q} M(\widehat{|\nabla u|})(z) \, dz  \lesssim \ell(Q)^{-1} \int_{Q} M(\widehat{|\nabla u|})(z) \, dz.
\end{align*}
\end{proof}

\begin{figure}
 \centering
 \includegraphics[width=0.95\textwidth]{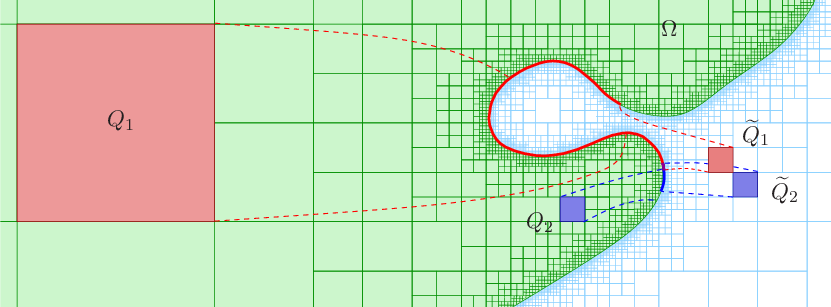}
 \caption{The shadows of neighboring squares $\widetilde Q_1$ and $\widetilde Q_2$ 
          can differ significantly in size from each other. Consequently the reflected squares 
          $Q_1$ and $Q_2$ may be of very different size.}
 \label{fig:shadows}
\end{figure}


\subsection{Intermediate Whitney squares}\label{sectionintermediate}

We would like to employ Lemma~\ref{special case} to estimate $|a_{\rf(\wz Q_i)}-a_{\rf(\wz Q_j)}|$ for pairs of neighboring squares $\wz Q_i$ and $\wz Q_j.$ Unfortunately, the reflected squares need not have comparable size (see Figure~\ref{fig:shadows}), and hence we cannot always directly rely 
on Lemma \ref{special case}. To fix this problem, we construct chains of suitable intermediate Whitney squares in order to be able to use our estimate.  Each of these chains consists of a finite number of elements, but there is no uniform bound for these numbers.

\begin{lemma}\label{feikit}
Let $\wz Q_i,\,\wz Q_j$ be distinct squares so that $\wz Q_i\cap \wz Q_j\neq\emptyset$ and 
$$\ell(\wz Q_i),\,\ell(\wz Q_j)\le 3\diam(\Omega).$$ Suppose that $\diam(S(\wz Q_i))\le \diam(S(\wz Q_j)).$  
Then there exist $l=l(i,j)\in \mathbb N$ and $G(\wz Q_i,\wz Q_j):=\{Q^0,\dots,Q^l\}$ consisting of squares of $W$ so that the following holds:
 \begin{equation}\label{paat} Q^0=\rf(\wz Q_i) \mbox{ and } Q^l= \rf(\wz Q_j),
 \end{equation}
for $0\le n \le l-1$ we have the estimate
\begin{equation}\label{kuutioiden etaisyys}
\dist_{\Omega}(S(Q^n),\,S(Q^{n+1}))\lesssim \ell(Q^n)\sim \ell(Q^{n+1}),
\end{equation}
and,  for $0\le m \le l,$ the estimate
\begin{equation}\label{kuutioiden koot}
\ell(Q^m)\sim 2^{-m}\diam (S(\wz Q_j)),
\end{equation}
with constants only depending on $J.$
\end{lemma}
\begin{proof}
Let distinct squares
$\wz Q_i,\,\wz Q_j$ with $\wz Q_i\cap \wz Q_j\neq \emptyset$ satisfy both $$\diam(S(\wz Q_i))\le \diam(S(\wz Q_j))$$ and 
$$\ell(\wz Q_i),\,\ell(\wz Q_j)\le 3\diam(\Omega).$$
If 
\begin{equation}\label{equat202}
\frac 1 8 \diam(S(\wz Q_j)) \le \diam(S(\wz Q_i)),
\end{equation}
we set $l(i,j)=1$ and define $Q^0=\rf(\wz Q_i),\,Q^1=\rf(\wz Q_j).$
Then by Lemma~\ref{shadow estimate}, Lemma~\ref{existence} and the fact that $\wz Q_i\cap \wz Q_j\neq \emptyset$, we have that \eqref{kuutioiden etaisyys} holds  with  constants depending only on $J$. 
Moreover, \eqref{kuutioiden koot} holds with an absolute constant.

Suppose that \eqref{equat202} fails. Set $Q^0=\rf(\wz Q_i).$
Pick a connected closed set $\wz F^1$ (referred to as a fake square) such that $\wz \Omega\setminus \wz F^1$ is connected, 
$$\wz Q_i\subset \wz F^1 \subset \wz Q_i\cup \wz Q_j,\, S(\wz Q_i)\subset S(\wz F^1)$$
and
\begin{equation}\label{equat56}
2\diam(S(\wz F^1))=\diam(S(\wz Q_i\cup \wz Q_j)).
\end{equation}
The existence of $\wz F^1$ is clear since
$\wz \varphi\colon\mathbb R^2\setminus \mathbb D\to \overline{\wz \Omega}$ is a homeomorphism and conformal outside $\overline {\mathbb D}$. 
For example, we can construct  $\wz F^1$ in the following way. Since $\wz \varphi$ is a homeomorphism, we know that both $\wz \varphi^{-1}(\partial \wz Q_i)$ and $\wz \varphi^{-1}(\partial \wz Q_j)$ are   Jordan curves, and they intersect each other. Pick $z\in \partial \wz Q_i\cap \partial \wz Q_j$. Then  parameterizing $\wz \varphi^{-1}(\partial \wz Q_j)$ via $\gamma\colon[0,\,1]\to \wz \varphi^{-1}(\partial \wz Q_j)$ with $\gamma(0)=\gamma(1)=z$, by continuity there is $0<t<1$ such that, by letting
$\wz F^1=\wz \varphi(\gamma[0,\,t]\cup \wz Q_i),$
we have that \eqref{equat56} holds; notice that the preimages under $\wz \varphi$ of hyperbolic rays are radial rays, and then $\wz \varphi^{-1}(S(\partial \wz Q_j))=\wz \varphi^{-1}(S(\wz Q_j))$. 
Then by our construction it is clear that $\wz Q_i\subset \wz F^1 \subset \wz Q_i\cup \wz Q_j$ and
that $\wz \Omega \setminus \wz F^1$ is connected. Hence $\wz F^1$  is  a desired set. 

Notice that $\wz F^1$ is a Whitney-type set since $\ell(\wz Q_i)\sim \ell(\wz Q_j)\sim \diam(\wz F^1)$ and $\wz Q_i \subset \wz F^1$. By Lemma~\ref{existence}, there is a Whitney square $Q^1\in W$ such that
$$ \diam(S(Q^1)) \le {C(J)} \diam(S(\wz F^1)) , $$
and
$$ \diam(S(\wz F^1 )) \le  {C(J)}\diam(S(Q^1)\cap S(\wz F^1)), $$
where $C(J)$ depends only on $J$.  
We did not need the assumption that $\wz \Omega \setminus \wz F^1$
be connected above; we will later use it in oder to apply Lemma~\ref{sum estimate}.

Next we pick a connected closed set $\wz F^2$ such that $\wz \Omega \setminus \wz F^2$ is connected,
$\wz Q_i\subset \wz F^2 \subset \wz F^1\subset \wz Q_i\cup \wz Q_j$, $S(\wz Q_i)\subset S(\wz F^2)$ and
$$4\diam(S(\wz F^2))=\diam(S(\wz Q_i\cup \wz Q_j)), $$
and select a Whitney square $Q^2\subset \Omega$ such that
$$ \diam(S(Q^2)) \le {C(J)} \diam(S(\wz F^2)) , $$
and
$$ \diam(S(\wz F^2 )) \le  {C(J)}\diam(S(Q^2)\cap S(\wz F^2)), $$
where $C(J)$ depends only on $J.$ We continue this process to find squares $Q^l\in W$ until we have 
$$\frac 1 2\diam(S(\wz F^l))\le \diam(S( \wz Q_i) )\le \diam(S(\wz F^l)) $$
for some $l\in \mathbb N$. 

By our construction,
\begin{equation}\label{kopio}
2^m \diam(S(\wz F^m))=\diam(S(\wz Q_i\cup \wz Q_j))
\end{equation}
for $m=1,\dots,l.$ 
Next, $Q^m$ was obtained via  Lemma~\ref{existence} where the corresponding
square satisfies by \eqref{samakoko} the additional 
requirement that
\begin{equation}\label{saamaat}
\diam(S(Q^m))\sim_J \ell(Q^m).
\end{equation}
Taking into account the estimate
\begin{equation}\label{vaarajarj}
\diam(S(\wz F^m))\lesssim \diam(S(Q^m))\lesssim \diam(S(\wz F^m))
\end{equation} 
with constants only depending on $J$ that
follows from our choice of $Q^m,$ we conclude with  \eqref{kuutioiden koot}. 

Regarding  \eqref{kuutioiden etaisyys},
recall from the construction that $S(Q^m)\cap S(\wz F^m)\neq \emptyset$
and 
$S(\wz F^n)\cap S(\wz F^{n+1})\neq \emptyset$ for all relevant $n,m.$
Since $\dist_{\Omega}$ satisfies
a triangle inequality by Lemma~\ref{inner triangle}, 
we conclude that
$$\dist_{\Omega}(S(Q^n),\,S(Q^{n+1}))\ls $$
\begin{equation}\label{kyantaa}
 \ls \diam_{\Omega}(S(Q^n))+\diam_{\Omega}(S(\wz F^n))+ \diam_{\Omega}(S(\wz F^{n+1}))+\diam_{\Omega}(S(Q^{n+1})).
\end{equation}
Hence \eqref{kopio}, \eqref{saamaat}, \eqref{vaarajarj} and \eqref{kyantaa}
together with Lemma~\ref{inner diameter diameter} 
give \eqref{kuutioiden etaisyys}.
\end{proof}

From now on, in this subsection, we will always assume that $\ell(\wz Q_i)\le 3\diam(\Omega)$ and that $\ell(\wz Q_j)\le 3\diam(\Omega).$  We call such Whitney squares allowable.
The preceding lemma gives the chain $G(\wz Q_i,\,\wz Q_j)$ when 
$$\diam(S(\wz Q_i))\le \diam(S(\wz Q_j)).$$  
Especially,  both  $G(\wz Q_i,\,\wz Q_j)$ and $G(\wz Q_j,\,\wz Q_i)$ have been constructed when $$\diam(S(\wz Q_i))= \diam(S(\wz Q_j)).$$
Even though the claim of the lemma does not imply that these chains coincide as sets, the construction in the proof of the lemma gives this.
In order not to make our notation overly complicated, we abuse notation and extend our definition also to the case where $$\diam(S(\wz Q_i))> \diam(S(\wz Q_j))$$  by setting 
$G(\wz Q_i,\,\wz Q_j):=G(\wz Q_j,\,\wz Q_i).$ Under this convention, $Q^0$ is one of the squares   $\rf(\wz Q_i), \rf(\wz Q_j)$ and $Q^l,$ $l=l(i,j)=l(j,i),$ is the other one, \eqref{kuutioiden etaisyys}
holds as stated, but for \eqref{kuutioiden koot} we need to replace $\diam (S(\wz Q_j))$ with the maximum of 
$\diam (S(\wz Q_j))$ and $\diam (S(\wz Q_j)).$ 

Given an allowable $\wz Q_i\in \wz W$ with we define
$$G(\wz Q_i)=\cup_j G(\wz Q_i,\,\wz Q_j),$$
where the union runs over all the squares $\wz Q_j \in \wz W$ that intersect $\wz Q_i.$ 

Our next lemma gives estimates for the overlaps of our chains.

\begin{lemma}  \label{ketjujenpaall}
There is a positive integer $N=N(J)$ so that
\begin{equation} 
\sum_{Q\in G(\wz Q_i,\,\wz Q_j)} \chi_Q(x)\le N
\end{equation}
for all $i,j$ and every $x\in \Omega.$ 
Moreover,
\begin{equation}
\sum_{j} \sum_{Q\in G(\wz Q_i,\,\wz Q_j)} \chi_Q(x) \le 20\sum_{Q\in G(\wz Q_i)}\chi_Q(x)
\end{equation}
 for each $i$ and
\begin{equation}
\sum_i \sum_{Q\in G(\wz Q_i)}\chi_Q(x)<\infty,
\end{equation}
for all $x\in \Omega.$
\end{lemma}
\begin{proof}
The first claim follows from \eqref{kuutioiden koot}. The second claim is an immediate consequence of the fact that the Whitney square $\wz Q_i$ has at most 20 neighbors.
Towards the final claim, recall that 
$\widetilde \varphi \colon \mathbb R^2 \setminus \mathbb D \to \overline {\widetilde \Omega}$ is a homeomorphism  (and conformal in  
$\mathbb R^2 \setminus \overline{\mathbb D}$). This implies that the diameter 
of the shadow of $\tilde A$ tends to zero uniformly when 
$\diam(\tilde A)\to 0.$ Consequently, given $\delta>0,$ there can be only a finite number of 
$\wz Q_j\in \wz W$ with 
$\ell(\wz Q_i)\le 3\diam(\Omega)$ for which $\diam(S(\wz Q_i\cup \wz Q_j))\ge \delta$ for some neighbor $\wz Q_i$ of $\wz Q_j.$ The final claim follows from this together
with our first claim, \eqref{saamaat} and \eqref{vaarajarj}.
\end{proof}

Notice that we are not claiming a uniform bound for the number of distinct $\wz Q_i$ for which a given $Q$ belongs to $G(\wz Q_i).$ In fact, such a bound
does not necessarily exist.
The following lemma provides us with a crucial substitute for such an estimate.

\begin{figure}
\centering
\includegraphics[width=1\textwidth]{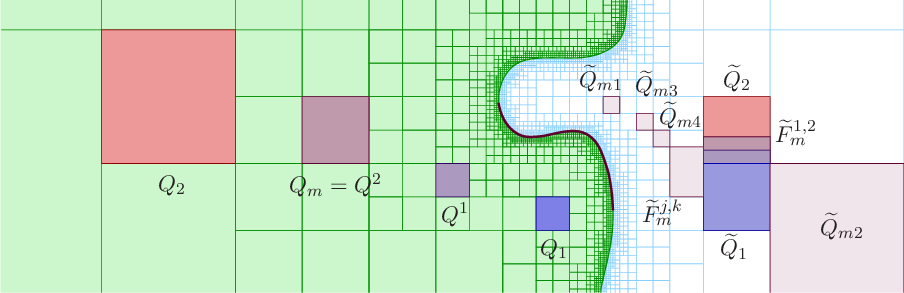}
     \caption{A square $Q \in W$  might be associated to several squares 
$\widetilde Q_{l}$ as well as to fake squares $\widetilde F_Q^{j,k}$.
     In the illustration the squares $\widetilde Q_1$ and $\widetilde Q_2$ 
give rise to two fake squares, one of which is associated with $Q$.
     Another fake square as well as four (real) squares that are associated 
with $Q$ are exhibited. Also the shadow of $Q$ is shown.  }
   \label{fig:fakes}
  \end{figure}

\begin{lemma}\label{strong sum estimate}
For each $Q\in W$, we have
\[
\sum_{Q\in G(\wz Q_i) } \ell(\wz Q_{i})^{2-\hat p}\lesssim \ell({Q})^{2-\hat p}, 
\]
where the constant depends only on $\hat p$ and the constant $C$ in  
\eqref{eq:condition}.
\end{lemma}

\begin{proof}
Recall that $\Omega$ is $J$-John with a constant that only depends on $\hat p$ and the constant $C$ in \eqref{eq:condition} by 
Lemma \ref{komplementtionjohn}.

Fix $Q\in W$ so that $Q\in G(\wz Q_i)$ for at least one $i.$ By Lemma \ref{ketjujenpaall} the number of such indices $i$ is finite and, for each of the
at most 20 neighbors $\wz Q_j,$ $Q$ corresponds
to at most $N(J)$ different fake squares $\wz F_{i,j}^m$ used in the construction of $G(\wz Q_i,\,\wz Q_j).$ Consider this finite collection of the sets
$\wz F_{i,j}^m.$ 
We relabel them as $\wz F_{n}$ with
respect to $n$, say $1\le n\le k,$ so 
that the diameters of $\widetilde\varphi^{-1}(S(\wz F_n))$ decrease when 
$n$ increases.

We set $\wz F_Q^1:=\wz F_k$ and stop the construction if $k=1$ or if the shadow of each $F_n$ with $1\le n\le k-1$ intersects the shadow of $\wz F_Q^1.$
If this is not the case and  
$S(\widetilde  F_{k-1})\cap S(\wz F_Q^1)= \emptyset$ we set
$\wz F_Q^2=\wz F_{k-1}.$ Otherwise, we consider $\wz F_{k-2}$ as
a candidate for $\wz F_Q^2$ and continue inductively via the following procedure.
We choose $\wz  F_Q^2$ to be $\wz F_{n}$ for the largest integer $n$ 
smaller than $k$ for which
$$S(\widetilde  F_{n})\cap S(\wz F_Q^1)= \emptyset. $$ We stop the process if $n=1$ or if the shadow of each $F_m$ with $1\le m\le n-1$ intersects $S(\wz F_Q^1)$ or $S(\wz F_Q^2).$
Otherwise, we choose $\wz F_Q^3$ to be $\wz F_{m}$  with the largest $m\le n-1$ such that its 
shadow does not intersect $S(\wz F_Q^1)$ nor $S(\wz F_Q^2)$, 
and continue this process. 
This gives us   $\wz F_Q^1,\dots,\wz F_Q^{n_0}$
with pairwise disjoint shadows. By the construction of these sets, 
Lemma~\ref{auxiliary 1} gives us a universal bound on $n_0$ in terms of 
$C(J);$ see \eqref{sameshadow} and \eqref{sameinter}.

Let $\wz  F_{n}$ be a set from above which was not chosen as one of the sets 
$\wz F_Q^{i}$. 
By the construction in the previous paragraph, there
is an index $l$ so that  $S(\widetilde  F_{n})\cap S(\wz F_Q^l)\neq \emptyset.$ 
Since $\widetilde\varphi^{-1}(S(\wz F_n))=S(\widetilde\varphi^{-1}(\wz F_n))$ and
$\widetilde\varphi^{-1}(S(\wz F_Q^l))=S(\widetilde\varphi^{-1}(\wz F_Q^l))$ are closed arcs of the unit circle,
at least one of the end points of 
$S(\wz  F_Q^{l})$ is contained in $S(\wz  F_{n})$; 
otherwise $S(\wz  F_{n})$ is strictly contained in $S(\wz F_Q^{l}),$ which means that 
$$ \diam (\widetilde\varphi^{-1} (S(\wz F_Q^{l})))> 
\diam (\widetilde\varphi^{-1} ( S(\wz  F_{n}))),$$ 
contradicting our selection of the sets
$\wz F_Q^{l}.$ 
Therefore, by assigning two hyperbolic rays to each $\wz F_Q^{l},$ 
we obtain a collection
of $2n_0$ hyperbolic rays that intersect all of our sets 
$\widetilde  F_{i,\,j}^m$ with 
$i\in I(Q).$

Let $\Gamma$ be one of our $2n_0$ hyperbolic rays.
Denote by $\Gamma_0$ the tail of 
$\Gamma$ with respect to a set in 
$$\{ \wz F_{i,\,j}^m \mid Q\in G(\wz Q_i),\, \Gamma 
\cap\wz F_{i,\,j}^m \neq \emptyset \}$$
whose preimage under $\widetilde \varphi$ is furthest away from the origin, 
that is, a last set that $\Gamma$ hits towards infinity. 
Let $\wz F_0$ be such a set. Then  $\ell(Q)\sim 
\diam(S(\wz F_0))$ by \eqref{kuutioiden koot} as $\wz F_0$ is one of the sets $\wz F_{i,\,j}^m.$ 
Moreover, $\wz F_0$ is of $8\sqrt 2$-Whitney-type and $\wz \Omega\setminus \wz F_0$ is connected since $\wz F_0$ is one of the sets $\wz F_{i,\,j}^m.$ See Figure~\ref{fig:fakes} for an illustration.
Hence Lemma~\ref{sum estimate} gives the estimate
\begin{equation}\label{oikeavaite}
\sum_{\wz Q_l\in \wz W,\, \wz Q_l\cap \Gamma_0\neq \emptyset} \ell(\wz Q_{l})^{2-\hat p}\ls
\ell({Q}_m)^{2-\hat p}
\end{equation}
with a constant that only depends on $\hat p$ and the constant $C$ in \eqref{eq:condition}.

Since each $\wz  F_{i,\,j}^m$ is 
a subset of $\wz Q_i\cup \wz Q_j$ where $\wz Q_i\cap \wz Q_j\neq \emptyset,$
each Whitney square has at most 20 neighbors, and the number $n_0$ of our hyperbolic rays is
bounded in terms of $J,$
our claim follows from \eqref{oikeavaite}.
\end{proof}

\subsection{Sufficiency in the Jordan case}\label{suffjordan}

Recall the definition of $Eu$ via \eqref{eq:jordanextdefinition1} from Subsection~\ref{jordandef} and of the chains
$G(\wz Q_i,\,\wz Q_j)$ and the sets $G(\wz Q_i)$ from Subsection~\ref{sectionintermediate}. We begin by
estimating the 
norm of the gradient of our extension over each 
square $\widetilde Q \in \widetilde W$ with  $\wz Q\cap B_{\Omega}\neq \emptyset.$ 

\begin{lemma}\label{lma:cubeestimate}
 For all $\widetilde Q_i \in \widetilde W$ with 
$\wz Q_i\cap B_{\Omega}\neq \emptyset,$ we have
 \[
\|\nabla E u\|^p_{L^p(\widetilde Q_i)}  \le C\, \sum_i \sum_{Q\in G(\wz Q_i,\,\wz Q_k)}\ell(\widetilde Q_i)^{2-\hat p} \ell(Q)^{\hat p-2}   \int_{Q} M(\widehat{|\nabla u|})(z))^p \, dz, 
 \]
 where the sum is over all the indices $k$ for which $\widetilde Q_{k}\cap \widetilde Q_j\neq\emptyset$.
Here $C$ depends only on $p,\hat p$ and the  
constant $C$ in \eqref{eq:condition}.
\end{lemma}
\begin{proof}
Recall that 
$\wz \varphi:\R^2\setminus \overline {\mathbb D}\to \wz\Omega$ 
extends homeomorphically up to the
boundary. 

Fix $\widetilde Q_{j}$ with $\widetilde Q_{j}\cap \widetilde Q_i\neq \emptyset.$ 
Let $Q^n,Q^{n+1}\in G(\wz Q_i,\,\wz Q_j)$ be consecutive squares. Then
\begin{equation} \label{equat7111}
\dist_{\Omega}(S(Q^n),S(Q^{n+1}))\ls_J \ell(Q^n) \sim_J \ell(Q^{n+1}) 
\end{equation} 
by \eqref{kuutioiden etaisyys}.
Let $q>0.$ Then, by \eqref{kuutioiden koot} together with Lemma~\ref{shadow estimate}, we have the
estimate
\begin{align}
\sum_{Q^n\in G(\wz Q_i,\,\wz Q_j)} \ell(Q^n)^{-q}&\le C(q,J)\min\{
\diam(S(\wz Q_i)),\diam(S(\wz Q_j)\}^{-q}\\ &\le C(q,J) \ell(\wz Q_i)^{-q}.\label{kuusumma}
\end{align}

Recall that $\{\phi_{k}\}$ is a partition of unity with $\phi_k=0$ in $\wz Q_j$ if $\wz Q_j\cap \wz Q_k=\emptyset.$ Hence, for each $x\in \wz Q_i,$ we have 
$$\nabla E u(x)=\nabla \left( \sum_{\widetilde Q_{k}\cap \widetilde Q_i\neq\emptyset} a_{k} \phi_k(x) \right) =  \nabla \left(\sum_{\widetilde Q_{k}\cap \widetilde Q_i\neq\emptyset} (a_{k}-a_{i}) \phi_k(x) \right),$$
where $a_l$ refers to the average of $u$ over $\rf(\wz Q_l).$
Since $|\nabla \phi_k|\ls \ell(\wz Q_i)^{-1}$ whenever $\wz Q_{k}\cap \wz Q_i\neq\emptyset$, we further have
\begin{align} 
\|\nabla E u\|^p_{L^p (\widetilde Q_i)}&\ls \int_{\wz Q_i} \sum_{\wz Q_{k}\cap \wz Q_i\neq\emptyset} |a_{k}-a_{i}|^p  |\nabla   \phi_k(x)|^p \, dx \nonumber \\
&\lesssim \sum_{\wz Q_{k}\cap \wz Q_i\neq\emptyset}|a_{k}-a_{i}|^p\ell(\widetilde Q_j)^{-p}|\widetilde Q_i| \\
&\lesssim  \sum_{\wz Q_{k}\cap \wz Q_i\neq\emptyset}|a_{k}-a_{i}|^p\ell(\widetilde Q_i)^{2-p} \label{esti}
\end{align}
with an absolute constant.

Let $\ez=\frac {\hat p-p}{p}>0$.  We apply Lemma~\ref{special case} via 
\eqref{equat7111},  H\"older's 
inequality and \eqref{kuusumma} with $q=\frac {\ez p}{p-1}$ to get
\begin{align*}
 |a_{k}-a_{i}|^p & \lesssim\left(\sum_{Q^n\in G(\wz Q_k,\,\wz Q_i)}|a_{Q^n}-a_{Q^{n+1}}| \right)^{p} \\
 &\lesssim\left(\sum_{Q^n\in G(\wz Q_k,\,\wz Q_i)} \ell(Q^n) \bint_{Q^n} M(\widehat{|\nabla u|})(z) \, dz \right)^{p}\\
&\lesssim\left[\sum_{Q^n\in G(\wz Q_k,\,\wz Q_i)} \ell(Q^n)^{1+\ez-\ez} \left( \bint_{Q^n} (M(\widehat{|\nabla u|})(z))^p \, dz\right)^{\frac 1 p} \right]^{p}\\
&\lesssim \left(\sum_{Q_m\in G(\wz Q_k,\,\wz Q_i)} \ell(Q^n)^{p+p\ez}   \bint_{Q^n} (M(\widehat{|\nabla u|})(z))^p \, dz \right)  \left(\sum_{Q^n\in G(\wz Q_k,\,\wz Q_i)} \ell(Q^n)^{-\frac {\ez p}{p-1}} \right)^{p-1}\\
&\lesssim \ell(\wz Q_i)^{-\ez p}  \sum_{Q^n\in G(\wz Q_k,\,\wz Q_i)} \ell(Q^n)^{p+p\ez-2}   \int_{Q^n} (M(\widehat{|\nabla u|})(z))^p \, dz. 
\end{align*}
Above, the constants only depend on  $p,\hat p$ and the
constant $C$ in \eqref{eq:condition}.

By recalling that $\ez p=\hat p-p$ and inserting the above estimate into \eqref{esti}, we obtain
\begin{align*}
\|\nabla E u\|^p_{L^{ p}(\widetilde Q_i)}& \lesssim  \sum_{\widetilde Q_{k}\cap \widetilde Q_i\neq\emptyset}|a_{k}-a_{i}|^p\ell(\widetilde Q_i)^{2-p}\\
& \lesssim \sum_{\widetilde Q_{k}\cap \widetilde Q_i\neq\emptyset}   \sum_{Q^n\in G(\wz Q_i,\,\wz Q_k)}\ell(\widetilde Q_i)^{2-\hat p} \ell(Q^n)^{\hat p-2}   \int_{Q^n} (M(\widehat{|\nabla u|})(z))^p \, dz, 
\end{align*}
with the desired control on the constants.
\end{proof}

\begin{proof}[Proof of Theorem~\ref{prop:Jordancase}]
Recall that $B_{\Omega}=B(x_0,\diam(\Omega)),$ $Eu$ is defined on  $B_{\Omega}$ as in \eqref{eq:jordanextdefinition1} and that
$\ell(\wz Q)\le 3
\diam(\Omega)$ whenever $\wz Q\in \wz W$ intersects $B_{\Omega}$ or is a neighbor of such 
a square.
By Lemma~\ref{lma:cubeestimate}, we have
$$\|\nabla E u\|^p_{L^p(B_{\Omega}\setminus\overline{\Omega})} \ls \sum_{\wz Q_i\cap B_{\Omega}\neq \emptyset} \sum_{\widetilde Q_{k}\cap \widetilde Q_i\neq\emptyset}   \sum_{Q\in G(\wz Q_i,\,\wz Q_k)}\ell(\widetilde Q_i)^{2-\hat p} \ell(Q)^{\hat p-2}   \int_{Q} (M(\widehat{|\nabla u|})(z))^p \, dz$$
with a constant only depending on our data: $p,\hat p$ and the constant $C$ in \eqref{eq:condition}. 

Towards interchanging the order of summation, notice that a fixed 
$Q\in W$ appears in our triple sum only when $Q\in G(\wz Q_i)$ in which case it is counted for each of the at most 20 neighbors $\wz Q_j$ at most $N(J)$ times by Lemma \ref{ketjujenpaall}.
Hence by interchanging the order of summation (Tonelli's theorem),  we obtain by Lemma~\ref{strong sum estimate} the estimate
\begin{align}
\|\nabla E u\|^p_{L^p(B_{\Omega}\setminus\overline{\Omega})} &\ls \sum_{\wz Q_i\cap B_{\Omega}\neq \emptyset} \sum_{\widetilde Q_{k}\cap \widetilde Q_i\neq\emptyset}   \sum_{Q\in G(\wz Q_i,\,\wz Q_k)}\ell(\widetilde Q_i)^{2-\hat p} \ell(Q)^{\hat p-2}   \int_{Q} M(\widehat{|\nabla u|})(z))^p \, dz \nonumber \\ 
& \lesssim  \sum_{Q\in W}\sum_{Q\in G(\wz Q_i)}  \ell(\widetilde Q_i)^{2-\hat p}  \ell(Q)^{\hat p-2}   \int_{Q} (M(\widehat{|\nabla u|})(z))^p \, dz\nonumber \\
& \lesssim \sum_{Q\in W} \int_{Q} (M(\widehat{|\nabla u|})(z))^{p} \, dz\nonumber \\
&\lesssim \int_{\mathbb R^2} \widehat{|\nabla u|}^p(z) \, dz \le  \int_{\Omega} |\nabla u|^p  \, dz. \label{grad norm inequ}
\end{align}
Here the constants depend only on our data.


Next, recall that $Eu(x)=\sum_j a_{j}\phi_j(x)$ when $x\in B_{\Omega}\setminus \overline \Omega,$ where $a_{j}$ is the average over $\rf(\wz Q_j)\in W$ 
with $\ell(\wz Q_j)\le 3\,\diam(\Omega).$ Write 
$\rf^{-1}( Q)$ for the collection of all
$\wz Q_{j}\in \wz W$ with $Q=\rf(\wz Q_j).$ 
Now
$$\sum_{\wz Q_{j}\in \rf^{-1}(Q)}\ell(\wz Q_{j})^2\leq C(J) \ell(Q)^{ 2} $$
since for every $\widetilde Q_{j}\in \rf^{-1}(Q)$ we have $\wz Q_{j}\subset C(J)Q$ by 
Lemma~\ref{shadow estimate}, 
 \eqref{kuutionkokoet} and the triangle inequality.
Then, 
by the definition of $Eu,$ Tonelli's theorem for series and H\"older's inequality we obtain
\begin{align}
\|Eu\|^{p}_{L^p(B_{\Omega}\setminus\overline {\Omega})}&\lesssim \sum_{Q\in W}\sum_{\widetilde Q_{j}\in \rf  (Q)}\ell(\wz Q_{j})^2 \left(\bint_{Q} |u| \,dx\right)^p \nonumber\\
&\lesssim \sum_{Q\in W}\sum_{\widetilde Q_{j}\in \rf^{-1}( Q)}   \ell(\wz Q_{j})^2 \ell(Q)^{-2} \int_{Q} |u|^p\,dx \nonumber \\
&\lesssim \sum_{Q\in W}   \int_{Q} |u|^p\,dx\lesssim \int_{\Omega} |u|^p\,dx \label{equat8}
\end{align}
with constants only depending on our data.
By combining \eqref{grad norm inequ} and \eqref{equat8} we conclude that
$$\int_{B_{\Omega}\setminus \partial \Omega} |\nabla E u|^p + |Eu|^p\, dx\leq C 
\|u\|^p_{W^{1,\,p}(\Omega)},$$
where $C$ depends only on $p,\hat p$ and the constant $C$ in \eqref{eq:condition}.

Suppose now that $u\in W^{1,\,p}(\Omega)\cap C^{\infty}(\overline{\Omega})$. We extend $Eu$ to all of $B_{\Omega}$ by letting
$$\hat E u(x)=Eu(x) \ \text{ when $x\in B_{\Omega}\setminus \partial \Omega$,  } \  \ \hat Eu(x)=u(x) \ \text{ when $x\in \partial \Omega$.  }$$
We claim that $\hat Eu(x)$ is continuous in  $B_{\Omega}$.

Notice that $E u$ is clearly continuous (even smooth) in 
$B_{\Omega}\setminus \overline \Omega$ and smooth in $\Omega.$ Hence we are reduced to 
show continuity at every $x\in \partial \Omega.$ 
Recall that $\Omega$ is Jordan. This implies that $\diam(S(\wz Q))$ tends to zero uniformly
when $\ell(\wz Q)$ tends to zero. Given $x\in \partial \Omega$ and points $x_k$ converging to
$x$ from within $\wz \Omega,$ pick Whitney squares $\wz Q_k$ containing $x_k.$ Then by the fact that $\{\phi_j\}$ forms a partition of unity, we have
\begin{align*}
|\hat Eu(x_k)-u(x)|
&=\left|\sum_{\wz Q_j\cap \wz Q_k\neq \emptyset}a_j \phi_j(x_k)  - \sum_{\wz Q_j\cap \wz Q_k\neq \emptyset} \phi_j(x_k) u(x)\right|\\
&\le \sum_{\wz Q_j\cap \wz Q_k\neq \emptyset} \phi_j(x_k) |a_j-u(x)|. 
\end{align*}

Since $\wz Q_k$ tend to $x$, also the neighboring squares of  $\wz Q_k$ tend to $x.$ We claim that
also their shadows converge to $x.$ Towards this, it suffices to check that the preimages of their
shadows tend to $\wz \varphi^{-1}(x)$ under our homeomorphism $\wz \varphi\colon \mathbb R^2\setminus
\mathbb D\to \overline {\wz \Omega}$ that is conformal in $\mathbb R^2\setminus \overline {\mathbb D}.$
Now the preimages of the shadows of these squares are the radial projections of
the preimages $\wz \varphi^{-1}(\wz A_k)$ of these squares and the desired conclusion follows since
 $\wz \varphi^{-1}(\wz A_k)$ tend to $\wz \varphi^{-1}(x).$
Hence, it follows from Lemma \ref{shadow estimate} and Lemma~\ref{existence}
that the Whitney squares of $\Omega$ associated to the neighboring squares of $\wz Q_k$ also tend to
$x.$
Thus we have 
$$\hat Eu(x_k)\to u(x)$$
by the assumption that $u$ is the restriction of a smooth (especially continuous) function to $\Omega$
and $Eu(x_k)$ is defined via averages over the squares associated to the neighboring squares of 
$\wz Q_k.$

Recall that $\Omega$ is John and that the Lebesgue measure of $\partial \Omega$ is zero Lemma~\ref{bdyzero}.
With the continuity of $\hat Eu$, \cite[Theorem 4]{JS2000} then guarantees that the above definition gives a Sobolev function with the desired norm control. 
 Also by Lemma~\ref{bdyzero} we know that $\hat Eu=Eu$ as Sobolev functions. 
Thus
$E\colon W^{1,\,p}(\Omega)\cap C^{\infty}(\overline{\Omega}) \to W^{1,\,p}(B_{\Omega}) $
is a bounded operator, and it is  also linear by its definition.

Recall that $C^{\infty}(\overline{\Omega})$ is dense in $W^{1,\,p}(\Omega)$ for $1<p<\infty$ if
$\Omega$ is a planar Jordan domain, see \cite{lewis1987}. 
By our norm estimates above, we can (uniquely) extend $E$ to entire $W^{1,\,p}(\Omega)$ as 
a bounded operator. This extension is given by the original definition of $E.$ 
Since $B_{\Omega}$ is an extension domain, we conclude that the claim of the theorem follows. 
\end{proof}

\begin{remark}\label{lopulta}
The norm of our extension operator from $W^{1,p}(\Omega)$ into $W^{1,p}(B_{\Omega})$ only
depends on $p,\hat p$ and the constant $C$ in \eqref{eq:condition}, both for the homogeneous and
the full Sobolev norms; see \eqref{grad norm inequ} and \eqref{equat8}. 
Here $B_{\Omega}=B(x_0,\diam(\Omega))$ and $x_0$ is a chosen John center of
$\Omega.$ If we wish to extend
to entire $\mathbb R^2,$ then the norm of the extension operator will also necessarily
depend on the diameter of $\Omega$ if we use the full Sobolev norm.
\end{remark}

\subsection{Proof in the general case}\label{sufffinal}

We establish the existence of an extension operator in the general case of a bounded simply connected domain $\Omega$ 
via an approximation process, relying on our earlier results.

Recall that we are claiming the existence of a bounded extension operator under the 
condition \eqref{eq:extcharcompl} for a given bounded simply connected domain $\Omega.$ 
We have already verified a version of this if $\Omega$ is Jordan.

In order to be able to prove the general case by using the result for the Jordan 
case, we need a sequence of approximating Jordan
domains to have extension operators with uniform norm bounds. 
For this purpose we have stated the dependence
of the norm of the extension operator in Theorem \ref{prop:Jordancase} explicitly in Remark~\ref{lopulta}.

From now on, $\Omega$ is a bounded simply connected domain that satisfies 
\eqref{eq:extcharcompl}. Towards the existence of a suitable approximating sequence, 
recall that \eqref{eq:extcharcompl} guarantees that 
$\Omega$ is John, see Corollary \ref{komplementtionjohn}. 
Fix a conformal map $\varphi\colon \mathbb D \to \Omega$ 
so that $\varphi(0)$ is a John center of $\Omega.$ By Remark
\ref{uptobdy} we may extend $\varphi$ continuously up to the boundary.
We still denote the extended map by $\varphi$.

Let $B_n=B(0,\, 1-\frac 1 n)$ for $n\ge 2$. 
Then $\Omega_n=\varphi(B_n)$ are Jordan John domains
(with constant independent of $n$)  
contained in $\Omega$ by Lemma~\ref{John subdomain}, and converge to $\Omega$ 
uniformly in Hausdorff distane
because of the uniform continuity
of $\varphi$ up to the boundary. Actually, $\varphi$ is even uniformly H\"older continuous \cite{GM1985},
\cite{Ch1992}.

\begin{figure}
\includegraphics[width=0.9\textwidth]{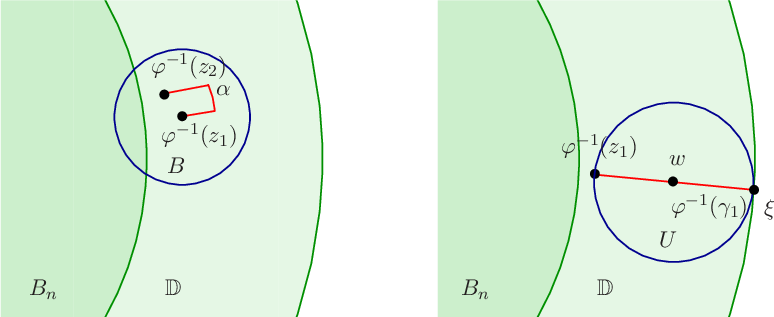}
\caption{The proof of the existence of the curve satisfying \eqref{eq:condition}
for the domain $\widetilde\Omega_n$ is split into two cases. On the left
we have the case where the preimages of the points $z_1$ and $z_2$ are
close enough so that one can use a curve $\alpha$ from Lemma~\ref{pisteetlahella} connecting them in the annular
domain $\Omega \setminus \Omega_n$. On the right is the case where the
preimages are far from each other and the constructed curve exits the
annular domain.}
\label{fig:curvext}
\end{figure}

Before giving the proof of Theorem~\ref{thm:main} we establish a technical 
result according to which the complementary 
domain of $\Omega_n$ satisfies condition \eqref{eq:condition}
with $\hat p>p$ and $C$ that are independent of $n$. 
This eventually allows us to apply Theorem 
\ref{prop:Jordancase} to $\Omega_n$ so as to complete the proof by a 
compactness argument.

\begin{lemma}\label{reduktio}
Each of the complementary domains $\wz \Omega_n$ of $\Omega_n$ satisfies 
condition  
\eqref{eq:condition}
with curves $\gamma\subset \wz \Omega_n$ for a fixed $\hat p>p$
and a constant independent of $n.$
\end{lemma}

\begin{proof}
Fix $n\ge 2$ and let $z_1,z_2\in \wz \Omega_n.$ Write $R=1-\frac 1 n.$
We begin by noticing that, if $z_1$ and $z_2$ are both outside 
$\Omega$, then condition \eqref{eq:condition} with a fixed $\hat p>p$ and $C$ follows immediately from 
\eqref{eq:extcharcompl} and the self-improvement property for $\Omega$ from 
Lemma~\ref{selfimprove},
since 
$\dist(z,\,\partial \Omega)\le\dist (z,\,\partial \Omega_n)$ for $z\in\R^2 \setminus\Omega$. 
Hence, switching $z_1$ and $z_2$ if necessary, we may assume that $z_1\in \Omega\setminus \Omega_n.$

We fix $\hat p>p$ as in the first paragraph of the proof.
Suppose first that also $z_2\in \Omega\setminus \Omega_n.$ 
Let us consider the case where 
$$\varphi^{-1}(z_2)\in B(\varphi^{-1}(z_1),\delta (1-|\varphi^{-1}(z_1)|)),$$  where $\delta$ is as in Lemma \ref{pisteetlahella}. Then
Lemma \ref{pisteetlahella} gives us a curve $\alpha$ joining $\varphi^{-1}(z_1)$ to $\varphi^{-1}(z_2)$ in $B(\varphi^{-1}(z_1),(1-|z_1|)/2)\setminus \overline{B}(0,R)$ so that
\begin{equation}\label{tamatahan}
\int_{\alpha}\dist (z,\partial B(0,R)\cup \partial B(\varphi^{-1}(z_1),(1-|\varphi^{-1}(z_1)|)/2))^{1-\hat p}\, ds(z)\le C(\hat p)|\varphi^{-1}(z_1)-\varphi^{-1}(z_2)|^{2-\hat p}.
\end{equation}
See Figure \ref{fig:curvext}.
Because $B:=B(\varphi^{-1}(z_1),(1-|\varphi^{-1}(z_1)|)/2)$ is of $2$-Whitney type, Lemma \ref{bilipominaisuus} gives us the estimate
$$C^{-1}|\varphi'(\varphi^{-1}(z_1))||w_2-w_1|\le |\varphi(w_2)-\varphi(w_1)|\le C|\varphi'(\varphi^{-1}(z_1))||w_2-w_1|$$
for all $w_1,w_2\in B$ for an absolute constant $C.$ Since $\alpha\subset B,$ we may apply this bi-Lipschitz estimate to the above integral over $\alpha$ so as
to conclude for $\gamma=\varphi\circ \alpha$ that
$$\int_{\gamma}\dist(z,\varphi(S(0,R)\cap B) \cup \varphi(\partial B)))^{1-\hat p}\, ds(z)\le C'(\hat p)|z_2-z_1|^{2-\hat p},$$
where $C'(\hat p)$ depends only on $\hat p.$ The desired inequality follows since $\varphi$ is a homeomorphism and  hence $\dist (z,\partial \wz \Omega_n)\ge \dist(z,\varphi((S(0,R)\cap B)\cup \partial B))$
when $z\in \varphi(\alpha)=\gamma.$
The desired conclusion also follows if the roles of $z_1,z_2$ above are reversed. 

Next, \eqref{quasi} (applied to $\varphi^{-1}$) gives us an 
absolute constant $C$ such that if 
$$C|z_1-z_2|\le \max\{\dist(z_1,\partial \Omega),\,\dist(z_2,\partial \Omega)\},$$
then we are in a situation covered by the previous paragraph. Thus we may  assume that
\begin{equation}\label{oikeasti}
C|z_1-z_2|\ge \max\{\dist(z_1,\partial \Omega),\,\dist(z_2,\partial \Omega)\}.
\end{equation}

Recall from  Lemma~\ref{quasisymmetry} that $\varphi$ is $\eta$-quasisymmetric 
with respect to the inner distance with $\eta$ that only depends on the John constant of 
$\Omega.$
Define
$$U= B\left(\frac {1  + |\varphi^{-1}(z_1)|} 2 \frac{\varphi^{-1}(z_1)}{|\varphi^{-1}(z_1)|},\, \frac {1 - |\varphi^{-1}(z_1)|} 2 \right).$$
Then the disk $U$ is contained
in $\mathbb D\setminus \overline {B_n}$, 
$z_1\in \varphi(\overline U),$ $\varphi(\overline U)\cap \partial
\Omega\neq \emptyset,$ and Lemma~\ref{John subdomain} gives that
$\varphi(U)$ is $J'$-John with center $\varphi(w),$ where $w$ is the center of 
$U,$ and $J'$ only depends on the John constant $J$ of $\Omega.$ 

We claim that 
\begin{equation}\label{halkaisija}
\diam(\varphi(U))\le C(J)\dist(z_1,\partial \Omega).
\end{equation}
Towards this, let $\xi=\varphi^{-1}(z_1)/|\varphi^{-1}(z_1)|,$ 
the tangent point of $U$ with the unit circle, and pick a point $z_3\in \partial \Omega$ 
satisfying 
$$\dist(z_1,\,\partial \Omega)=|z_1-z_3|.$$ 
Pick a sequence of points $x_j$ along the Euclidean segment between $z_1,z_3$ so
that $x_j\to z_3.$ 
Then
\begin{equation}\label{tovika}
\dist_{\Omega}(z_1,x_j)=|z_1-x_j|\le \dist_{\Omega}(z_1,x_j).
\end{equation} 
Since $\varphi$ is a homeomorphism of the unit disk onto $\Omega,$
there is a subsequence of the sequence $(x_j)$ so that the preimages converge to some
$w_3\in \partial \mathbb D.$ For simplicity, we refer to the elements of this subsequence
still by $x_j.$ 
Write $r$ for the radius of $U.$ By the definition of $U$ and $\xi$ we have 
$$2r=|\varphi^{-1}(z_1)-\xi|=\dist(\varphi^{-1}(z_1),\partial \mathbb D)$$
and since $w_3\in \partial \mathbb D,$ we conclude that
$$2r\le |\varphi^{-1}(z_1)-w_3|.$$
In particular, for any $w_2\in U$ we have
$$|\varphi^{-1}(z_1)-w_2|\le 2r\le |\varphi^{-1}(z_1)-w_3|$$
and consequently
\begin{equation}\label{totovika}
|\varphi^{-1}(z_1)-w_2|\le 2|\varphi^{-1}(z_1)-\varphi^{-1}(x_j)|
\end{equation}
for all sufficiently large $j.$
 Quasisymmetry of $\varphi$ together with \eqref{totovika} and \eqref{tovika} now gives for all sufficiently large $j$ 
 the estimate
 $$|z_1-\varphi(w_2)|\le \dist_{\Omega}(z_1,\varphi(w_2))\le \eta(2) \dist_{\Omega}(z_1,x_j)\le \eta(2)\dist(z_1,\partial \Omega).$$ 
Hence \eqref{halkaisija} follows.

By connecting $z_1$ to the John center $\varphi(w)$ of $\varphi(U)$ and then 
the John center to $\varphi(\xi)\in \partial \Omega$ via hyperbolic segments in $\varphi(U)$, we obtain by Remark~\ref{muistutus},  \eqref{halkaisija} and \eqref{oikeasti}
a curve $\gamma_1\subset
\varphi(U)$ consisting of two John curves and joining $z_1$ to 
$\partial \Omega$ so that
\begin{align*}
&\int_{\gamma_1}\dist(z,\,\partial \Omega_n)^{1-\hat p} \, \d s(z) \le \int_{\gamma_1}\dist(z,\,\partial (\varphi(U)))^{1-\hat p} \, \d s(z)\\ &\lesssim \dist(\varphi(w),\partial (\varphi(U))^{2-s}\lesssim \diam(\varphi(U))^{2-\hat p}\lesssim \dist(z_1,\partial \Omega)^{2-\hat p} \lesssim |z_2-z_1|^{2-\hat p}. 
\end{align*}
See Figure \ref{fig:curvext}. Here the constants depend only on $J.$
 Analogously, we find a corresponding curve $\gamma_2$ for $z_2.$
It remains to join the two endpoints $\wz z_1,\wz z_2$
of $\gamma_1$ and $\gamma_2$ in $\partial \Omega$ by a curve $\gamma_3$ outside $\Omega$ as in the first paragraph of the proof; notice here that \eqref{oikeasti} guarantees that
$$|\wz z_1-\wz z_2|\le C|z_1-z_2|.$$
 By the triangle inequality, the curve obtained by concatenating  $\gamma_1,\gamma_3$ and $\gamma_2$
satisfies our requirements. 

We are left to consider the case where $z_1\in \Omega\setminus \overline \Omega_n$ and  $z_2\notin \Omega.$ In this case, we simply use the curve $\gamma_1$ constructed above for $z_1$ together with a curve
$\gamma_3$ outside $\Omega$ joining $z_2$ and the endpoint of  $\gamma_1$ in $\partial \Omega$ as above. 
\end{proof}

\begin{proof}[Proof of Theorem~\ref{thm:main}]
By Section~\ref{sec:nec}, we only need to prove the sufficiency of \eqref{eq:extcharcompl}.
Recall the conformal map $\varphi$ and the domains 
$$\Omega_n=\varphi(B_n)$$ from the beginning of this subsection. 
By Lemma~\ref{John subdomain}, the domains
$\Omega_n$ are John domains with John center $x_0=\varphi(0)$ with a John constant
only depending on $J.$

By Lemma \ref{reduktio} and Theorem \ref{prop:Jordancase}, \eqref{eq:extcharcompl} yields 
that 
there exist extension operators 
$$E_n\colon W^{1,\,p}(\Omega_n)\to W^{1,\,p}(B(x_0,\diam(\Omega_n)),$$ where 
the norms of the extension operators $E_n$ are independent of $n,$ see Remark~\ref{lopulta}. 
Since
$\Omega_n=\varphi(B_n)$ and $\varphi$ is continuous up to boundary, $\diam(\Omega_n)\to
\diam(\Omega)$ when $n$ tends to infinity. Hence $B(x_0,r)\subset B(x_0,\diam(\Omega_n))$
for all sufficiently large $n$ when $r= \diam(\Omega)-\dist(x_0,\partial \Omega).$
Define $B=B(x_0,r).$ 
We conclude that
$$E_n\colon W^{1,\,p}(\Omega_n)\to W^{1,\,p}(B)$$
is a bounded extension operator with a norm bound independent of $n,$ 
for all sufficiently large
$n.$



Fix $u\in W^{1,\,p}(\Omega)$, and let $u_n=u|_{\Omega_n}$ for $n\ge 2$. 
Now
$\|\nabla E_n u_n\|_{L^{p}(B)} + \|E_n u_n\|_{L^p(B)}$ is bounded independently of $n$ for large
$n.$
Hence, by the assumption
$p>1,$ there exists a subsequence that converges weakly in $L^p(B)$ to some 
$v\in W^{1,\,p}
(B)$ with 
$$ \|\nabla v\|_{L^{p}(B)} + \|v\|_{L^p(B)}\le \liminf_{n\to \infty}\left(\|\nabla E_n u_n\|_{L^{p}(B)} + \|E_n u_n\|_{L^p(B)}\right).$$
Define $Eu:=v$ and notice that $\Omega\subset B$ and that the sequence 
$\{E_n u_n\}$ converges to $u$ pointwise a.e. on 
$\Omega$. Hence we know 
that $Eu$ is an extension of $u,$ and the desired norm bound over $B$ follows 
from the uniform bound
on the extension operators $E_n.$ 
Since $B$ is a $W^{1,\,p}$-extension domain, this completes the proof of 
Theorem \ref{thm:main}.
\end{proof}


\section{Proof of Corollary~\ref{cor:dual}}\label{sec:cor}

Before giving the proof of Corollary~\ref{cor:dual}, we present  a lemma stating that we can always 
swap 
an unbounded domain with compact boundary to a bounded domain (and vice versa) with the same 
extendability and curve properties.
This is the main observation needed to conclude Corollary \ref{cor:dual} from Theorem 
\ref{thm:main}
and Theorem \ref{thm_sh}. 

\begin{lemma}\label{lma:unboundedbounded}
 Let $\Omega \subset \R^2$ be a bounded domain. Fix $x \in \Omega$ and define
 an unbounded domain $\hat \Omega = i_x(\Omega \setminus\{x\})$ using the inversion 
 \[
  i_x \colon \R^2\setminus\{x\} \to \R^2\setminus\{x\} \colon y \mapsto x + \frac{y-x}{|y-x|^2}.
 \]
 Then
 \begin{enumerate}
  \item For any $1 \le p \le \infty$ the domain $\Omega$ is a $W^{1,p}$-extension domain 
        if and only if $\hat \Omega$ is a $W^{1,p}$-extension domain.
  \item For any $q>2$ there exists a constant $C>0$ such that 
  for all $z_1, z_2 \in \Omega$ there exists 
a rectifiable curve $\gamma \subset \Omega$ joining $z_1$ and $z_2$ such that
\begin{equation}
  \int_{\gamma}\dist(z,\partial \Omega)^\frac{1}{1-q}\,\d s(z) 
\le C |z_1-z_2|^\frac{q-2}{q-1}\label{eq:extcharagain}
\end{equation} if and only if exists a constant $\hat C>0$ such that for every $\hat z_1, \hat z_2 \in \hat \Omega$
\eqref{eq:extchar} there exists 
a rectifiable curve $\hat\gamma \subset \hat \Omega$ joining $\hat z_1$ and $\hat z_2$ such that
\begin{equation}\label{eq:extcharagainandagain}
  \int_{\gamma}\dist(z,\partial \hat \Omega)^\frac{1}{1-q}\,\d s(z) 
\le \hat C |\hat z_1- \hat z_2|^\frac{q-2}{q-1}.
\end{equation}
\end{enumerate}
\end{lemma}

\begin{proof}
Let  $R=2\diam(\Omega)$ and $2r= \dist(x,\,\partial \Omega) $. Then $\partial \Omega \subset A(x,\,r,\,R) := B(x,\,R) \setminus \overline{B(x,\,r)}$. 
Notice that $i_x$ is a bi-Lipschitz map 
when restricted to $A(x,\,r,\,R)$, with the bi-Lipschitz constant only depending on $r$ and $R$, and that $i_x(A(x,\,r,\,R)) = A(x,1/R,1/r)$. 

(1) Notice that for $0 < r_1 < r_2 < \infty$, also the annulus $A(x,r_1,r_2)$ is an 
$W^{1,p}$-extension domain with an operator $E_{r_1,r_2}$.
Now, assume that $\Omega$ is a $W^{1,p}$-extension domain with extension operator $E$. Let us show that $\hat \Omega$ is also a $W^{1,p}$-extension domain. Towards this, take $u\in W^{1,\,p}(\hat \Omega)$. By the fact that $i_x$ is bi-Lipschitz on $A(x,r,R)$, we have $ u\circ i_x|_{A(x,\,r,\,R)\cap \Omega}\in W^{1,\,p}(\Omega \setminus B(x,r) )$.
Since $A(x,\,r,\,2r)\subset A(x,\,r,\,R)\cap \Omega$,
we have $E_{r,2r}(u\circ i_x|_{A(x,\,r,\,2r)})\in W^{1,\,p}(\mathbb R^2)$. Now, define $v \in W^{1,p}(\Omega)$ by
\[
 v(y) = \begin{cases}
         E_{r,2r}(u\circ i_x)(y), & \text{if }
         y \in B(x,2r)\\
         u\circ i_x (y) & \text{if }y \in \Omega \setminus B(x,2r).
        \end{cases}
\]
This can then be extended to $Ev\in W^{1,\,p}(\mathbb R^2)$.
 Again, by the bi-Lipschitz property of $i_x$ on $A(x,r,R)$, we have that
 \[
  (Ev) \circ i_x^{-1}\in W^{1,\,p}(A(x,1/R,1/r)),
 \]
 which finally gives the required extension $\hat Eu \in W^{1,p}(\mathbb R^2)$ as
 \[
  \hat Eu(y) = \begin{cases}
                 E_{1/R,1/r}((Ev)\circ i_x^{-1})(y)
                , &\text{if }y \in B(x,1/r)\\
                u(y), &\text{if }y \notin B(x,1/r).
               \end{cases}
 \]
 This shows that $\hat \Omega$ is a $W^{1,p}$-extension domain.
 
 Let us then show the converse and assume that $\hat \Omega$ is a $W^{1,p}$-extension domain with an extension operator $\hat E$. The construction of the extension is done analogously to the previous case.
 Let $u \in W^{1,p}(\Omega)$.
 Then $v \in W^{1,p}(\hat\Omega)$, when we define
\[
 v(y) = \begin{cases}
         E_{1/(2r),1/r}(u\circ i_x^{-1})(y), & \text{if }
         y \notin B(x,1/r)\\
         u\circ i_x^{-1} (y) & \text{if }y \in \Omega \cap B(x,1/r),
        \end{cases}
\] 
and the required extension $Eu \in W^{1,p}(\mathbb R^2)$ is then given by
\[
  Eu(y) = \begin{cases}
                 E_{r,R}((\hat Ev) \circ i_x)(y)
                , &\text{if }y \notin B(x,r)\\
                u(y), &\text{if }y \in B(x,r).
               \end{cases}
 \]
 Hence, $\Omega$ is a $W^{1,p}$-extension domain.
 \medskip
 
%
%
%
%

(2) Suppose the existence of curves $\gamma \subset \Omega$ satisfying \eqref{eq:extcharagain}. Let us show the condition \eqref{eq:extcharagainandagain} for $\hat \Omega$. Towards this, take $\hat z_1,\hat z_2 \in \hat \Omega$. Suppose
first that $\hat z_1,\hat z_2 \in \overline{B}(x,1/r)$. Define $z_1 = i_x^{-1}(\hat z_1)$ and $z_2 = i_x^{-1}(\hat z_2)$. Let $\gamma \subset \Omega$ be a curve joining $z_1$ and $z_2$ so that \eqref{eq:extcharagain} holds.
By \cite[Lemma 2.1]{sh2010} we may assume that $\elle(\gamma)\le C|z_1-z_2|$.
If $\gamma \subset A(x,r,R)$, then by the bi-Lipschitz property of $i_x$, the curve $\hat \gamma = i_x(\gamma)$ satisfies \eqref{eq:extcharagainandagain}. If $\gamma \not\subset A(x,r,R)$, let $\tilde z_1,\tilde z_2 \subset \gamma \cap \partial B(x,r)$ be such that $\gamma[z_1,\tilde z_1]$ connects $z_1$ to $\tilde z_1$ in $A(x,r,R)$, and 
$\gamma[z_2,\tilde z_2]$ connects $z_2$ to $\tilde z_2$ in $A(x,r,R)$. Then,
\begin{equation}\label{eq:quasiconvexforcircle}
|\tilde z_1 - \tilde z_2| \le \elle(\gamma) \le C|z_1-z_2|.
\end{equation}
Let $\alpha$ be a shorter arc of $\partial B(x,r)$ joining $\tilde z_1$ and $\tilde z_2$. Since $\dist(\partial B(x,r), \partial\Omega)\ge r$, by \eqref{eq:quasiconvexforcircle} we have
\[
  \int_{\alpha}\dist(z,\partial \Omega)^\frac{1}{1-q}\,\d s(z) 
\le C |z_1-z_2|^\frac{q-2}{q-1}.
\]
Hence, again by the bi-Lipschitz property of $i_x$ on $A(x,r,R)$ the curve $\hat \gamma = i_x(\gamma[z_1,\tilde z_1] \ast \alpha \ast \gamma[\tilde z_2,z_2])$ satisfies \eqref{eq:extcharagainandagain}.

Suppose then that $\hat z_1,\hat z_2 \in \overline{B}(x,1/r)$ fails. Then, if $[\hat z_1, \hat z_2] \cap \partial B(x,1/r)$ contains two distinct points, we can use the previous case to connect these by a curve $\gamma \subset \hat \Omega$. For the remaining part, we can simply use the remaining parts of $[\hat z_1, \hat z_2] \setminus B(x,1/r)$. Finally, if $[\hat z_1, \hat z_2] \cap \partial B(x,1/r)$ is a singleton, we simply use $[\hat z_1,\hat z_2]$.

The proof of the converse implication is analogous. Towards it, let us assume that there exist curves $\hat \gamma \subset \hat \Omega$ satisfying \eqref{eq:extcharagainandagain}. Let
$z_1,z_2 \in \Omega$ and define $\hat z_1 = i_x(z_1)$ and $\hat z_2 = i_x(z_2)$. Let $\hat \gamma \subset \hat \Omega$ be a curve connecting $\hat z_1$ and $\hat z_2$ that satisfies \eqref{eq:extcharagainandagain}. By Lemma~\ref{kvasikonveksi} we may assume that $\elle(\hat\gamma)\le C|\hat z_1 -\hat z_2|$.
If $\hat \gamma \subset \overline{B}(x,1/r)$, again the bi-Lipschitz property of $i_x$ inside $A(x,r,R)$ gives that $\gamma = i_x^{-1}(\hat\gamma)$ satisfies \eqref{eq:extcharagain}. Let us then suppose that $\hat \gamma \not\subset \overline{B}(x,1/r)$.
If $\hat z_1 \in \overline{B}(x,1/r)$, we take $\tilde z_1\in \hat\gamma \cap \partial B(x,1/r)$ so that $\hat \gamma[\hat z_1,\tilde z_1] \subset \overline{B}(x,1/r)$.
If $\hat z_1 \notin \overline{B}(x,1/r)$, we define $\tilde z_1 = \hat z_1$.
Similarly, if $\hat z_2 \in \overline{B}(x,1/r)$, we take $\tilde z_2\in \hat\gamma \cap \partial B(x,1/r)$ so that $\hat \gamma[\hat z_2,\tilde z_2] \subset \overline{B}(x,1/r)$,
and if $\hat z_2 \notin \overline{B}(x,1/r)$, we set $\tilde z_2 = \hat z_2$. Then, the curve
\[
 \gamma = i_x^{-1}(\hat \gamma[\hat z_1,\tilde z_1]) \ast [i_x^{-1}(\tilde z_1),i_x^{-1}(\tilde z_2)] \ast i_x^{-1}(\hat \gamma[\tilde z_2,\hat z_2])
\]
connects $z_1$ to $z_2$ in $\Omega$ and satisfies  \eqref{eq:extcharagain} because of 
\[
 |\tilde z_1 -\tilde z_2| \le \elle(\hat\gamma)\le C|\hat z_1 -\hat z_2|
\]
and $\dist(x,\gamma)\ge r$.
\end{proof}

\begin{proof}[Proof of Corollary~\ref{cor:dual}]
By Lemma~\ref{lma:unboundedbounded} it suffices to show that the complementary domain
$\wz \Omega$ of a given Jordan $W^{1,p}$-extension domain $\Omega,$ where $1<p<\infty,$
is a $W^{1,q}$-extension domain for $q=p/(p-1).$

Suppose first that our Jordan domain $\Omega$ is a $W^{1,p}$-extension domain for a given 
$1<p<2.$
Then Theorem \ref{neceJordan} and Remark~\ref{remark 1} give the existence of curves as in
\eqref{eq:extcharcompl} in the complementary domain $\wz \Omega.$ Notice that 
\eqref{eq:extcharcompl} is precisely \eqref{eq:extchar} with $q=p/(p-1)>2.$ Thus, by 
applying Lemma~\ref{lma:unboundedbounded} (twice) and Theorem~\ref{thm_sh} we conclude
that $\wz \Omega$ is a $W^{1,q}$-extension domain. 

If $\Omega$ is a $W^{1,p}$-extension domain for some 
$p>2,$ then  \eqref{eq:extchar} holds by Theorem~\ref{thm_sh} (for points in $\Omega$).
Let $x \in \Omega$, and take $i_x$ and $\hat \Omega$ as in Lemma~\ref{lma:unboundedbounded}.
Now, by applying again Lemma~\ref{lma:unboundedbounded}, \eqref{eq:extcharcompl} holds for points in $\hat \Omega$. For any pair $z_1, z_2 \in \overline{\hat \Omega}$ there exist sequences $(x_j)_j,(y_j)_j$ in $\hat \Omega$ that converge to $z_1$ and $z_2$, respectively. For every pair $(x_j,y_j)$ we take a curve $\gamma_j \subset \hat \Omega$ satisfying \eqref{eq:extcharcompl} (with the obvious notational changes). By Lemma~\ref{arsela}, there exists a limiting curve $\gamma \subset \overline{\hat \Omega}$ connecting $z_1$ and $z_2$ also satisfying \eqref{eq:extcharcompl}. Hence, by Theorem \ref{thm:main}, $\mathbb R^2 \setminus \overline{\hat \Omega}$ is a $W^{1,q}$-extension domain, and so, via Lemma \ref{lma:unboundedbounded}, also $\tilde \Omega$.


We are left with the case $p=2.$ In this case, $\Omega$ is necessarily a uniform domain and
hence so is $\wz \Omega$. Thus $\wz \Omega$ is also a $W^{1,2}$-extension domain; see 
\cite{golavo1979,gore1990,govo1981,jo1981}.
\end{proof}


\begin{thebibliography}{10}


\bibitem{AIM2009}
K. Astala, T. Iwaniec, G. Martin,  \emph{Elliptic partial differential equations and quasiconformal mappings in the plane}. Princeton Mathematical Series, 48. Princeton University Press, Princeton, NJ, 2009. 

\bibitem{BV1996}
Z. Balogh, A. Volberg, \emph{Geometric localization, uniformly John property and separated semihyperbolic dynamics}. Ark. Mat. \textbf{34} (1996), no. 1, 21--49. 

\bibitem{B1997}
B. Bojarski,
\emph{Remarks on Sobolev imbedding inequalities} in  Complex analysis, 
Joensuu 1987,  52--68,
Lecture Notes in Math., 1351, Springer, Berlin, 1988. 

\bibitem{BKR1998}
M. Bonk, P. Koskela, S. Rohde, \emph{Conformal metrics on the unit ball in Euclidean space}. Proc. London Math. Soc. (3) \textbf{77} (1998), no. 3, 635--664. 



\bibitem{Br06}
O. J. Broch, \emph{Extension of internally bilipschitz maps in John disks}
Ann. Acad. Sci. Fenn. Math. \textbf{31} (2006), no. 1, 13--30.

\bibitem{buko1996} S. Buckley and P. Koskela, 
\emph{Criteria for imbeddings of Sobolev-Poincar\'e type},
 Int. Math. Res. Not. \textbf{18} (1996), 881--902.
 
 
\bibitem{cal1961} A. P. Calder\'on, \emph{Lebesgue spaces of differentiable 
functions and distributions,} in Proc. Symp. Pure Math., Vol. IV, 1961, 33--49. 



\bibitem{deheu2014} T. Deheuvels, \emph{Sobolev extension property for tree-shaped domains with
self-contacting fractal boundary}. Ann. Sc. Norm. Super. Pisa Cl. Sci. (5) \textbf{15} 
(2016), 209--247.  


\bibitem{Falconer198} K. J. Falconer, \emph{The geometry of fractal sets. Cambridge Tracts in Mathematics},
vol. 85. Cambridge University Press, Cambridge, 1986.

\bibitem{G1962} 
F. W. Gehring, \emph{Rings and quasiconformal mappings in space}. 
Trans. Amer. Math. Soc. \textbf{103} 1962, 353--393.

 \bibitem{GHM1989} 
F. W. Gehring, K.  Hag, O.  Martio, 
\emph{Quasihyperbolic geodesics in John domains}. 
Math. Scand.  \textbf{65} (1989), no. 1, 75--92. 

 \bibitem{GH1962} 
F. W. Gehring, W. K. Hayman, \emph{An inequality in the theory of conformal mapping}. 
J. Math. Pures Appl. (9) \textbf{41} (1962), 353--361. 




\bibitem{GM1985}  F. W. Gehring, O. Martio, 
\emph{Lipschitz classes and quasiconformal mappings}.
Ann. Acad. Sci. Fenn. Ser. A I Math.  \textbf{10}  (1985), 203--219.

\bibitem{GM1985 2} 
 F. W. Gehring, O. Martio, \emph{Quasiextremal distance domains and extension of quasiconformal mappings}. 
J. Analyse Math. \textbf{45} (1985), 181--206. 

\bibitem{GNV94} M. Ghamsari, R. N\"akki and J. V\"ais\"al\"a,
\emph{John disks and extension of maps}. Mh. Math. \textbf{117} (1994),
63--94.



\bibitem{golavo1979} V. M. Gol'dshte{\u\i}n, T. G. Latfullin and S. K. Vodop'yanov,
 \emph{A criterion for the extension of functions of the class $L_2^1$ from unbounded plain domains 
(Russian)},
 Sibirsk. Mat. Zh. \textbf{20} (1979), 416--419.

\bibitem{gore1990} V. M. Gol'dshte{\u\i}n and Yu. G. Reshetnyak,
 \emph{Quasiconformal mappings and {S}obolev spaces}, 
 Mathematics and its Applications (Soviet Series),
 \textbf{54} (1990), Kluwer Academic Publishers Group, Dordrecht.

\bibitem{govo1981} V. M. Gol'dshte{\u\i}n and S. K. Vodop'yanov,
 \emph{Prolongement de fonctions diff\'erentiables hors de domaines planes (French)},
 C. R. Acad. Sci. Paris Ser. I Math. \textbf{293} (1981), 581--584.

\bibitem{hakotu2008} P. Haj{\l}asz, P. Koskela and H. Tuominen,
 \emph{Sobolev embeddings, extensions and measure density condition},
 J. Funct. Anal. \textbf{254} (2008), no. 5, 1217--1234. 

\bibitem{hakotu2} P. Haj{\l}asz, P. Koskela and H. Tuominen,
 \emph{Measure density and extendability of Sobolev functions},
 Rev. Mat. Iberoam.  \textbf{24} (2008), no. 2, 645--669. 

 

\bibitem{H1989}
J. Heinonen, \emph{Quasiconformal mappings onto John domains}. Rev. Mat. Iberoamericana \textbf{5} (1989), 3-4, 97--123. 

\bibitem{HKM1993}
J. Heinonen, T.  Kilpel\"ainen,  O. Martio,\emph{ Nonlinear potential theory of degenerate elliptic equations}. Oxford Mathematical Monographs. Oxford Science Publications. The Clarendon Press, Oxford University Press, New York, 1993.

\bibitem{H2012}
D. Herron, 
  \emph{Riemann maps and diameter distance. }
Amer. Math. Monthly \textbf{119} (2012), no. 2, 140--147.
 
\bibitem{heko90} D. Herron and P. Koskela, \emph{Quasiextremal distance domains and conformal mappings onto circle domains},
  Complex Variables Theory Appl. \textbf{15} (1990), 167--179.

 
\bibitem{heko91} D. Herron and P. Koskela,
  \emph{Uniform, Sobolev extension and quasiconformal circle domains},
  J. Anal. Math. \textbf{57} (1991), 172--202.



\bibitem{hlpw} M. Huang, Y. Li, S. Ponnusamy and X. Wang \emph{The quasiconformal subinvariance property of John domains in $\mathbb R^n$ and its applications},
 Math. Ann. \textbf{363} (2015), 549--615. 

\bibitem{jo1981} P. W. Jones, \emph{Quasiconformal mappings and extendability of Sobolev functions},
 Acta Math. \textbf{47} (1981), 71--88.

\bibitem{JS2000} Peter W. Jones, Stanislav K. Smirnov, \emph{Removability theorems for Sobolev functions and quasiconformal maps}, Ark. Mat. \textbf{38} (2000), no. 2, 263--279. 

\bibitem{kos1990} P. Koskela, \emph{Capacity extension domains}, 
Ann. Acad. Sci. Fenn. Ser. A I Math. Dissertationes No. 73 (1990), 42 pp.

\bibitem{ko1998} P. Koskela, \emph{Extensions and imbeddings},
 J. Funct. Anal. \textbf{159} (1998), 1--15.


\bibitem{kosmirsha10} P. Koskela, M. Miranda Jr. and N. Shanmugalingam, 
 \emph{Geometric properties of planar BV-extension domains}, 
 International Mathematical Series (N.Y.) \textbf{11} (2010), no. 1, 255--272.

\bibitem{koyazh2010} P. Koskela, D. Yang and Y. Zhou,
 \emph{A Jordan Sobolev extension domain}, Ann. Acad. Sci. Fenn. Math. \textbf{35} (2010), 309--320.

\bibitem{KZ2016} 
P.  Koskela, Y. R.-Y.  Zhang, \emph{A Density Problem for Sobolev Spaces on Planar Domains}, Arch. Rational Mech. Anal. \text{222} (2016)  no. 1, 1--14.




\bibitem{la1985} V. Lappalainen, \emph{$Lip_h$-extension domains}, Ann. Acad. Sci. Fenn. Ser. AI Math
Diss. \textbf{56} (1985), 1--52.

 
\bibitem{lewis1987} J. Lewis,
 \emph{Approximation of {S}obolev functions in {J}ordan domains},
 Ark. Mat. \textbf{25} (1987), 255--264.

\bibitem{ma1981} V. G. Maz'ya, \emph{On the extension of functions belonging to S. L. Sobolev spaces},
 Zap. Nauchm. Sem. Leningrad Otdel. Mat. Inst. Steklov. (LOMI) \textbf{113} (1981), 231--236.

\bibitem{nava1991} R. N\"akki and J. V\"ais\"al\"a, \emph{John disks},
 Exposition. Math. \textbf{9} (1991), 3--43.
 
\bibitem{P1991} B. Palka, An introduction to complex function theory,
  Undergraduate Texts in Mathematics, New York: Springer-Verlag, 1991.

\bibitem{Ch1992}
Ch. Pommerenke, \emph{Boundary behaviour of conformal maps},  
Grundlehren der Mathematischen Wissenschaften [Fundamental Principles of Mathematical Sciences], 299. Springer-Verlag, Berlin, 1992.

\bibitem{R1993}
S. Rickman, \emph{Quasiregular mappings}, Ergebnisse der Mathematik und ihrer 
Grenzgebiete (3) [Results in Mathematics and Related Areas (3)], 26. 
Springer-Verlag, Berlin, 1993. 

 

\bibitem{ro1993} A. S. Romanov, \emph{On the extension of functions that belong to Sobolev spaces (Russian)},
 Sibirsk. Mat. Zh. \textbf{34} (1993), 149--152. (English transl. in Siberian Math. J. \textbf{34} (1993), 723--726.)

\bibitem{sh2006} P. Shvartsman, \emph{Local approximations and intrinsic characterization of spaces of smooth 
 functions on regular subsets of $\R^n$}, Math. Nachr. \textbf{279} (2006), no. 11, 1212--1241.

\bibitem{sh2010} P. Shvartsman, \emph{On Sobolev extension domains in $\R^n$}, J. Funct. Anal. \textbf{258} (2010), no. 7, 2205--2245.

\bibitem{stein1970} E. M. Stein,  \emph{Singular integrals and differentiability
properties of functions}.
Princeton University Press, Princeton, New Jersey, 1970.

\bibitem{V1971}
J. V\"ais\"al\"a,  \emph{Lectures on n-dimensional quasiconformal mappings}. Lecture Notes in Mathematics, Vol. 229. Springer-Verlag, 
Berlin-New York, 1971. 

\bibitem{V1998}
J. V\"ais\"al\"a, 
 \emph{Relatively and inner uniform domains.}
Conform. Geom. Dyn. \textbf{2} (1998), 56--88 . 


\bibitem{V1988}
M. Vuorinen, \emph{Conformal geometry and quasiregular mappings}. 
Lecture Notes in Mathematics, 1319. Springer-Verlag, Berlin, 1988. 

\end{thebibliography}
\end{document}